\documentclass[11pt]{article}
\usepackage{float}
\usepackage[margin=1in]{geometry}
\usepackage{amsmath,amsfonts,amssymb,amsthm}
\usepackage[T1]{fontenc}
\usepackage{graphicx}
\usepackage{enumerate}
\usepackage{bbm}
\usepackage{verbatim}
\usepackage{hyperref,color}
\usepackage[capitalize,nameinlink]{cleveref}
\usepackage[dvipsnames]{xcolor}
\hypersetup{
	colorlinks=true,
	pdfpagemode=UseNone,
	citecolor=OliveGreen,
	linkcolor=NavyBlue,
	urlcolor=Magenta,
	pdfstartview=FitW
}
\usepackage{appendix}
\usepackage{makecell}
\usepackage{tablefootnote}
\crefname{appsec}{Appendix}{Appendices}
\usepackage{tikz}
\usepackage{pgfplots}
\usepackage{xifthen}
\usepackage{subcaption}
\usepackage{hhline}

\theoremstyle{plain}
\newtheorem{theorem}{Theorem}[section]
\newtheorem{proposition}[theorem]{Proposition}
\newtheorem{lemma}[theorem]{Lemma}
\newtheorem{corollary}[theorem]{Corollary}

\newtheorem{fact}[theorem]{Fact}

\theoremstyle{definition}
\newtheorem{definition}[theorem]{Definition}
\newtheorem{example}[theorem]{Example}

\newtheorem*{assumption*}{Assumption}

\theoremstyle{remark}
\newtheorem{remark}[theorem]{Remark}

\crefname{lemma}{Lemma}{Lemmas}
\crefname{theorem}{Theorem}{Theorems}
\crefname{definition}{Definition}{Definitions}
\crefname{fact}{Fact}{Facts}
\crefname{claim}{Claim}{Claims}
\crefname{proposition}{Proposition}{Propositions}
\crefname{equation}{Eq.}{Eqs.}

\usetikzlibrary{arrows.meta,positioning}

\newcommand{\dif}{\,\mathrm{d}}

\newcommand{\ind}{\*1}

\newcommand{\norm}[1]{\left\lVert #1 \right\rVert}

\newcommand{\grad}{\nabla}

\newcommand{\diag}{\mathrm{diag}}

\newcommand{\sgn}{\mathrm{sgn}}

\newcommand{\eps}{\varepsilon}

\newcommand{\Z}{\mathbb{Z}}

\newcommand{\R}{\mathbb{R}}

\newcommand{\TV}[2]{d_{\mathrm{TV}}\left({#1},\,{#2}\right)}
\newcommand{\e}{\mathrm{e}}
\renewcommand{\epsilon}{\varepsilon}

\newcommand{\set}[1]{\left\{#1\right\}}
\newcommand{\tuple}[1]{\left(#1\right)} 
\newcommand{\inner}[2]{\left\langle #1,#2\right\rangle}
\newcommand{\tp}{\tuple}

\newcommand{\abs}[1]{\left\vert#1\right\vert}

\def\*#1{\boldsymbol{#1}} 
\def\+#1{\mathcal{#1}} 
\def\-#1{\mathrm{#1}} 
\def\=#1{\mathbb{#1}} 
\def\!#1{\mathfrak{#1}} 

\def\oPr{\mathrm{Pr}}
\renewcommand{\Pr}[2][]{ \ifthenelse{\isempty{#1}}
  {\oPr\left[#2\right]}
  {\oPr_{#1}\left[#2\right]} } 

\def\oE{\mathbb{E}}
\newcommand{\E}[2][]{ \ifthenelse{\isempty{#1}}
  {\oE\left[#2\right]}
  {\oE_{#1}\left[#2\right]} }

\def\oVar{\mathrm{Var}}
\newcommand{\Var}[2][]{ \ifthenelse{\isempty{#1}}
  {\oVar\left(#2\right)}
  {\oVar_{#1}\left(#2\right)} }

\def\oEnt{\mathrm{Ent}}
\newcommand{\Ent}[2][]{ \ifthenelse{\isempty{#1}}
  {\oEnt\left[#2\right]}
  {\oEnt_{#1}\left[#2\right]} }

\newcommand{\PhiEnt}[2][]{ \ifthenelse{\isempty{#1}}
  {\oEnt^\phi\left[#2\right]}
  {\oEnt^\phi_{#1}\left[#2\right]} }

\newcommand{\mathsc}[1]{{\normalfont\textsc{#1}}}

\newcommand{\Sch}{\mathrm{Sch}}
\newcommand{\LSI}{\mathsc{lsi}}
\newcommand{\MLSI}{\mathsc{mlsi}}
\newcommand{\BE}{\mathsc{be}}
\newcommand{\RBE}{{\mathrm{r}\mathsc{be}}}
\newcommand{\WGE}{\mathsc{wge}}
\newcommand{\SGE}{\mathsc{sge}}
\newcommand{\CSGE}{\mathsc{csge}}

\newcommand{\Inf}{\mathsf{Inf}}
\newcommand{\SurfArea}{\mathsf{SA}}
\newcommand{\SqInf}{\mathsf{SqInf}}
\newcommand{\cube}{\{\pm 1\}^n}

\title{Discrete Isoperimetric Inequalities via Curvature}

\author{
Zejia Chen\thanks{
    School of Computer Science, Georgia Institute of Technology, Atlanta, GA, USA.\\
    Email: \texttt{\{zchen3091,chenzongchen\}@gatech.edu}, \texttt{normalvra@gmail.com}
}
\and
Zongchen Chen\footnotemark[1]
\and
Xinyuan Zhang\footnotemark[1]
}
\date{\today}

\pgfplotsset{compat=1.18} 
\begin{document}

\maketitle

\begin{abstract}
Isoperimetric inequalities on the Boolean hypercube play a fundamental role in the analysis of Boolean functions. These inequalities have been established primarily for the uniform measure and for biased product measures, often through Fourier-analytic or inductive arguments. We develop a curvature-based semigroup framework to establish isoperimetric inequalities for measures with weakly dependent coordinates. In particular, we show that a Dobrushin-type condition, together with marginal boundedness of the distribution, implies the local Bobkov inequality, Talagrand's $L^1$--$L^2$ and variance--surface-area inequalities, the Kahn--Kalai--Linial inequality, and the Eldan--Gross inequality. Our results apply to zero-field Ising models with interaction matrix $J$ throughout the Dobrushin uniqueness regime $\norm{J}_1<1$. The framework builds on discrete Bakry--\'Emery theory and gradient estimates and can also be extended to measures on Hamming slices or hypergrids.
\bigskip

\end{abstract}

\thispagestyle{empty}

\newpage 
\thispagestyle{empty}

\begingroup
\small
\tableofcontents
\endgroup

\newpage

\setcounter{page}{1}

\newpage
\section{Introduction}

Boolean functions arise naturally in probability theory, combinatorics, and theoretical computer science. A Boolean function \(f:\{\pm 1\}^n \to \{\pm 1\}\) can, for example, encode a voting rule or a graph property. Isoperimetric inequalities are fundamental tools in the analysis of Boolean functions: they relate global quantities, such as variance, to boundary measures that quantify sensitivity to changes in individual input coordinates. These inequalities play an important role in sharp-threshold theorems, junta approximation, noise sensitivity, percolation, social choice, hardness of approximation, and the geometry of product spaces; see~\cite{ODonnell14} for further discussion.

Most previous work presents isoperimetric inequalities on the Boolean hypercube equipped with the uniform measure. Extensions are known for general product measures and for structured spaces such as the hypergrid $[m]^n$ or the $k$-slice $\binom{[n]}{k}$.
Although uniform distributions are natural and cover many applications, many probabilistic and algorithmic models require non-uniform or non-product measures, including Gibbs measures of Markov random fields and graphical models that allow local dependence. 
Recent work has studied learning Boolean functions under such non-uniform measures \cite{CGMV26,FYYZ26}, as well as voting and coalition formation when voters' preferences are correlated \cite{KoehlerLifshitzMinzerMossel23}. Influence inequalities for dependent measures have also been used to establish sharp-threshold and phase-transition results for dependent percolation and spin systems, including the random-cluster, Ising, and Potts models \cite{GrahamGrimmett06,GG11,DuminilCopinRaoufiTassion19}. The purpose of this paper is to establish new isoperimetric inequalities beyond the uniform or product setting.

To introduce our main contributions, we first review basic definitions related to Boolean functions and classical isoperimetric inequalities under the uniform measure. We focus on four key inequalities: the Kahn--Kalai--Linial (KKL) inequality \cite{KKL88}, Talagrand's \( L^1 \)--\( L^2 \) inequality \cite{Talagrand94}, Talagrand's variance--surface-area inequality \cite{Talagrand93}, and the Eldan--Gross inequality \cite{EldanGross22}. Although closely related, these four inequalities measure the boundary of a Boolean function in different ways.

\subsection{Influence and boundary on the hypercube}
\label{subsec:inf-SA}

Let $\mu$ be a probability measure on $\cube$. In the classical setting, $\mu$ is the uniform measure, but the definitions in this section extend to any distribution. 

For $x=(x_1,\ldots,x_n)\in\cube$ and $i\in[n]$, let $x^i$ denote the point obtained by flipping the $i$th coordinate of $x$. 
Let $f: \cube \to \{\pm 1\}$ be a Boolean-valued function.
The expectation of $f$ is $\E[\mu]{f} = \sum_{x \in \cube} \mu(x) f(x)$, and the variance is $\Var[\mu]{f} = \E[\mu]{f^2} - (\E[\mu]{f})^2$.
The pointwise \emph{sensitivity} $s(f)$ of a Boolean function $f$ is
\[
    s(f)(x) :=\sum_{i=1}^n \ind{\{f(x^i)\neq f(x)\}},
\]
namely, $s(f)(x)$ counts the number of pivotal coordinates that change the function value when flipped. 
We define the following standard notation:
\begin{enumerate}[(1)]
    \item The \emph{influence of the $i$-th coordinate} $\Inf_{\mu,i}(f)$ is the probability that coordinate $i$ is pivotal:
    \begin{align*}
        \Inf_{\mu,i}(f)
        := \Pr[\mu]{f(x^i)\neq f(x)};
    \end{align*}

    \item The \emph{total influence} or \emph{average edge boundary} $\Inf_\mu(f)$ is defined as
    \begin{align*}
        \Inf_\mu(f)
        := \sum_{i=1}^n \Inf_{\mu,i}(f) = \sum_{i=1}^n \Pr[\mu]{f(x^i)\neq f(x)} = \E[\mu]{s(f)};
    \end{align*}

    \item The \emph{(Boolean) surface area} or \emph{Talagrand boundary} $\SurfArea_\mu(f)$ is the square-root sensitivity boundary, defined by
    \begin{align*}
        \SurfArea_\mu(f)
        := \E[\mu]{\sqrt{s(f)}};
    \end{align*}

    \item The \emph{total squared influence} $\SqInf_\mu(f)$ is the sum of squared influences:
    \begin{align*}
        \SqInf_\mu(f)
        := \sum_{i=1}^n \Inf_{\mu,i}(f)^2 = \sum_{i=1}^n
            \Pr[\mu]{f(x^i)\neq f(x)}^2.
    \end{align*}
\end{enumerate}


\subsection{The classical uniform-measure inequalities}

Throughout this subsection, $\mu$ denotes the uniform probability measure on $\cube$. We summarize the key isoperimetric inequalities studied in this work.

\paragraph{Poincaré inequality.} The Poincaré inequality gives the baseline comparison between the variance and the total influence. 
For any Boolean function $f: \cube \to \{\pm 1\}$,
\begin{equation}\label{eq:poincare-B}
    \Var[\mu]{f}
    \leq \Inf_\mu(f).
\end{equation}
The Poincaré inequality also holds for real-valued functions, with $\Inf_\mu(f)$ replaced by
\[
    \frac14\sum_{i=1}^n\E[\mu]{(f(x^i)-f(x))^2};
\]
equality is attained by dictators in the Boolean case.
It asserts that a function with nontrivial variance has a nontrivial edge boundary, but by itself guarantees only a coordinate with influence of order $\Var{f}/n$.

\paragraph{Kahn--Kalai--Linial (KKL) inequality \cite{KKL88}.}
The KKL theorem provides the decisive logarithmic improvement. 
There is a universal constant $C>0$ such that every Boolean-valued $f: \cube \to \{\pm 1\}$ satisfies
\begin{equation}\label{eq:kkl}
    \max_{i\in[n]} \,\Inf_{\mu,i}(f)
    \geq C\,\frac{\log n}{n} \,\Var[\mu]{f}.
\end{equation}
In particular, it identifies a coordinate whose influence improves on the Poincaré scale by a logarithmic factor.
This order of magnitude is optimal, as witnessed by the Tribes construction of Ben-Or and Linial \cite{BenOrLinial90}. KKL is not merely a refinement of \eqref{eq:poincare-B}: the diverging factor $\log n$ is what drives many sharp-threshold and coalition results.

\paragraph{Talagrand's $L^1$--$L^2$ inequality \cite{Talagrand94}.}
Talagrand's $L^1$--$L^2$ inequality provides a stronger statement. 
For every Boolean $f: \cube \to \{\pm 1\}$, it holds
\begin{equation}\label{eq:tal-influence}
    C\, \Var[\mu]{f}
    \leq \sum_{i=1}^n
    \frac{\Inf_{\mu,i}(f)}{1+\log(1/\Inf_{\mu,i}(f))},
\end{equation}
where a summand is zero if $\Inf_{\mu,i}(f) = 0$ by continuity, and $C>0$ is a universal constant. 
Like the Poincaré inequality, the functional $L^1$--$L^2$ inequality holds for general real-valued functions and compares $\Var[\mu]{f}$ with an $L^1$--$L^2$ interpolation of the coordinate derivatives.
Talagrand's $L^1$--$L^2$ inequality \eqref{eq:tal-influence} implies the KKL inequality \eqref{eq:kkl} by comparing each summand with the maximum influence. Analytically, the gain in the denominator comes from hypercontractivity of the heat semigroup.

\paragraph{Talagrand's variance--surface-area inequality \cite{Talagrand93}.}
Talagrand's second inequality is geometrically different: it relates $\SurfArea_\mu(f)$ to $\Var[\mu]{f}$ with the appropriate logarithmic correction. In the terminology of this paper, it is a variance--surface-area inequality: 
for any Boolean function $f:\cube\to\{\pm 1\}$, 
\begin{equation}\label{eq:tal-surface}
    \SurfArea_\mu(f)
    \geq C\,\Var[\mu]{f}
    \sqrt{\log\left(\frac{\e}{\Var[\mu]{f}}\right)},
\end{equation}
where $C>0$ is universal and the right-hand side is zero if $\Var[\mu]{f} = 0$ (i.e., $f$ is constant).
Bobkov~\cite{Bob97} established a Gaussian-profile functional inequality on the uniform cube, whose specialization to indicator functions yields a surface-area bound of the above form.
Here $\SurfArea_\mu(f)=\E[\mu]{\sqrt{s(f)}}$ gives less weight to vertices of very high sensitivity than the ordinary edge boundary $\Inf_\mu(f) = \E[\mu]{s(f)}$. This makes \eqref{eq:tal-surface} effective for functions such as majority, for which the edge boundary is concentrated on a thin layer of vertices with many incident bichromatic edges. 
By contrast, KKL \eqref{eq:kkl} is sharp for Tribes. Neither viewpoint alone fully captures both examples.

\paragraph{Eldan--Gross inequality \cite{EldanGross22}.}
Talagrand asked for a single inequality that reconciles the KKL inequality \eqref{eq:kkl} with the variance--surface-area inequality \eqref{eq:tal-surface} \cite{Talagrand97}. 
Eldan and Gross resolved this question using pathwise stochastic analysis \cite{EldanGross22}. 
Their inequality states that every Boolean $f:\cube\to\{\pm 1\}$ satisfies
\begin{equation}\label{eq:eldan-gross}
    \SurfArea_\mu(f)
    \geq C\,\Var[\mu]{f}
    \sqrt{\log\left(2+\frac{\e}{\SqInf_\mu(f)}\right)},
\end{equation}
where $C>0$ is a universal constant.
For a fixed total influence, the quantity $\SqInf_\mu(f)$ reflects the concentration of the boundary among coordinate directions, distinguishing a boundary spread thinly over many coordinates from one concentrated in a few directions.
This is the strongest geometric statement since it preserves the square-root sensitivity boundary while recovering the KKL bound through the squared influences.
Subsequent work gave alternative proofs based on Fourier analysis and semigroup methods \cite{RvH-remarks,EldanKindlerLifshitzMinzer25,BIM23,IvanisviliZhang26}.

\subsection{Main result}

In this work, we extend all four inequalities to non-product measures on the Boolean hypercube with weak dependence among the coordinates.
Previous work \cite{KoehlerLifshitzMinzerMossel23} pursued this direction for Markov random fields on bounded-degree graphs; in contrast, our setting is more general.

For simplicity, we present our results here under a Dobrushin-type condition \cite{Dobrushin68}, an intuitive and common way to characterize weak dependence of variables. The assumptions in our formal results are more general and are implied by the Dobrushin condition stated here.

For an arbitrary distribution $\mu$ fully supported on $\cube$, its associated \emph{Dobrushin influence matrix} $R = R(\mu) \in [0,1]^{n \times n}$ is defined as follows:
\begin{align}
    R_{ij} &:= \max_{x \in \cube} \TV{\mu_i(\cdot \mid X_{-i} = x_{-i})}{\mu_i(\cdot \mid X_{-i} = x^j_{-i})}, \quad \forall i,j \in [n],\; i \neq j; \nonumber \\
    R_{ii} &:= 0, \quad \forall i \in [n].
\end{align}
Here, the conditioning $X_{-i} = x_{-i}$ means all coordinates other than $i$ are given by $x$, and $X_{-i} = x^j_{-i}$ means the same conditioning but with the $j$-th coordinate flipped.
Thus, $R_{ij}$ characterizes the influence on the $i$-th coordinate in total variation distance when the $j$-th coordinate is flipped.
Observe that $R$ is a non-negative square matrix that is not necessarily symmetric. For a matrix $M$, we use $\norm{M}_1=\max_j\sum_i|M_{ij}|$ and $\norm{M}_2=\sup_{v\ne0}\norm{Mv}_2/\norm{v}_2$.

The Dobrushin matrix quantifies the dependence among the random variables.
Informally, the ``smaller'' $R$ is, typically measured in an appropriate matrix norm, the weaker the dependence among the variables.
For example, if $\mu$ is a product measure with independent coordinates, then $R = 0$. 
In particular, under the \emph{Dobrushin uniqueness condition} $\norm{R} < 1$, where $\norm{\cdot}$ is any consistent matrix norm, the Gibbs sampler (also known as Glauber dynamics) mixes rapidly \cite{Hayes06,DGJ09}. For the functional inequalities used here, we impose $\norm{R}_2 \le 1-\delta$ and uniformly positive single-site conditional marginals. These assumptions imply a dimension-free log-Sobolev inequality and hence hypercontractivity and concentration of measure~\cite{marton2019logarithmic,SS20}.

\begin{theorem}
    \label{thm:informal}
    Let $b \in (0,1]$, $\delta \in (0,1)$, and $C_0 > 0$ be constants.
    Let $\mu$ be a distribution fully supported on $\cube$ satisfying the following conditions: 
    \begin{enumerate}[(i)]
        \item (Marginal boundedness) For every $i \in [n]$ and $x \in \cube$,
        \begin{align*}
            b \le \sqrt{\frac{\mu(x^i)}{\mu(x)}} \le \frac{1}{b};
        \end{align*}

        \item (Dobrushin condition) The Dobrushin matrix $R$ of $\mu$ satisfies $\norm{R}_1 \le C_0$ and $\norm{R}_2 \le 1-\delta$.
    \end{enumerate}
    Then, for any Boolean $f: \cube \to \{\pm 1\}$, all of the following isoperimetric inequalities hold, where $C>0$ depends only on $b$, $\delta$ and $C_0$, but not on $n$.
    \begin{enumerate}[(1)]
        \item KKL inequality:
        \begin{equation}\label{eq:inf-KKL}
            \max_{i\in[n]} \,\Inf_{\mu,i}(f)
            \geq C\, \frac{\log n}{n} \,\Var[\mu]{f};
        \end{equation}

        \item Talagrand's $L^1$--$L^2$ inequality:
        \begin{equation}\label{eq:inf-L1-L2}
            C\,\Var[\mu]{f}
            \leq \sum_{i=1}^n
            \frac{\Inf_{\mu,i}(f)}{1+\log(1/\Inf_{\mu,i}(f))};
        \end{equation}

        \item Talagrand's variance--surface-area inequality:
        \begin{equation}\label{eq:inf-var-SA}
            \SurfArea_\mu(f)
            \geq C\,\Var[\mu]{f}
            \sqrt{\log\left(\frac{\e}{\Var[\mu]{f}}\right)};
        \end{equation}

        \item Eldan--Gross inequality:
        \begin{equation}\label{eq:inf-EG}
            \SurfArea_\mu(f)
            \geq C\,\Var[\mu]{f}
            \sqrt{\log\left(1+\frac{\e}{\SqInf_\mu(f)}\right)}.
        \end{equation}
    \end{enumerate}
\end{theorem}

As a concrete example, we consider the zero-field Ising model on $n$ spins. Let $J \in \R^{n \times n}$ be a symmetric interaction matrix. The Gibbs distribution $\mu_J$ for the Ising model assigns each spin configuration $x \in \cube$ the probability mass
\begin{align}\label{eq:Ising}
    \mu_J(x) := \frac{1}{Z_J} \exp\left( \frac{1}{2} x^\top J x \right),
\end{align}
where $Z_J := \sum_{x \in \cube} \exp(\frac{1}{2} x^\top J x)$ is the partition function that normalizes the distribution.
An entrywise comparison gives $\max\set{\norm{R}_1,\norm{R}_2} \le \norm{J}_1$, where $R$ denotes the Dobrushin matrix. 
Moreover, $\norm{J}_1$ controls the marginal boundedness since it corresponds to the largest effective field resulting from pinning vertices.
Therefore, the above inequalities hold when $\norm{J}_1 < 1$.

\begin{corollary}
    \label{cor:informal-ising}
    Let $J \in \R^{n \times n}$ be a symmetric interaction matrix for a zero-field Ising model. Fix $\delta \in (0,1)$ and suppose $\norm{J}_1 \le 1-\delta$. Then the KKL inequality~\eqref{eq:inf-KKL}, Talagrand's $L^1$--$L^2$ inequality~\eqref{eq:inf-L1-L2}, Talagrand's variance--surface-area inequality~\eqref{eq:inf-var-SA}, and the Eldan--Gross inequality~\eqref{eq:inf-EG} all hold with constants depending only on $\delta$.
\end{corollary}

Our results generalize the classical inequalities under the uniform measure, allowing for bias and dependence in the underlying measure. 
The $L^1$--$L^2$ and KKL statements should be compared with those of \cite{KoehlerLifshitzMinzerMossel23} for mixing measures. The main result of \cite{KoehlerLifshitzMinzerMossel23} applies to sparse Markov random fields, especially to undirected graphical models or spin systems such as the Ising model.
In addition to marginal boundedness, their result requires a log-Sobolev inequality together with a ``bounded-degree'' assumption. In particular, for the Ising model, their result applies when $\norm{J}_1<1$ and the underlying graph, corresponding to the support of $J$, has bounded maximum degree. 
In contrast, our result applies in the same $\norm{J}_1<1$ regime but imposes no degree bound. Furthermore, our main theorem \Cref{thm:informal} applies beyond sparse Markov random fields to any distribution on weakly dependent random variables.
For Ising models, our results also hold beyond the Dobrushin condition in the whole tree-uniqueness regime on bounded-degree graphs, where the $L^1$--$L^2$ and KKL inequalities were previously established in~\cite{KoehlerLifshitzMinzerMossel23}.

\Cref{thm:informal,cor:informal-ising} also have applications to voting with correlated preferences. We refer to~\cite[Sections~2.2 and~3.3]{KoehlerLifshitzMinzerMossel23} for these applications, as the same proofs apply.

Our results also extend beyond the Boolean hypercube.
In particular, all four results admit counterparts on the hypergrid $[m]^n$ with suitable definitions of influences; the corresponding measure assumptions and constants depend on the alphabet size and local geometry.
They also hold for uniform measures on balanced $k$-slices $\binom{[n]}{k}$, with influences defined via coordinate transpositions.
To keep the main narrative focused on $\cube$, we do not state these extensions in this introduction; their precise formulations and proofs appear later in \Cref{sec:Schreier}. This extension is consistent with earlier appearances of KKL inequalities on $\mathbb Z_m^n$ and other finite product or Schreier spaces \cite{ODonnellWimmer13,ODonnellWimmerSharp13,CorderoErausquinLedoux12}, while covering the non-uniform setting addressed here.

\paragraph{Disclosure of AI assistance and author contributions.}
At an early stage of this project, the authors developed a direct semigroup proof of the local Bobkov inequality~\cite{BakryLedoux96} for the uniform measure on the Boolean hypercube. This inequality is a key ingredient in the approach to the Eldan--Gross inequality of~\cite{IvanisviliZhang26}, where the discrete-cube argument uses coordinate independence and applies induction on the dimension~\cite{BobkovGotze99}. The authors then used ChatGPT 5.5 Pro to explore extensions of the semigroup argument to the uniform measure on the $k$-slice and to Ising measures at sufficiently high temperature. This exploration produced several preliminary, problem-specific arguments. Drawing on these examples and on earlier work such as~\cite{BGL-book,CorderoErausquinLedoux12}, the authors identified curvature as a unifying framework for the hypotheses underlying these inequalities, thereby systematizing the arguments and clarifying the connection between techniques from sampling and Boolean analysis. ChatGPT 5.5, 5.6, and 6 were also used during the preparation of the manuscript to assist with checking calculations, exploring proof steps, and improving the exposition. 
ChatGPT-6 Astra was used to help find the examples in \Cref{sec:different-rates-weights} and to assist with the numerical experiments and proofs in \Cref{app:one-edge-Ising}.
We also used Google DeepReview for STOC 2027 to improve the clarity of the writing.
The authors take full responsibility for the content of the paper.

\paragraph{Acknowledgements.}
Zongchen Chen thanks Shiping Liu for introducing him to various notions of discrete curvature and for pointing him to relevant references.

\paragraph{Paper organization.}
In \Cref{sec:curvature}, we illustrate our approach through a variant of Talagrand's $L^1$--$L^2$ inequality on the Boolean hypercube, showing how curvature and semigroup methods yield isoperimetric inequalities.
In \Cref{sec:preliminaries}, we present the necessary preliminaries.
In \Cref{sec:BE-GE}, we recall the Bakry--\'Emery and reinforced Bakry--\'Emery conditions and explain their connections with gradient estimates.
In \Cref{sec:boolean}, we combine these conditions with the log-Sobolev inequality to prove isoperimetric inequalities.
In \Cref{sec:GE}, we establish Dobrushin-type sufficient conditions for curvature bounds and gradient estimates on the Boolean hypercube and apply the resulting inequalities to the Ising model.
In \Cref{sec:Schreier}, we extend this framework to Schreier graphs and derive isoperimetric inequalities for the $k$-slice and hypergrids.

\section{Our Approach: Curvature}\label{sec:curvature}

Many approaches to isoperimetric inequalities on the discrete Boolean cube draw on tools such as Fourier analysis, hypercontractivity, martingale methods, induction, random restrictions, and heat-semigroup techniques. Extending these approaches to general measures, especially non-product measures, presents immediate difficulties. For a general measure $\mu$, the standard parity functions $\chi_S(x)=\prod_{i\in S}x_i$ need not form an orthonormal basis, the associated semigroup need not act diagonally in this basis, and the coordinate-flip operators need not be reversible with respect to~$\mu$. At the operator level, these difficulties are further reflected in the general failure of coordinate derivatives to commute with the associated semigroup.

We establish isoperimetric inequalities through a purely semigroup approach.
We show that these inequalities can be deduced from a curvature condition and the hypercontractivity of the associated semigroup. More precisely, our method uses discrete versions of the Bakry--\'Emery curvature-dimension conditions and gradient estimates that are well developed in continuous space. Because establishing such a curvature condition on the Boolean hypercube is nontrivial, we introduce a weighted version and reduce its verification to local algebraic properties of the measure.

At a high level, our proof is a novel extension and combination of the approaches in \cite{CorderoErausquinLedoux12,IvanisviliZhang26} for establishing isoperimetric inequalities and those in \cite{CKLP22} for deducing curvature conditions. Here, we give a brief discussion. A more detailed comparison is presented in \Cref{subsec:related-work}.

Cordero-Erausquin and Ledoux \cite{CorderoErausquinLedoux12} introduced a curvature-type condition and derived Talagrand's $L^1$--$L^2$ inequality from it. However, their curvature assumption is strong in its original form; therefore, their results apply mainly to product measures or structured spaces.
Ivanisvili and Zhang \cite{IvanisviliZhang26} established the Eldan--Gross inequality in continuous space for the product of measures satisfying the Bakry--\'Emery curvature-dimension condition. They also proved the Eldan--Gross inequality on the Boolean hypercube under a product measure, but their approach ultimately relies on coordinate independence and induction on the dimension~\cite{BobkovGotze99}, rather than on curvature.
Our approach draws inspiration from both works \cite{CorderoErausquinLedoux12,IvanisviliZhang26}: although neither mentions curvature on the Boolean hypercube, their approaches strongly suggest its use. 
Cushing, Kamtue, Liu, and Peyerimhoff reduced the Bakry--\'Emery curvature on arbitrary graphs to a local minimal eigenvalue problem \cite{CKLP22}. We focus on measures on the hypercube and introduce appropriate choices for the semigroup and the reweighting to control the curvature.

\subsection{A running example and our proof approach}
In this section, we illustrate the key ingredients of our approach by sketching a proof of an alternative version of Talagrand's $L^1$--$L^2$ inequality.

\begin{theorem}[Alternative $L^1$--$L^2$ inequality]
\label{thm:alt-L1-L2}
    Let $\mu$ be a distribution fully supported on $\cube$ satisfying the marginal boundedness condition and the Dobrushin condition stated in \Cref{thm:informal}. 
    Then, for any Boolean $f: \cube \to \{\pm 1\}$,
    \begin{align}\label{eq:alt-L1-L2}
        \Var[\mu]{f} \le C \frac{\Inf_\mu(f)}{1 + \log (\Inf_\mu(f)/\SurfArea_\mu(f)^2)}.
    \end{align}
    Here, $C>0$ is a constant independent of $n$.
\end{theorem}

Our strategy for establishing \Cref{thm:alt-L1-L2}, as well as all the inequalities in \Cref{thm:informal}, consists of two steps.
\begin{enumerate}[\bfseries Step 1.]
    \item Identify an appropriate curvature condition for the measure $\mu$ and its associated Markov processes from which the target isoperimetric inequality can be derived. 

    \item Choose suitable Markov semigroups and weighting schemes under which the desired curvature conditions hold.
\end{enumerate}

In Step 1, we use discrete versions of the Bakry--\'Emery curvature-dimension conditions that have been studied extensively in continuous space, and we give black-box reductions to these conditions, together with hypercontractivity.
In Step 2, we introduce a canonical Markov process with suitable weights and verify the curvature properties for this particular choice.

\subsection{Semigroup and hypercontractivity}

Although the statement of \Cref{thm:alt-L1-L2} does not involve a Markov process, its proof requires a suitable hypercontractive semigroup.

Let $\mu$ be a distribution fully supported on $\cube$. 
We define a \emph{canonical} continuous-time Markov process with semigroup $(P_t)_{t\ge 0}$ and generator $L$. Specifically, $P_t = e^{t L}$ for $t \ge 0$, and, for any function $f: \cube \to \R$,
\begin{align*}
    Lf(x) = \sum_{i=1}^n q_i(x) (f(x^i)-f(x)),
\end{align*}
where $q_i(\cdot)$ denotes the transition rate for flipping the $i$-th coordinate and is given by
\begin{align*}
    q_i(x) = \frac{1}{2} \sqrt{\frac{\mu(x^i)}{\mu(x)}}, \quad \forall x \in \cube.
\end{align*}
This semigroup is reversible with respect to $\mu$ because the detailed balance identity $\mu(x) q_i(x) = \mu(x^i) q_i(x^i)$ holds for each $i \in [n]$ and $x \in \cube$. The significance of this canonical choice will become clear later.

As in previous works, we require the heat semigroup $(P_t)_{t\ge 0}$ to be hypercontractive. Hypercontractivity is equivalent to a log-Sobolev inequality and plays a key role in the analysis of mixing times for Markov processes.
For a real $p \ge 1$ and a function $f: \cube \to \R$, the $L^p(\mu)$-norm of $f$ is defined as
\begin{align*}
    \norm{f}_{p,\mu} := \left( \E[\mu]{|f(x)|^p} \right)^{1/p}.
\end{align*}
For a constant $\rho >0$, \emph{hypercontractivity} means that, for any $1 < p \leq q < \infty$ satisfying $\frac{q-1}{p-1}\leq e^{2\rho t}$ and every $f: \cube \to \R$, the following holds:
\begin{align}\label{eq:hypercontractivity-intro}
    \norm{P_t f}_{q,\mu} \leq \norm{f}_{p,\mu}.
\end{align}

The following lemma is a direct consequence of hypercontractivity.
\begin{lemma}\label{lem:hypercontractivity}
    Let $t \ge 0$. For any function $f: \cube \to \R$,
    \begin{align*}
        \norm{P_t f}_{2,\mu} \le \norm{f}_{1,\mu}^{\theta_t} \norm{f}_{2,\mu}^{1-\theta_t},
    \end{align*}
    where $\theta_t = \tanh (\rho t)$.
\end{lemma}
\begin{proof}[Proof sketch]
    Let $p_t = 1 + \e^{-2\rho t}$. Then $\theta_t + \frac{1-\theta_t}{2} = \frac{1}{p_t}$, and we deduce that
    \begin{align}
        \norm{P_t f}_{2,\mu} 
        &\le \norm{f}_{p_t,\mu} \tag{Hypercontractivity} \\
        &\le \norm{f}_{1,\mu}^{\theta_t} \norm{f}_{2,\mu}^{(1-\theta_t)}, \tag{H\"older's inequality}
    \end{align}
    as claimed.
\end{proof}

\subsection{Bakry--\'Emery theory and gradient estimates}
We now introduce the basic notions of Bakry--\'Emery theory and the associated gradient estimate conditions. For our purposes, we focus on the Boolean hypercube $\cube$, although these notions extend to a general finite state space and a reversible Markov semigroup. 

The \textit{carr\'e du champ operator} $\Gamma$ is defined, for any functions $f,g:\cube \to \R$, by
\begin{align*}
    \Gamma(f,g) := \frac 12 (L(fg)-gLf-fLg), \quad \Gamma(f) := \Gamma(f,f).
\end{align*}
A straightforward calculation yields
\begin{align*}
    \Gamma(f,g)(x) = \frac 12\sum_{i=1}^n q_i(x)(f(x^i) - f(x))(g(x^i) - g(x)),
    \quad \Gamma(f)(x) = \frac 12\sum_{i=1}^n q_i(x)(f(x^i) - f(x))^2.
\end{align*}
The following versions of the discrete derivatives and gradient provide an intuitive representation of the carr\'e du champ operator.
The \textit{discrete gradient} of a function $f: \cube \to \mathbb{R}$ is defined as the vector 
\begin{align*}
    \nabla f := (\delta_1 f, \ldots, \delta_n f),\quad \text{where} \quad \delta_i f := \sqrt{\frac 12 q_i(x)}(f(x^i) - f(x)).
\end{align*}
It follows that $\Gamma(f,g) = \inner{\nabla f}{\nabla g}$ and $\Gamma(f) = \norm{\nabla f}_2^2$.

The \textit{iterated carr\'e du champ operator} $\Gamma_2$ is defined, for any function $f:\cube \to \R$, by
\begin{align*}
    \Gamma_2(f) := \frac 12 L\Gamma(f) - \Gamma(f,Lf).
\end{align*}

Let $\rho \in \R$ be a constant. We now introduce the principal definitions of Bakry--\'Emery criteria and the associated gradient-estimate properties:
\begin{itemize}
    \item \emph{Bakry--\'Emery (BE) condition:} For every function $f: \cube \to \R$, it holds
    \begin{align*}
        \Gamma_2(f) \geq \rho \cdot \Gamma(f);
    \end{align*}

    \item \emph{Reinforced Bakry--\'Emery (rBE) condition:} For every function $f: \cube \to \R$, it holds
    \begin{align}
    \label{eq:rBE-intro}
        \Gamma_2(f) - \Gamma(\sqrt{\Gamma(f)}) \ge \rho \cdot \Gamma(f);
    \end{align}

    \item \emph{Weak Gradient Estimate (WGE):} For every function $f:\cube \to \R$ and $t\geq 0$, it holds
    \begin{align}
        \Gamma(P_t f) \leq e^{-2\rho t} P_t(\Gamma(f)) \quad\iff\quad \norm{\nabla P_t f}^2_2 \leq e^{-2\rho t} P_t(\norm{\nabla f}^2_2);
    \end{align}

    \item \emph{Strong Gradient Estimate (SGE):} For every function $f: \cube \to \R$ and $t\geq 0$, it holds
    \begin{align}\label{eq:SGE-intro}
        \sqrt{\Gamma(P_t f)} \leq e^{-\rho t} P_t \sqrt{\Gamma(f)} \quad\iff\quad 
        \norm{\nabla P_t f}_2 \leq e^{-\rho t} P_t(\norm{\nabla f}_2).
    \end{align}
\end{itemize}

One crucial motivation and application of the Bakry--\'Emery theory lies in the analysis of the convergence of the Markov process to equilibrium. Specifically, to bound the mixing time, one aims to study the decay of the variance functional $\Var[\mu]{P_t f}$ as time $t$ grows for a test function $f$. A direct calculation yields
\begin{align}\label{eq:dif-Var}
    \frac{\dif}{\dif t} \Var[\mu]{P_t f} = - 2 \E[\mu]{\Gamma(P_t f)}
    \qquad 
    \text{and}
    \qquad
    \frac{\dif}{\dif t} \E[\mu]{\Gamma(P_t f)} = - 2 \E[\mu]{\Gamma_2(P_t f)}.
\end{align}
Therefore, the Bakry--\'Emery condition gives $\frac{\dif^2}{\dif t^2} \Var[\mu]{P_t f} \ge -2\rho \cdot \frac{\dif}{\dif t} \Var[\mu]{P_t f}$, yielding an exponential decay of variance $\Var[\mu]{P_t f} \le e^{-2\rho t} \Var[\mu]{f}$ when $\rho > 0$.

The following proposition establishes the equivalence between curvature conditions and gradient estimates and is central to Bakry--\'Emery theory.

\begin{proposition}\label{prop:intro-BE-GE}
    Fix $\rho \in \R$. Then: (1) BE $\Leftrightarrow$ WGE; (2) rBE $\Leftrightarrow$ SGE.
\end{proposition}

\Cref{prop:intro-BE-GE} is a fundamental result whose proof can be found, for example, in the survey \cite{Salez-survey} or the thesis \cite{Kamtue-thesis}.
We sketch the proof of the implication BE $\Rightarrow$ WGE, which is based on the \emph{interpolation argument} from continuous space \cite{BGL-book}.
\begin{proof}[Proof sketch of BE $\Rightarrow$ WGE]
    Let $F_s := e^{-2\rho s} P_s \Gamma(P_{t-s} f)$. Our goal is to show $F_0 \le F_t$.
    Differentiating with respect to $s$ yields
    \begin{align*}
        \frac{\dif}{\dif s} F_s 
        &= e^{-2\rho s}\left( - 2 \rho \cdot P_s \Gamma(P_{t-s} f) + \frac{\dif}{\dif s} P_s \Gamma(P_{t-s} f) \right) \\
        &= e^{-2\rho s}\left( - 2 \rho \cdot P_s \Gamma(P_{t-s} f) + P_s \left( L + \frac{\dif}{\dif s} \right) \Gamma(P_{t-s} f) \right) \tag{$\frac{\dif}{\dif s} P_s g_s = P_s L g_s + P_s (\frac{\dif}{\dif s} g_s)$} \\
        &= 2 e^{-2\rho s} P_s \left( - \rho \cdot \Gamma(P_{t-s} f) + \frac{1}{2} L \Gamma(P_{t-s} f) - \Gamma(P_{t-s} f, L P_{t-s} f) \right) \tag{$\frac{\dif}{\dif s} \Gamma(h_s) = 2 \Gamma(h_s, \frac{\dif}{\dif s} h_s)$} \\
        &= 2 e^{-2\rho s} P_s \left( - \rho \cdot \Gamma(P_{t-s} f) + \Gamma_2(P_{t-s} f) \right) \ge 0,
    \end{align*}
    and hence $F_0 \le F_t$.
\end{proof}

It is not hard to see that rBE and SGE are stronger than BE and WGE and imply them.
For diffusion processes in continuous space, such as Langevin diffusion, the converse implication also holds by the diffusion chain rule. Thus, BE $\Leftrightarrow$ rBE $\Leftrightarrow$ WGE $\Leftrightarrow$ SGE. These equivalences are a striking feature of the elegant Bakry--\'Emery theory.
In discrete space, however, rBE and SGE are typically strictly stronger than BE and WGE; see \Cref{sec:different-rates-weights}.

\subsection{From curvature conditions to isoperimetric inequalities}

We show that the $L^1$--$L^2$ inequality \eqref{eq:alt-L1-L2} follows from two properties of the underlying measure and its associated semigroup: 
(1) the hypercontractivity, or equivalently, the log-Sobolev inequality of the semigroup; and
(2) the reinforced Bakry--\'Emery (rBE) condition, or equivalently, the strong gradient estimate (SGE) property.

\begin{theorem}[\cite{CorderoErausquinLedoux12}]
\label{thm:intro-L1-L2}
    Suppose the measure $\mu$ satisfies marginal boundedness, and the associated semigroup $(P_t)_{t \ge 0}$ satisfies both hypercontractivity~\eqref{eq:hypercontractivity-intro} and rBE~\eqref{eq:rBE-intro}. Then the $L^1$--$L^2$ inequality~\eqref{eq:alt-L1-L2} holds.
\end{theorem}

\begin{proof}[Proof sketch]
    Let $\rho \in \mathbb{R}$ denote the constant for SGE, and $\rho_0 > 0$ denote the log-Sobolev constant. 
    We first consider a real-valued function $f: \cube \to \R$ and later specialize to the Boolean-valued case.
    We begin with an integration formula for the variance, which follows from \eqref{eq:dif-Var}:
    \begin{align*}
        (1-\e^{-2 \rho_{0} T}) \Var[\mu]{f} \le \Var[\mu]{f} - \Var[\mu]{P_T f} = 2 \int_0^{T} \E[\mu]{\Gamma(P_t f)} \dif t.
    \end{align*}
    We then deduce from SGE and \Cref{lem:hypercontractivity} that
    \begin{align*}
        (1-\e^{-2\rho_{0}T}) \Var[\mu]{f} 
        &\le 2 \int_0^{T} \E[\mu]{\left( \sqrt{\Gamma(P_t f)} \right)^2} \dif t \\
        &\le 2 \int_0^{T} \e^{-2 \rho t} \cdot \E[\mu]{\tp{P_t \sqrt{\Gamma f}}^2} \dif t \tag{SGE} \\
        &= 2 \int_0^{T} \e^{-2 \rho t} \cdot \norm{P_t \sqrt{\Gamma f}}_{2,\mu}^2 \dif t \\ 
        &\le 2 \int_0^{T} \e^{-2 \rho t} \cdot \norm{\sqrt{\Gamma f}}_{1,\mu}^{2 \theta_t} \norm{\sqrt{\Gamma f}}_{2,\mu}^{2 (1-\theta_t)} \dif t \tag{\Cref{lem:hypercontractivity}, $\theta_t = \tanh (\rho_0 t)$} \\
        &= 2 \norm{\sqrt{\Gamma f}}_{2,\mu}^{2} \int_0^{T} \e^{-2 \rho t} \cdot \tp{\frac{\norm{\sqrt{\Gamma f}}_{1,\mu}^2}{\norm{\sqrt{\Gamma f}}_{2,\mu}^2}}^{\theta_t} \dif t.
    \end{align*}
    Choose $T = \tp{\rho_{0} + \max \set{-\rho,0}}^{-1}$.
    A suitable upper bound on this integral, whose technical proof is omitted from this sketch, gives
    \begin{align*}
        \Var[\mu]{f} \lesssim \frac{\norm{\sqrt{\Gamma f}}_{2,\mu}^2}{1+\log \tp{ \norm{\sqrt{\Gamma f}}_{2,\mu}^2 / \norm{\sqrt{\Gamma f}}_{1,\mu}^2 }}.
    \end{align*}
    This inequality can be viewed as an $L^1$--$L^2$ inequality for real $f$. Finally, for Boolean $f$, one simply observes that, if the transition rates $\{q_i\}_{i\in[n]}$ are bounded, 
    \begin{align*}
        \Inf_\mu(f) \approx \norm{\sqrt{\Gamma f}}_{2,\mu}^2
        \quad \text{and} \quad
        \SurfArea_\mu(f) \approx \norm{\sqrt{\Gamma f}}_{1,\mu}
    \end{align*}
    and hence \eqref{eq:alt-L1-L2} follows.
\end{proof}

The proofs of the alternative $L^1$--$L^2$ inequality in \Cref{thm:intro-L1-L2} and the original $L^1$--$L^2$ inequality in \Cref{thm:informal} follow the same strategy developed by \cite{CorderoErausquinLedoux12}. We point out that we actually require a weighted version of rBE in \Cref{thm:intro-L1-L2}; this will be discussed in the next subsection. However, the main steps of the proof remain the same as in \cite{CorderoErausquinLedoux12} even with the weights. We believe our main contribution lies in reorganizing the proof and \emph{identifying the precise curvature condition} it requires. In particular, for the original $L^1$--$L^2$ inequality in \Cref{thm:informal}, we introduce a coordinate-wise version of SGE that is the essential property needed for the proof.
By contrast, \cite{CorderoErausquinLedoux12} imposes a structural assumption that appears feasible only for product measures.

\subsection{Establishing curvature conditions via reweighting}
By \Cref{thm:intro-L1-L2}, establishing \Cref{thm:alt-L1-L2} reduces to establishing both hypercontractivity and rBE. Hypercontractivity follows from the Dobrushin condition, via a result in the study of Markov chain mixing times \cite{marton2019logarithmic}. The main obstacle is therefore to prove rBE, or equivalently, the SGE property. To apply \Cref{thm:intro-L1-L2}, we need only a dimension-free lower bound on $\rho_\RBE$ or $\rho_\SGE$; these constants may be negative but must remain bounded below uniformly in $n$.

Proving rBE is challenging in general, and we cannot establish it in its original form. Instead, we consider weighted versions of the carr\'e du champ operator and the discrete gradient, and then define and prove corresponding weighted versions of rBE and SGE.
Specifically, for each $i \in [n]$ and $x \in \cube$, we define transition weights
\begin{align*}
    w_i(x) = 2 q_i(x) = \sqrt{\frac{\mu(x^i)}{\mu(x)}}.
\end{align*}
We then define the weighted carr\'e du champ operator and the weighted discrete gradient, respectively, by
\begin{align*}
    &\Gamma_w(f)(x) := \frac 12\sum_{i=1}^n w_i(x) q_i(x)(f(x^i) - f(x))^2
    = \sum_{i=1}^n q_i(x)^2(f(x^i) - f(x))^2;\\
    &\nabla_w f := (\delta_{w,1} f, \ldots, \delta_{w,n} f),
    \text{ where } \delta_{w,i} f := \sqrt{\frac 12 w_i(x) q_i(x)}(f(x^i) - f(x)) = q_i(x) (f(x^i) - f(x)).
\end{align*}
As before, $\Gamma_w(f) = \norm{\nabla_w f}_2^2$.
We define the weighted version of the iterated carr\'e du champ operator $\Gamma_{2;w}$ analogously.
The rBE condition and SGE property then take the following weighted forms: for every function $f: \cube \to \R$,
\begin{align}
    &\Gamma_{2;w}(f) - \Gamma(\sqrt{\Gamma_w(f)}) \ge \rho \cdot \Gamma_w(f); \tag{rBE} \\
    &\norm{\nabla_w P_t f}_2 \leq e^{-\rho t} P_t(\norm{\nabla_w f}_2). \tag{SGE}
\end{align}
The equivalence rBE $\Leftrightarrow$ SGE continues to hold, as in \Cref{prop:intro-BE-GE}.
More importantly, this choice of weights has the following two key properties.

\begin{itemize}
    \item \Cref{thm:intro-L1-L2} continues to hold in the weighted setting, provided only that the weights are bounded above and below; marginal boundedness of the measure guarantees this property.
    \item We can establish the weighted versions of rBE and SGE using a useful \emph{commutator identity} (\Cref{lem:commutator}) for the weighted gradient $\grad_w$ and the generator $L$: for every function $f$,
    \begin{align}\label{eq:com-id}
        (L\grad_w - \grad_w L) f = H \grad_w f,
    \end{align}
    where $L \grad_w$ denotes the entrywise application of $L$ and $H: \cube \to \R^{n \times n}$ is a matrix field that can be written down explicitly. 
    Thus, the commutator $L\grad_w - \grad_w L$ acts pointwise and linearly on the gradient $\grad_w f$. This commutator identity plays a crucial role in our derivation of SGE and its coordinate-wise generalization on the discrete hypercube.
\end{itemize}

The commutator identity \eqref{eq:com-id} has a familiar analog in Bakry--\'Emery theory for Langevin diffusion with generator $L=\Delta-\langle\nabla V,\nabla\rangle$ and invariant measure $\mu(\dif x)\propto e^{-V(x)}\dif x$. In this setting, $H=\nabla^2V$ is the Hessian field of the potential $V$, and the commutator identity yields the corresponding Bochner identity.

\subsection{Comparison with previous work}
\label{subsec:related-work}

Our approach draws insights from three previous works: coordinate-wise gradient estimates for the $L^1$--$L^2$ inequality \cite{CorderoErausquinLedoux12}, the local Bobkov route to the surface-area inequality \cite{IvanisviliZhang26}, and local matrix criteria for discrete Bakry--\'Emery curvature \cite{CKLP22}.

\paragraph{Comparison with \cite{CorderoErausquinLedoux12}.}
Cordero-Erausquin and Ledoux \cite{CorderoErausquinLedoux12} developed a hypercontractive semigroup proof of the KKL inequality and Talagrand's $L^1$--$L^2$ inequality based on coordinate-wise gradient estimates. An absolute-gradient formulation on the hypercube is that for every $i \in [n]$, every function $f: \cube \to \R$, and every $t\geq 0$,
\begin{align}\label{eq:CL-condition}
    \abs{\delta_i P_t f} \le e^{-\rho t} P_t \abs{\delta_i f}.
\end{align}
Summing the squares of \eqref{eq:CL-condition} over $i$ gives the weaker condition:
\begin{align}\label{eq:CL-condition-sum}
    \norm{\abs{\grad} P_t f}_2 \le e^{-\rho t} \norm{P_t \abs{\grad} f}_2.
\end{align}
Thus, while \eqref{eq:CL-condition} seems only to hold for product measures, \eqref{eq:CL-condition-sum} is a much more suitable condition for non-product measures. In our proof, we use a weighted version of \eqref{eq:CL-condition-sum} (see \Cref{def:CSGE}), which is the key ingredient for both the $L^1$--$L^2$ and the Eldan--Gross inequality. Furthermore, our gradient estimates require weighting so that they can hold under weak dependence.

\paragraph{Comparison with \cite{IvanisviliZhang26}.}
Ivanisvili and Zhang \cite{IvanisviliZhang26} gave an alternative semigroup proof of the Eldan--Gross inequality based on a \emph{local Bobkov inequality}. They treated both biased product measures on the Boolean hypercube and products of diffuse Markov triples satisfying the Bakry--\'Emery curvature-dimension condition. In the diffuse setting, the local Bobkov inequality follows from the curvature condition \cite{BGL-book}.
On the discrete Boolean hypercube, they obtain the required inequality by exploiting coordinate independence and applying the product-measure isoperimetric inequality of Bobkov and G\"otze~\cite{BobkovGotze99}.
In contrast, we derive a discrete local Bobkov inequality directly from gradient estimates without relying on coordinate independence.
We then follow the overall strategy of~\cite{IvanisviliZhang26} to establish the Eldan--Gross inequality.
The technical novelty lies in using appropriate curvature conditions to control the dependence among coordinates systematically.

\paragraph{Comparison with \cite{CKLP22}.}
Cushing, Kamtue, Liu, and Peyerimhoff \cite{CKLP22} showed that discrete Bakry--\'Emery curvature on a weighted graph can be formulated as a local matrix eigenvalue problem. We adapt this viewpoint to the Boolean hypercube by choosing canonical reversible transition rates and compatible gradient weights. A commutator identity (\Cref{lem:commutator}) then reduces the control of the curvature to bounding the minimum eigenvalue of explicit local matrices. Finally, such local matrices are further bounded entrywise by the Dobrushin influence matrix, and hence a Dobrushin-type condition suffices to certify a constant bound on the discrete curvature.

\subsection{Other related works}

\paragraph{Influence inequalities.}
Classical KKL- and Talagrand-type inequalities were extended from the uniform cube to biased and general product spaces, certain continuous product measures, and symmetric non-product spaces such as slices and Schreier graphs \cite{Talagrand94,FriedgutKalai96,BKKKL92,GrimmettJansonNorris16,KellerMosselSen12,CorderoErausquinLedoux12,ODonnellWimmer13,ODonnellWimmerSharp13}.
More recently, Hopkins~\cite{Hop26} developed a local-to-global approach to KKL-type theorems on simplicial complexes, with applications to high-dimensional expanders.
For dependent measures, monotonicity-based methods yielded sharp-threshold and phase-transition results for increasing events \cite{GrahamGrimmett06,DuminilCopinRaoufiTassion19}. The closest predecessor to our influence inequalities is \cite{KoehlerLifshitzMinzerMossel23}, which treats mixing bounded-degree Markov random fields. Our results instead apply to arbitrary measures satisfying a Dobrushin-type weak-dependence condition and additionally yield surface-area inequalities.

\paragraph{Surface-area inequalities.}
Talagrand introduced the square-root inequality for the one-sided boundary on biased cubes~\cite{Talagrand93}. Bobkov~\cite{Bob97} established a Gaussian-profile functional inequality on the uniform cube. Bobkov and G\"otze~\cite{BobkovGotze99} developed corresponding inequalities for biased cubes and general product spaces.
Eldan and Gross proved the influence-sensitive two-sided inequality on the uniform cube \cite{EldanGross22}. Subsequent work gave alternative proofs and refinements on the uniform cube, as well as extensions to biased cubes and diffuse product settings \cite{RvH-remarks,EldanKindlerLifshitzMinzer25,BIM23,IvanisviliZhang26}. More recently, the sharp one-sided inequality on the uniform cube was established in \cite{DurcikIvanisviliRoosXie26}. Ivanisvili, Xie, and Zhang strengthened the Eldan--Gross inequality on the uniform cube by replacing the surface area with a nonlocal Beckmann boundary \cite{IvanisviliXieZhang26}. These results retain product structure, whereas the present work allows weak dependence and non-product spaces.

\paragraph{Bakry--\'Emery theory.}
Our analytic framework belongs to the Bakry--\'Emery $\Gamma$-calculus \cite{BakryEmery85,BGL-book}; see \cite{Kamtue-thesis,Salez-survey} for discrete perspectives. Relevant developments on finite state space include local matrix formulations of Bakry--\'Emery curvature and adapted $\Gamma$-calculus \cite{CKLP22,KLMP23}. Related work studies weakly interacting chains through entropic curvature and derives discrete curvature bounds from coupling methods \cite{EHMT17,Pedrotti25}. The present paper uses Bakry--\'Emery curvature; for other transport-based notions of curvature, such as Ollivier curvature or Erbar--Maas entropic curvature, we refer to the works \cite{Ollivier09,ErbarMaas12,KLMP24}, the book \cite{NR-book}, and the thesis \cite{Kamtue-thesis}.

\section{Preliminaries}
\label{sec:preliminaries}
\subsection{\texorpdfstring{$\Gamma$}{Γ}-calculus on the Boolean hypercube}
Let $\Omega = \{\pm 1\}^n$ be the Boolean hypercube.
For any $x \in \Omega$, we use $x^i$ to denote the vector obtained by flipping the $i$-th coordinate of $x$ and $x^{ij} := (x^i)^j$.
Given the transition rates $\{q_i: \Omega \to \R_+\}_{i\in[n]}$, we define the operator $L$ on functions $f:\Omega \to \R$ as follows:
\begin{align*}
    \forall x \in \Omega,\quad Lf(x) = \sum_{i=1}^n q_i(x)(f(x^i) - f(x)).
\end{align*}
The \textit{continuous-time Markov semigroup} $(P_t)_{t\geq 0}$ is given by $P_t = \e^{t L}$.
Let $\mu$ be the \textit{stationary distribution} of the Markov semigroup $(P_t)_{t\geq 0}$.
Throughout the paper, we assume that $L$ is \textit{reversible} with respect to $\mu$, that is, for any $x\in \Omega$ and $i\in [n]$, it holds
$$\mu(x)q_i(x) = \mu(x^i)q_i(x^i).$$

\begin{definition}[Weighted carr\'e du champ operator $\Gamma_w$]
    Given the weight functions $w_i: \Omega \to \R_+$ for every $i\in [n]$, the \textit{weighted carr\'e du champ operator} $\Gamma_w$ is defined, for any functions $f,g:\Omega \to \R$, by
    \begin{align*}
        \Gamma_w(f,g) := \frac 12 (L_w(fg)-gL_wf-fL_wg),
    \end{align*}
    where $L_w$ is the operator defined by the reweighted transition rates $\{w_iq_i\}_{i\in[n]}$ (for which $\mu$ is not necessarily stationary). 
    We write $\Gamma_w(f) := \Gamma_w(f,f)$ for simplicity.
    By a straightforward calculation, for any $x \in \Omega$, it holds
    \begin{align*}
       \Gamma_w(f,g)(x) = \frac 12\sum_{i=1}^n w_i(x)q_i(x)(f(x^i) - f(x))(g(x^i) - g(x)).
    \end{align*}
    We also use $\Gamma_{w,i}$ to denote the contribution of the $i$-th coordinate to $\Gamma_w$, that is, 
    \begin{align*}
        \forall x \in \Omega,\quad \Gamma_{w,i}(f,g)(x) := \frac 12 w_i(x)q_i(x)(f(x^i) - f(x))(g(x^i) - g(x)).
    \end{align*}
    We write $\Gamma_{w,i}(f) := \Gamma_{w,i}(f,f)$ for simplicity.
    Note that $\Gamma_w(f) = \sum_{i=1}^n \Gamma_{w,i}(f)$.
\end{definition}

We omit the subscript $w$ when $w_i \equiv 1$ for all $i \in [n]$.

It is also convenient to use the weighted discrete gradient to represent the carr\'e du champ operator.

\begin{definition}[Weighted discrete gradient of $f$]\label{def:weighted-gradient}
    The \textit{weighted discrete gradient} of a function $f: \Omega \to \mathbb{R}$ is defined as the vector
    \begin{align*}
        \nabla_w f := (\delta_{w,1} f, \ldots, \delta_{w,n} f),\quad \text{where} \quad \delta_{w,i} f := \sqrt{\frac 12w_i(x)q_i(x)}(f(x^i) - f(x)).
    \end{align*}
    The \textit{absolute weighted discrete gradient} is the entrywise absolute value of $\nabla_w f$, denoted by $|\nabla_w| f := (|\delta_{w,1}| f, \ldots, |\delta_{w,n}| f)$.
\end{definition}

\begin{remark}
    We note that there are other common ways to define the discrete derivatives, such as
    \begin{align*}
        \delta_{w,i} f := \sqrt{\frac 12w_i(x)q_i(x)} \left( f(x^{i\gets +}) - f(x^{i \gets -}) \right)
        \qquad\text{or}\qquad
        \delta_{w,i} f := \sqrt{\frac 12w_i(x)q_i(x)} (f(x) - f(x^i)).
    \end{align*}
    Each convention can be used to derive the main results of the paper.
    We choose the current definition because, for the canonical weights $w_i=2q_i$ (see \Cref{ex:canonical}), the derivative $\delta_{w,i}f(x)=q_i(x)(f(x^i)-f(x))$ is the $i$-th coordinate contribution to the generator $L$. This is convenient for deriving curvature bounds later in \Cref{sec:GE}.
\end{remark}

\begin{fact}
    We have that $\Gamma_{w,i}(f,g) = (\delta_{w,i} f) (\delta_{w,i} g)$ and $\Gamma_w(f,g) = \inner{\nabla_w f}{\nabla_w g}$.
    In particular, $\Gamma_{w,i}(f) = (\delta_{w,i} f)^2$ and $\Gamma_w(f) = \norm{\nabla_w f}_2^2$.
\end{fact}

\begin{definition}[Iterated weighted carr\'e du champ operator $\Gamma_{2;w}$]
    The \textit{iterated weighted carr\'e du champ operator} $\Gamma_{2;w}$ is defined, for any function $f:\Omega \to \R$, by
    \begin{align*}
        \Gamma_{2;w}(f) := \frac 12 L\Gamma_w(f) - \Gamma_w(f,Lf).
    \end{align*}
\end{definition}

Discrete spaces like the Boolean hypercube lack the diffusion chain rule available in continuous spaces. However, it was observed in \cite{BHLLMY15,FS18} that a specific chain rule does exist, as stated in the lemma below. This will help in calculations later on.

\begin{lemma}[\cite{BHLLMY15,FS18}, Chain rule]
\label{lem:chain-rule}
    For any function $f: \Omega \to \R_{\ge 0}$, it holds
    \begin{align*}
        \sqrt{f} L \sqrt{f} = \frac{1}{2} L f - \Gamma(\sqrt{f}).
    \end{align*}
\end{lemma}

\begin{proof}
    By the definition of carr\'e du champ operator, 
    \begin{align*}
        \Gamma(\sqrt{f}) = \frac{1}{2} \tp{L(\sqrt{f}^2) - 2\sqrt{f} L \sqrt{f}} = \frac{1}{2} Lf - \sqrt{f} L \sqrt{f}.
    \end{align*}
    This completes the proof.
\end{proof}

\subsection{Influence and surface area of Boolean functions}

\begin{definition}
\label{def:inf-SA}
    Let $f: \Omega \to \mathbb{R}$ be a function.

    \begin{enumerate}[(1)]
        \item The \textit{total influence} of $f$ is defined as
        \begin{align*}
            I_w(f) := \E[\mu]{\Gamma_{w}(f)} = \E[\mu]{\norm{\nabla_w f}^2_2}.
        \end{align*}
        Also, let $I_{w,i}(f) := \E[\mu]{\Gamma_{w,i}(f)}$ denote the influence of the $i$-th coordinate.

        \item The \textit{(Boolean) surface area} of $f$ is defined as
        \begin{align*}
            A_w(f) := \E[\mu]{\sqrt{\Gamma_wf}} = \E[\mu]{\norm{\nabla_w f}_2}.
        \end{align*}

        \item The \textit{total squared influence} of $f$ is defined as
        \begin{align*}
            J_{w}(f) := \sum_{i=1}^n \tp{\E[\mu]{\abs{\delta_{w,i}} f}}^2 = \norm{\E[\mu]{\abs{\nabla_w} f}}^2_2.
        \end{align*}
    \end{enumerate}
    
\end{definition}

In \Cref{def:inf-SA}, influences and surface area are defined for real-valued functions. The following two examples explain how these notions relate to those defined for Boolean functions in \Cref{subsec:inf-SA}.

We first consider the uniform distribution with a standard jump process.

\begin{example}[Uniform distribution]
\label{ex:uniform}
    Consider the case where $\mu$ is uniform over $\Omega$. Let $q_i(x) = \frac{1}{2}$ and $w_i(x) = 1$ for all $i \in [n]$ and $x \in \Omega$.
    Let $f: \Omega \to \{\pm 1\}$ be a Boolean-valued function. 
    Below, we omit $w$ in the subscript since everything is unweighted.

    \begin{enumerate}[(1)]
        \item Total influence: We have
        \begin{align*}
            I_i(f) = \Pr[\mu]{f(x^i) \neq f(x)} = \mathsf{Inf}_{\mu,i}(f), \quad I(f) = \sum_{i=1}^n I_i(f) = \mathsf{Inf}_{\mu}(f).
        \end{align*}

        \item Surface area: We have
        \begin{align*}
            A(f) = \E[\mu]{\sqrt{s(f)}} = \SurfArea_\mu(f),
        \end{align*}
        where $s(f)$ denotes the sensitivity of $f$; namely, $s(f)(x) = |\{i \in [n]: f(x^i) \neq f(x)\}|$ for each $x \in \Omega$.
    
        \item Total squared influence: We have
        \begin{align*}
            J(f) = \sum_{i=1}^n \Pr[\mu]{f(x^i) \neq f(x)}^2 = \sum_{i=1}^n I_i(f)^2 = \mathsf{SqInf}_\mu(f).
        \end{align*}
    \end{enumerate}
\end{example}

We then consider an arbitrary distribution with a \emph{canonical} choice of transition rates and weights. For Boolean functions, equivalence with the corresponding definitions in \Cref{subsec:inf-SA} requires marginal boundedness of $\mu$.

\begin{example}[Canonical transition rates and weights]
\label{ex:canonical}
    Let $\mu$ be an arbitrary distribution fully supported on $\Omega$. 
    Consider the canonical transition rates and weights defined, for any $i \in [n]$ and $x \in \Omega$, by
    \begin{align*}
        q^\star_i(x) := \frac 12\sqrt{\frac{\mu(x^i)}{\mu(x)}}
        \qquad\text{and}\qquad
        w^\star_i(x) := 2q^\star_i(x) = \sqrt{\frac{\mu(x^i)}{\mu(x)}}.
    \end{align*}
    Let $f: \Omega \to \{\pm 1\}$ be a Boolean-valued function. 
    Below, we use the superscript $\star$ to represent the canonical weights $w^\star$.

    \begin{enumerate}[(1)]
        \item Total influence: We have
        \begin{align*}
            I^\star_i(f) = \Pr[\mu]{f(x^i) \neq f(x)} = \mathsf{Inf}_{\mu,i}(f), \quad I^\star(f) = \sum_{i=1}^n I^\star_i(f).
        \end{align*}
        The above equality is derived from definitions:
        \begin{align*}
            I^\star_i(f) &= \E[\mu]{q_i^\star(x)^2 (f(x^i) - f(x))^2}
            = \E[\mu]{\ind\{f(x^i)\neq f(x)\}\frac{\mu(x^i)}{\mu(x)}}
            \\&= \sum_{x^i\in \Omega} \ind\{f(x^i)\neq f(x^{ii})\}\mu(x^i)
            = \Pr[\mu]{f(x^i) \neq f(x)},
        \end{align*}
        where the third equality follows from the fact that $x^{ii}=x$ holds for $i\in[n]$ and $x\in \Omega$.

        \item Surface area: We have
        \begin{align*}
            A^\star(f) = \E[\mu]{\sqrt{s_\mu(f)}},
        \end{align*}
        where $s_\mu(f)$ denotes a weighted sensitivity of $f$ defined by
        \begin{align*}
            s_\mu(f)(x) = \sum_{i=1}^n \frac{\mu(x^i)}{\mu(x)} \mathbf{1}\{f(x^i) \neq f(x)\}.
        \end{align*}
        If $b \le \sqrt{\mu(x^i)/\mu(x)} \le 1/b$ for all $x \in \Omega$ for some $b \in (0,1]$, then 
        \begin{align*}
            bA^\star(f) \le \E[\mu]{\sqrt{s(f)}} = \SurfArea_\mu(f) \le \frac 1b A^\star(f).
        \end{align*}

        \item Total squared influence: We have
        \begin{align*}
            J^\star(f) = \sum_{i=1}^n \E[\mu]{\sqrt{\frac{\mu(x^i)}{\mu(x)}} \mathbf{1}\{f(x^i) \neq f(x)\}}^2.
        \end{align*}
        If $b \le \sqrt{\mu(x^i)/\mu(x)} \le 1/b$ for all $x \in \Omega$ for some $b \in (0,1]$, then 
        \begin{align*}
            b^2J^\star(f) \le \sum_{i=1}^n \Pr[\mu]{f(x^i) \neq f(x)}^2 = \mathsf{SqInf}_\mu(f) \le \frac 1{b^2}J^\star(f).
        \end{align*}
    \end{enumerate}
    Finally, we remark that under such canonical transition rates and weight functions, the weighted carr\'e du champ and the iterated weighted carr\'e du champ are a special case of the adapted Gamma calculus studied in \cite{KLMP23,KLMP24}.
\end{example}

By H\"older's and Jensen's inequalities, we obtain the following relations.
\begin{fact}
    For any function $f : \Omega \to \R$, we have that $J_w(f)\leq A_w(f)^2 \leq I_w(f)$.
\end{fact}

\subsection{Functional inequalities}
\begin{definition}
    The \textit{variance} of a function $f: \Omega \to \mathbb{R}$ is
    \begin{align*}
        \Var[\mu]{f} := \E[\mu]{f^2} - (\E[\mu]{f})^2.
    \end{align*}
    The \textit{entropy} of a function $f: \Omega \to \mathbb{R}_{\ge 0}$ is
    \begin{align*}
        \Ent[\mu]{f} := \E[\mu]{f \log f} - \E[\mu]{f} \log \E[\mu]{f}.
    \end{align*}
    Here we use the convention $0\log 0:=0$.
\end{definition}

\begin{definition}[Poincar\'e inequality]
    We say that the Poincar\'e inequality holds with constant $\lambda > 0$ if for any $f: \Omega \to \R$,
    \begin{align}\label{eq:Poincare}
        \Var[\mu]{f} \leq \frac 1\lambda \E[\mu]{\Gamma (f)}, \tag{Poincar\'e}
    \end{align}
    where $\Gamma(f)$ is the unweighted carr\'e du champ operator, i.e. $w_i \equiv 1$ for all $i\in [n]$.
    Or equivalently, $\Var[\mu]{P_t f}\leq e^{-2\lambda t}\Var[\mu]{f}$ holds.
\end{definition}

\begin{definition}[Log-Sobolev inequality]
    We say that the log-Sobolev inequality holds with constant $\rho_{\LSI} > 0$ if for any $f: \Omega \to \R$,
    \begin{align}\label{eq:LSI}
        \Ent[\mu]{f^2} \leq \frac 2{\rho_{\LSI}} \E[\mu]{\Gamma (f)}. \tag{LSI}
    \end{align}
\end{definition}
The log-Sobolev inequality is stronger than the Poincaré inequality.
\begin{fact}\label{fact:LSI_to_PI}
If the log-Sobolev inequality holds with constant $\rho_{\mathrm{LSI}}>0$, then the Poincaré inequality holds with constant $\lambda=\rho_{\mathrm{LSI}}$.
Consequently, for $f\in\R^\Omega$ and $t\geq 0$, $\Var[\mu]{P_t f}\leq e^{-2\rho_{\mathrm{LSI}}t}\Var[\mu]{f}$.
\end{fact}

\begin{definition}[Modified log-Sobolev inequality of gradient type]
    We say that the modified log-Sobolev inequality of gradient type holds with constant $\rho_{\MLSI} > 0$ if for any $f: \Omega \to \R_+$,
    \begin{align}\label{eq:mLSI-gradient}
        \Ent[\mu]{f} \leq \frac 2{\rho_{\MLSI}} \E[\mu]{\frac{\Gamma (f)}{f}}. \tag{mLSI-gradient}
    \end{align}
\end{definition}

\begin{theorem}[Hypercontractivity]\label{thm:hypercontractivity}
    If the log-Sobolev inequality holds with constant $\rho > 0$ and $1 < p \leq q < \infty$ satisfy $\frac{q-1}{p-1}\leq e^{2\rho t}$, then for all $f: \Omega \to \R$, we have $\norm{P_t f}_{q,\mu} \leq \norm{f}_{p,\mu}$.
\end{theorem}

\subsection{Gaussian isoperimetric profile}
\begin{definition}[Gaussian isoperimetric profile]\label{def:gaussian_isoperimetric_profile}
    We define the Gaussian isoperimetric profile $\+I: [0,1] \to \R_{\ge 0}$ by $\+I(x) = \varphi(\Phi^{-1}(x))$ for $x\in(0,1)$ and $\+I(0)=\+I(1)=0$, where $\varphi$ and $\Phi$ are the probability density function and cumulative distribution function of the standard Gaussian distribution, respectively.
\end{definition}
\begin{fact}[{\cite[Eq.~(3.9)]{BobkovGotze99}}]\label{fact:gaussian_isoperimetric_profile_ineq}
    For any $a\in (0,1)$ and $b\in [0,1]$, we have that
    $$
        \+I(a)(\+I(b)-\+I(a)-\+I'(a)(b-a))\geq -(b-a)^2 .
    $$
\end{fact}

\section{Bakry--\'Emery Conditions and Gradient Estimates}
\label{sec:BE-GE}

\subsection{BE theory, weak and strong gradient estimates}
For $f,g: \Omega \to \R$, we write $f\leq g$ for brevity if $f(x)\leq g(x)$ for all $x\in \Omega$.

We first define the Bakry--\'Emery conditions, also known as the curvature-dimension condition $\mathrm{CD}(\rho,\infty)$ or the Bakry--\'Emery curvature.

\begin{definition}[Bakry--\'Emery conditions]
    Let $\mu$ be a distribution supported on $\Omega = \{\pm 1\}^n$ and $\{q_i: \Omega \to \R_+\}_{i\in[n]}$ be transition rates reversible for $\mu$.
    Suppose $\{w_i:\Omega \to \R_+\}_{i\in[n]}$ is a collection of weight functions.
    \begin{itemize}
        \item We say that the \textit{Bakry--\'Emery condition} holds with constant $\rho_{\BE} \in \mathbb{R}$ if for any function $f: \Omega \to \mathbb{R}$, it holds
        \begin{align}\label{eq:BE}
            \Gamma_{2;w}(f) \geq \rho_{\BE} \cdot \Gamma_w(f); \tag{BE}
        \end{align}
        \item We say that the \textit{reinforced Bakry--\'Emery condition} holds with constant $\rho_{\RBE} \in \mathbb{R}$ if for any function $f: \Omega \to \R$, it holds
        \begin{align}\label{eq:rBE}
            \Gamma_{2;w}(f) - \Gamma(\sqrt{\Gamma_w(f)}) \ge \rho_{\RBE} \cdot \Gamma_w(f). \tag{rBE}
        \end{align}
    \end{itemize}
\end{definition}

The following comparison follows immediately from the non-negativity of $\Gamma$.
\begin{fact}\label{fact:rBE_impley_BE}
If~\ref{eq:rBE} holds with constant $\rho$, then~\ref{eq:BE} holds with the same constant $\rho$.
\end{fact}

\begin{definition}[Weak and strong gradient estimate]
    Let $\mu$ be a distribution supported on $\Omega = \{\pm 1\}^n$ and $\{q_i: \Omega \to \R_+\}_{i\in[n]}$ be transition rates reversible for $\mu$.
    Suppose $\{w_i:\Omega \to \R_+\}_{i\in[n]}$ is a collection of weight functions. 
    \begin{itemize}
        \item We say that the \emph{Weak Gradient Estimate (WGE)} holds with constant $\rho_{\WGE} \in \mathbb{R}$ if for any function $f:\Omega \to \R$ and $t\geq 0$,
        \begin{align}\label{eq:WGE}
            \Gamma_w(P_t f) \leq e^{-2\rho_{\WGE} t} P_t(\Gamma_w(f)) \quad\iff\quad \norm{\nabla_w P_t f}^2_2 \leq e^{-2\rho_{\WGE} t} P_t(\norm{\nabla_w f}^2_2); \tag{WGE}
        \end{align}
        \item We say that the \emph{Strong Gradient Estimate (SGE)} holds with constant $\rho_{\SGE} \in \mathbb{R}$ if for any function $f: \Omega \to \R$ and $t\geq 0$,
        \begin{align}\label{eq:SGE}
            \sqrt{\Gamma_w(P_t f)} \leq e^{-\rho_{\SGE} t} P_t  \sqrt{\Gamma_w(f)} \quad\iff\quad \norm{\nabla_w P_t f}_2 \leq e^{-\rho_{\SGE} t} P_t(\norm{\nabla_w f}_2). \tag{SGE}
        \end{align}
    \end{itemize}
\end{definition}

\begin{fact}\label{fact:SGE-to-WGE}
If~\ref{eq:SGE} holds with constant $\rho$, then~\ref{eq:WGE} holds with the same constant $\rho$.
\end{fact}

\begin{proof}
    By~\ref{eq:SGE}, it holds that for any $f : \Omega \to \mathbb{R}$ and $t \ge 0$,
    \begin{align*}
         \norm{\nabla_w P_t f}^2_2  \le \e^{-2\rho t} \tp{P_t \norm{\nabla_w f}_2}^2 \le \e^{-2\rho t}P_t \tp{\norm{\nabla_w f}^2_2},
    \end{align*}
    where the last inequality follows from the Cauchy--Schwarz inequality: $\tp{P_t g}^2 \le P_t (g^2)$ for any function $g : \Omega \to \mathbb{R}$.
\end{proof}

A fundamental result in Bakry--\'Emery theory is the equivalence between curvature conditions and the corresponding gradient estimates; see, e.g., \cite{Salez-survey,Kamtue-thesis} for the discrete setting.

\begin{proposition}\label{prop:equivalence_BE_GE}
    The following equivalences hold.
    \begin{itemize}
        \item The Bakry--\'Emery condition holds with constant $\rho$ if and only if the weak gradient estimate holds with the same constant.
        \item The reinforced Bakry--\'Emery condition holds with constant $\rho$ if and only if the strong gradient estimate holds with the same constant.
    \end{itemize}
\end{proposition}

The proof of \Cref{prop:equivalence_BE_GE} requires the following lemma.

\begin{lemma}\label{lem:gamma2-formula}
    For any function $f: \Omega \to \R$, it holds
    \begin{align*}
        \Gamma_{2;w}(f) - \Gamma(\sqrt{\Gamma_w(f)}) = \sqrt{\Gamma_w(f)} L \sqrt{\Gamma_w(f)} - \Gamma_w(f,Lf).
    \end{align*}
\end{lemma}

\begin{proof}
    A straightforward calculation shows that
    \begin{align*}
        \Gamma_{2;w}(f) - \Gamma(\sqrt{\Gamma_w(f)}) &= \frac 12 L \Gamma_w(f) - \Gamma_w(f,Lf) - \Gamma(\sqrt{\Gamma_w(f)}) &&\text{(by definition of $\Gamma_{2;w}$)}
        \\&= \sqrt{\Gamma_w(f)} L \sqrt{\Gamma_w(f)} - \Gamma_w(f,Lf). &&\text{(by \Cref{lem:chain-rule})} \qedhere
    \end{align*}
\end{proof}

\begin{proof}[Proof of \Cref{prop:equivalence_BE_GE}]
    Fix $f : \Omega \to \R$ and $t \ge 0$. Let $F_s = P_s \Gamma_w(P_{t-s} f)$. The derivative $\frac{\dif }{\dif s} F_s$ satisfies
    \begin{align*}
        \frac{\dif }{\dif s} F_s = L P_s \Gamma_w(P_{t-s} f) + P_s \frac{\dif}{\dif s} \Gamma_w(P_{t-s} f) &= L P_s \Gamma_w(P_{t-s} f) -2 P_s \Gamma_w(P_{t-s}f, LP_{t-s} f)\\
        &= 2 P_s \Gamma_{2;w}(P_{t-s} f).
    \end{align*}
    By the Bakry--\'Emery condition, the derivative satisfies $\frac{\dif}{\dif s} F_s \ge 2 \rho F_s$
    implying
    \begin{align*}
        P_t (\Gamma_w (f)) = F_t \ge \e^{2\rho t} F_0 = \e^{2\rho t} \Gamma_w(P_t f),
    \end{align*}
    which is exactly~\ref{eq:WGE} with constant $\rho$. Conversely, by~\ref{eq:WGE}, $\e^{2\rho t}\Gamma_w(P_t f) - P_t(\Gamma_w(f)) \le 0$ for all $t \ge 0$, which implies
    \begin{align*}
        \left.\frac{\dif}{\dif t}
        \tp{\e^{2 \rho t} \Gamma_w(P_t f) - P_t(\Gamma_w(f))}
        \right|_{t = 0}
        &= 2\rho \Gamma_w(f) + 2 \Gamma_w(f,Lf) - L(\Gamma_w(f))\\
        &=2 \rho \Gamma_w(f) - 2 \Gamma_{2;w}(f) \le 0.
    \end{align*}
    This proves the first part.

    Similarly, let $G_s = P_s \sqrt{\Gamma_w(P_{t-s} f)}$. We first assume that $\Gamma_w(P_{t-s}f)>0$ for all $s\in[0,t]$. The derivative satisfies
    \begin{align}\label{eq:dG-1}
    \frac{\dif}{\dif s} G_s = P_s \tp{L \sqrt{\Gamma_w(P_{t-s} f)} - \frac{\Gamma_w(P_{t-s} f, L P_{t-s} f)}{\sqrt{\Gamma_w(P_{t-s} f)}}}.
    \end{align} 
    By the definition of the iterated carr\'e du champ operator $\Gamma_{2;w}$, it holds
    \begin{align}\label{eq:dG-2}
        \Gamma_w(P_{t-s} f, L P_{t-s} f) = \frac{1}{2} L \Gamma_w(P_{t-s} f) - \Gamma_{2;w}(P_{t-s} f). 
    \end{align}
    By~\Cref{lem:chain-rule}, the quantity $\sqrt{\Gamma_w(P_{t-s} f)} L\sqrt{\Gamma_w(P_{t-s} f)}$ can be written as
    \begin{align}\label{eq:dG-3}
        \sqrt{\Gamma_w(P_{t-s} f)} L\sqrt{\Gamma_w(P_{t-s} f)} = \frac{1}{2} L \Gamma_w(P_{t-s} f) - \Gamma(\sqrt{\Gamma_w(P_{t-s} f)}).
    \end{align}
    Combining~\eqref{eq:dG-1},~\eqref{eq:dG-2} and~\eqref{eq:dG-3}, we have
    \begin{align*}
    \frac{\dif}{\dif s} G_s = P_s\tp{\frac{\Gamma_{2;w}(P_{t-s} f) - \Gamma(\sqrt{\Gamma_w(P_{t-s} f)})}{\sqrt{\Gamma_w(P_{t-s} f)}}} \geq \rho G_s.
    \end{align*}
    This shows that~\ref{eq:SGE} can be obtained by~\ref{eq:rBE}. For the general case, let $G_s^\varepsilon=P_s\sqrt{\Gamma_w(P_{t-s}f)+\varepsilon}$. The same calculation gives $\frac{\dif}{\dif s}G_s^\varepsilon\geq \rho G_s^\varepsilon-\max\{\rho,0\}\sqrt{\varepsilon}$. Integrating and letting $\varepsilon\downarrow0$ again yields~\ref{eq:SGE}. Conversely, differentiating $\e^{\rho t} \sqrt{\Gamma_w(P_t f)} - P_t \sqrt{\Gamma_w(f)}$ at $t=0$ gives~\ref{eq:rBE} wherever $\Gamma_w(f)>0$; at points where $\Gamma_w(f)=0$, it follows directly from \Cref{lem:gamma2-formula}. This completes the proof.
\end{proof}

To motivate the use of Bakry--\'Emery conditions and gradient estimates, we note that they imply contraction of the total influence and surface area when the associated curvature is non-negative.

\begin{proposition}[Contraction in total influence and surface area]
\label{prop:GE_to_IA_contraction}
    \quad
    \begin{itemize}
        \item If~\ref{eq:WGE} holds with constant $\rho \in \mathbb{R}$, then the total influence is controlled along the heat semigroup by the factor $\e^{-2\rho t}$; that is, for any function $f: \Omega \to \R$ and $t \ge 0$,
        \begin{align}\label{eq:I-contraction}
            I_w(P_t f) \leq e^{-2\rho t} I_w(f); \tag{$I$-contraction}
        \end{align}
        
        \item If~\ref{eq:SGE} holds with constant $\rho \in \mathbb{R}$, then the surface area is controlled along the heat semigroup by the factor $\e^{-\rho t}$; that is, for any function $f: \Omega \to \R$ and $t \ge 0$,
        \begin{align}\label{eq:A-contraction}
            A_w(P_t f) \leq e^{-\rho t} A_w(f). \tag{$A$-contraction}
        \end{align}
    \end{itemize}
\end{proposition}

\begin{proof}
By the definition of the total influence, it holds
\begin{align*}
    I_w(P_t f) = \E[\mu]{\Gamma_w(P_t f)} \le \e^{-2 \rho t} \E[\mu]{P_t \Gamma_w(f)} = \e^{-2 \rho t} \E[\mu]{\Gamma_w(f)}.
\end{align*}
This proves the contraction of the total influence. 
Similarly, by the definition of the surface area, 
\begin{align*}
    A_w(P_t f) = \E[\mu]{\sqrt{\Gamma_w(P_t f)}} \le \e^{-\rho t} \E[\mu]{P_t \sqrt{\Gamma_w(f)}} = \e^{-\rho t} A_w(f).
\end{align*}
This proves the contraction of the surface area.
\end{proof}

\subsection{Coordinate-wise strong gradient estimate}
Suppose $v: \Omega \to \R^n$ is a vector field; for example, given any function $f: \Omega \to \R$, the discrete gradient $\nabla_w f$ and the absolute discrete gradient $\abs{\nabla_w}f$ are both vector fields.
Suppose $T$ is an operator mapping a function $f: \Omega \to \R$ to another function $Tf: \Omega \to \R$.
With a slight abuse of notation, we extend $T$ to apply to vector fields entrywise; specifically, $T v: \Omega \to \R^n$ is again a vector field defined as
\begin{align*}
    T v := (T v_1,\dots,T v_n)
\end{align*}
where $v_1,\dots,v_n$ are components of $v$.

Inspired by \cite{CorderoErausquinLedoux12}, we introduce a coordinate-wise version of the strong gradient estimate.

\begin{definition}[Coordinate-wise strong gradient estimate]
\label{def:CSGE}
    Let $\mu$ be a distribution supported on $\Omega = \{\pm 1\}^n$ and $\{q_i: \Omega \to \R_+\}_{i\in[n]}$ be transition rates reversible for $\mu$.
    Suppose $\{w_i:\Omega \to \R_+\}_{i\in[n]}$ is a collection of weight functions. 
    
    We say that the \emph{Coordinate-wise Strong Gradient Estimate (CSGE)} holds with constant $\rho_{\CSGE} \in \mathbb{R}$ if for any $\tau \in [0,\infty]$, any function $f: \Omega \to \R$, and $t\geq 0$,
        \begin{align}\label{eq:CSGE}
            \norm{P_\tau \abs{\nabla_w}P_t f}_2 \leq e^{-\rho_{\CSGE} t}\norm{P_{\tau+t} \abs{\nabla_w}f}_2. \tag{CSGE}
        \end{align}
        We record two particularly important subcases:
        \begin{itemize}
            \item ($\tau = 0$) For any function $f: \Omega \to \R$ and $t\geq 0$,  
            \begin{align}\label{eq:0-CSGE}
                \norm{\nabla_w P_t f}_2 \leq e^{-\rho_{\CSGE} t}\norm{P_t \abs{\nabla_w}f}_2. \tag{$0$-CSGE}
            \end{align}

            \item ($\tau = \infty$) For any function $f: \Omega \to \R$ and $t\geq 0$,  
            \begin{align}\label{eq:infty-CSGE}
                \norm{\E[\mu]{\abs{\nabla_w}P_t f}}_2 \leq e^{-\rho_{\CSGE} t}\norm{\E[\mu]{\abs{\nabla_w}f}}_2. \tag{$\infty$-CSGE}
            \end{align}
        \end{itemize}
\end{definition}

\begin{remark}
\label{rmk:CSGE-CL12}
    We note that \cite{CorderoErausquinLedoux12} implicitly used a different version of \ref{eq:0-CSGE} in their proof of the $L^1$--$L^2$ inequality:
    \begin{align*}
        \norm{\nabla P_t f}_2 \leq e^{-\rho t}\norm{P_t \nabla f}_2.
    \end{align*}
    Our modification here includes adding the weight $w$ (necessary to bound the curvature) and taking the absolute value (needed for the Eldan--Gross inequality with $\tau = \infty$). For the $L^1$--$L^2$ inequality alone, \ref{eq:0-CSGE} without the absolute value is sufficient, as done in \cite{CorderoErausquinLedoux12}.
\end{remark}

\begin{fact}\label{fact:CSGE_to_SGE}
If~\ref{eq:0-CSGE} holds with constant $\rho \in \mathbb{R}$, then both~\ref{eq:WGE} and~\ref{eq:SGE} hold with the same constant $\rho$.
\end{fact}

\begin{proof}
    By~\ref{eq:0-CSGE} and Minkowski's inequality,
    \begin{align*}
        \norm{\nabla_w P_t f}_2 \le \e^{-\rho t} \norm{P_t \abs{\nabla_w} f}_2 \le \e^{-\rho t} P_t \norm{\nabla_w f}_2.
    \end{align*}
    Thus,~\ref{eq:SGE} holds with the constant $\rho$.
    By~\Cref{fact:SGE-to-WGE},~\ref{eq:WGE} also holds with the same constant $\rho$.
\end{proof}

Analogous to \Cref{prop:GE_to_IA_contraction}, \ref{eq:CSGE} implies the contraction in the total squared influence.

\begin{proposition}[Contraction in total squared influence]\label{prop:CSGE-to-J-contraction}
    If~\ref{eq:infty-CSGE} holds with constant $\rho$, then the total squared influence satisfies the following estimate: for any function $f: \Omega \to \R$ and $t \ge 0$,
    \begin{align}\label{eq:J-contraction}
        J_w(P_t f) \leq e^{-2\rho t} J_w(f). \tag{$J$-contraction}
    \end{align} 
\end{proposition}

\begin{proof}
    Assume~\ref{eq:infty-CSGE} holds with constant $\rho$. For any $f:\Omega \to \mathbb{R}$ and $t \ge 0$, it holds
    \begin{align*}
        J_w(P_t f) = \norm{\E[\mu]{\abs{\nabla_w} P_t f}}_2^2 \le \e^{-2\rho t} \norm{\E[\mu]{\abs{\nabla_w} f}}_2^2 = \e^{-2 \rho t} J_w(f).
    \end{align*}
\end{proof}

\subsection{Functional inequalities from gradient estimates}
\begin{definition}
    For $b \in (0,1]$, we say the weight functions $\{w_i:\Omega \to \R_+\}_{i\in[n]}$ are $b$-bounded if for all $i \in [n]$ and $x \in \Omega$, it holds
    \begin{align}\label{eq:w_lower_bound}
        b \le w_i(x) \le \frac{1}{b}.
    \end{align}
\end{definition}

\begin{proposition}[Gradient estimate and functional inequalities]
\label{prop:GE-to-FI}
    Let $\mu$ be a distribution supported on $\Omega = \{\pm 1\}^n$ and $\{q_i: \Omega \to \R_+\}_{i\in[n]}$ be transition rates reversible for $\mu$.
    Suppose $\{w_i:\Omega \to \R_+\}_{i\in[n]}$ is a collection of weight functions 
    that are $b$-bounded for some $b \in (0,1]$.
    \begin{itemize}
        \item If~\ref{eq:SGE} holds with $\rho > 0$, then~\ref{eq:mLSI-gradient} holds with constant $4b^2\rho$;
        \item If~\ref{eq:WGE} holds with $\rho > 0$, then~\ref{eq:Poincare} holds with constant $b^2\rho$.
    \end{itemize}
\end{proposition}

\begin{proof}
    By a standard calculation,
    \begin{align*}
    \Ent[\mu]{f} = \int_0^{+\infty} \E[\mu]{\Gamma(P_t f, \log P_t f)} \dif t  \overset{(\star)}{\le}  \int_0^{+\infty} \E[\mu]{\frac{\Gamma(P_t f)}{P_t f}} \dif t \le \frac{1}{b} \cdot \int_0^{+\infty} \E[\mu]{\frac{\Gamma_w(P_t f)}{P_t f}} \dif t,
    \end{align*}
    where the inequality $(\star)$ follows from $(s-t)\tp{\log s -\log t} \le \frac{(s-t)^2}{2}\tp{\frac{1}{s} + \frac{1}{t}}$. By~\ref{eq:SGE},
    \begin{align*}
        \Gamma_w(P_t f) \le \e^{-2\rho t} \tp{P_t \sqrt{\Gamma_w(f)}}^2 \le \e^{-2\rho t} \tp{P_t \tp{\frac{\Gamma_w (f)}{f}}} \tp{P_t f},
    \end{align*}
    where the last inequality follows from the Cauchy--Schwarz inequality. Hence,
    \begin{align*}
    \Ent[\mu]{f} \le \frac{1}{b} \int_0^{+\infty} \e^{-2 \rho t} \E[\mu]{\frac{\Gamma_w (f)}{f}} \dif t \le \frac{1}{2 b^2 \rho} \cdot \E[\mu]{\frac{\Gamma(f)}{f}}.
    \end{align*}
    This proves the first part. For the second part, we use a similar proof strategy. Specifically,
    \begin{align*}
    \Var[\mu]{f} = 2\int_0^{+\infty} \E[\mu]{\Gamma(P_t f)} \dif t &\le \frac{2}{b} \cdot \int_0^{+\infty} \E[\mu]{\Gamma_w(P_t f)} \dif t \\
    (\text{By~\ref{eq:WGE}})\quad & \le \frac{2}{b} \cdot \int_0^{+\infty} \e^{-2\rho t} \E[\mu]{\Gamma_w(f)} \dif t \le \frac{1}{b^2\rho} \cdot \E[\mu]{\Gamma(f)}.
    \end{align*}
    This establishes the desired Poincar\'e inequality.
\end{proof}
\Cref{fig} summarizes the implications established in this section among Bakry--\'Emery conditions, gradient estimates, contraction properties of influences, and functional inequalities.

\begin{figure}[!thbp]
\centering
\begin{tikzpicture}[
block/.style = {rectangle, rounded corners, minimum width=2cm, minimum height=1cm, text centered, text width=1cm, draw=black, fill=orange!23},
arrow/.style = {thick,->,>=stealth}
]

\node[block, text width=2.5cm, fill=green!18] (PI) at (-0.5,0) {\ref{eq:Poincare}};
\node[block] (BE) at (4,0) {\ref{eq:BE}};
\node[block] (WGE) at (8,0) {\ref{eq:WGE}};
\node[block, text width=2.5cm, fill=violet!15] (Ic) at (12.5,0) {\ref{eq:I-contraction}};

\node[block, text width=2.5cm, fill=green!18] (mLSI) at (-0.5,-2) {\ref{eq:mLSI-gradient}};
\node[block] (rBE) at (4,-2) {\ref{eq:rBE}};
\node[block] (SGE) at (8,-2) {\ref{eq:SGE}};
\node[block, text width=2.5cm, fill=violet!15] (Ac) at (12.5,-2) {\ref{eq:A-contraction}};

\node[block] (CSGE) at (8,-4) {\ref{eq:CSGE}};
\node[block, text width=2.5cm, fill=violet!15] (Jc) at (12.5,-4) {\ref{eq:J-contraction}};

\draw[arrow] (BE) -- (PI)
node[midway, above] {\hyperref[prop:GE-to-FI]{Prop.~\ref*{prop:GE-to-FI}}}
node[midway, below] {$\rho>0$};

\draw[thick,<->,>=stealth] (BE) -- (WGE)
node[midway, above] {\hyperref[prop:equivalence_BE_GE]{Prop.~\ref*{prop:equivalence_BE_GE}}};

\draw[arrow] (WGE) -- (Ic)
node[midway, above] {\hyperref[prop:GE_to_IA_contraction]{Prop.~\ref*{prop:GE_to_IA_contraction}}};

\draw[arrow] (rBE) -- (mLSI)
node[midway, above] {\hyperref[prop:GE-to-FI]{Prop.~\ref*{prop:GE-to-FI}}}
node[midway, below] {$\rho>0$};

\draw[thick,<->,>=stealth] (rBE) -- (SGE)
node[midway, above] {\hyperref[prop:equivalence_BE_GE]{Prop.~\ref*{prop:equivalence_BE_GE}}};

\draw[arrow] (SGE) -- (Ac)
node[midway, above] {\hyperref[prop:GE_to_IA_contraction]{Prop.~\ref*{prop:GE_to_IA_contraction}}};

\draw[arrow] (CSGE) -- (Jc)
node[midway, above] {\hyperref[prop:CSGE-to-J-contraction]{Prop.~\ref*{prop:CSGE-to-J-contraction}}}
node[midway, below] {$\tau=\infty$};

\draw[arrow] (rBE) -- (BE)
node[midway, right] {\cref{fact:rBE_impley_BE}};

\draw[arrow] (SGE) -- (WGE)
node[midway, right] {\cref{fact:SGE-to-WGE}};

\draw[arrow] (CSGE) -- (SGE)
node[midway, right] {\cref{fact:CSGE_to_SGE}}
node[midway, left] {$\tau=0$};

\end{tikzpicture}

\caption{Implications among Bakry--\'Emery conditions, gradient estimates, influence contractions, and functional inequalities.}
\label{fig}
\end{figure}

\subsection{Comparison and discussion}
\label{sec:different-rates-weights}

In the classical diffusion setting with $w\equiv1$, the Bakry--\'Emery condition, also called the curvature-dimension condition $\mathrm{CD}(\rho,\infty)$, is equivalent to the reinforced Bakry--\'Emery condition with the same constant $\rho$, thanks to the diffusion chain rule \cite{BGL-book}. This equivalence does not hold for discrete Markov semigroups. Moreover, even for fixed transition rates, different choices of weight functions can lead to different constants in the gradient estimates.

We illustrate these distinctions using the ferromagnetic Ising model on a single edge. Let
\[
    \Omega=\{\pm1\}^2,
    \qquad
    \mu_\beta(x)=\frac{1}{Z_\beta}\exp(\beta x_1x_2),
    \qquad x\in\Omega,
\]
where $\beta\in(0,1)$ is the inverse temperature and
\[
    Z_\beta
    =\sum_{x\in\Omega}\exp(\beta x_1x_2)
    =2e^\beta+2e^{-\beta}
    =4\cosh\beta.
\]
For each $i\in[2]$ and $x\in\Omega$, the canonical transition rates are
\[
    q_i^\star(x)
    =\frac12\sqrt{\frac{\mu_\beta(x^i)}{\mu_\beta(x)}}
    =\frac12e^{-\beta x_1x_2},
\]
and the canonical weights are
\[
    w_i^\star(x)
    =2q_i^\star(x)
    =e^{-\beta x_1x_2}.
\]
We fix these transition rates and compare the unweighted choice $w\equiv1$ with the canonical choice $w=w^\star$.

Throughout this example, $\rho$ and $\rho^\star$ denote the optimal constants for the unweighted and canonically weighted conditions, respectively, with subscripts indicating the relevant condition. We write $-\infty$ when the condition fails for every finite constant. Set
\[
    \beta_c:=\frac12\log2.
\]
The optimal constants are summarized in \Cref{tab:one-edge-Ising} and illustrated in \Cref{fig:One-edge-Ising}.

\begin{table}[th]
    \centering
    \small
    \renewcommand{\arraystretch}{1.8}
    \begin{tabular}{c|c|c}
        \hline
        Condition
        & Unweighted constant $\rho$
        & Canonical constant $\rho^\star$ \\
        \hline
        BE / WGE
        & $\displaystyle \frac12(3e^{-\beta}-e^\beta)$
        & $\displaystyle e^{-\beta}+\frac12e^{-3\beta}-\frac12e^\beta$ \\[3pt]
        rBE / SGE
        & $\displaystyle
            \begin{cases}
                1+e^{-\beta}-e^\beta,
                    & 0<\beta\leq\beta_c,\\
                -\infty,
                    & \beta_c<\beta<1
            \end{cases}$
        & $\displaystyle 2e^{-\beta}-e^\beta$ \\[3pt]
        CSGE
        & $-\infty$
        & $\displaystyle 2e^{-\beta}-e^\beta$ \\
        \hline
    \end{tabular}
    \caption{Optimal curvature and gradient-estimate constants for the one-edge Ising model with canonical transition rates $q^\star$, under the unweighted choice $w\equiv1$ and the canonical choice $w=w^\star$.}
    \label{tab:one-edge-Ising}
\end{table}

\begin{figure}[th]
    \centering
    \includegraphics[width=0.9\textwidth]{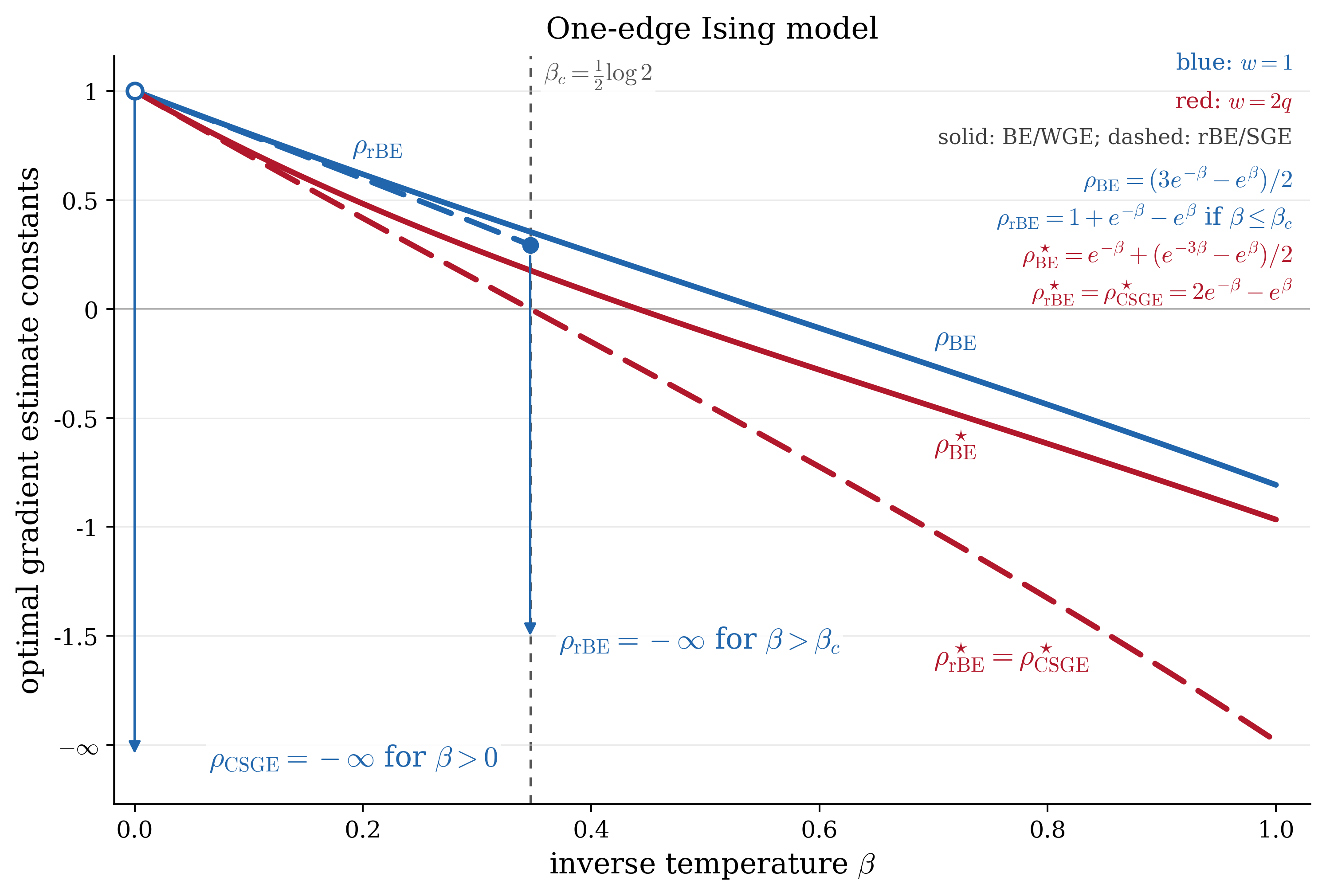}
    \caption{Optimal curvature and gradient-estimate constants for the one-edge Ising model with canonical transition rates. Blue curves correspond to the unweighted choice $w\equiv1$, and red curves correspond to the canonical weights $w=w^\star$. Solid curves show the BE/WGE constants, and dashed curves show the rBE/SGE constants.}
    \label{fig:One-edge-Ising}
\end{figure}

The proofs of the constants in \Cref{tab:one-edge-Ising} are included in \Cref{app:one-edge-Ising}.

For both choices of weights, the reinforced Bakry--\'Emery condition has a strictly smaller optimal constant than the Bakry--\'Emery condition. In the unweighted case, the weak, strong, and coordinate-wise strong gradient estimates have distinct optimal constants for $0<\beta\leq\beta_c$: the weak and strong gradient estimates have different finite optimal constants, whereas the coordinate-wise strong gradient estimate fails for every finite constant. When $\beta>\beta_c$, the strong gradient estimate also fails for every finite constant.

For the canonically weighted carr\'e du champ operator, the strong and coordinate-wise strong gradient estimates have the same optimal constant,
\[
    \rho_\RBE^\star
    =\rho_\SGE^\star
    =\rho_\CSGE^\star
    =2e^{-\beta}-e^\beta.
\]
This value coincides with the bound supplied by \Cref{thm:CSGE_boolean_hypercube_w=q}.

We further provide numerical evidence that the strong gradient estimate and the coordinate-wise strong gradient estimate can be strictly separated in the canonically weighted setting, when the optimal coordinate-wise strong gradient estimate constant is finite and positive.
Specifically, consider the following one-edge Ising model with external fields (this example was found with the help of ChatGPT-6 Astra, and the estimates below were obtained numerically):
$$
    \nu(x)\propto\exp\tp{\frac 12\log \frac 98x_1x_2 + \frac 12\log2(x_2-x_1)}.
$$
Numerical experiments suggest that
$$
    \rho_{\CSGE}^\star(\nu) = \rho_{0\text{-}\CSGE}^\star(\nu) = \frac 12 < \frac {33}{64} < \rho^\star_\SGE(\nu).
$$
Here $\rho_{0\text{-}\CSGE}^\star(\nu)$ denotes the optimal constant in the coordinate-wise strong gradient estimate at $\tau = 0$.
These results provide numerical evidence that the strong gradient estimate is strictly weaker than the coordinate-wise strong gradient estimate in the canonically weighted setting.
Our sufficient condition for the coordinate-wise strong gradient estimate is sharp for the zero-field model above; the numerical results suggest that it is also sharp for this external-field example.

Finally, consider the uniform distribution on $\Omega=\{\pm1\}^n$ from \Cref{ex:uniform}. With transition rates $q_i(x)=1/2$ and weights $w_i(x)=1$ for all $i\in[n]$ and $x\in\Omega$, the weak, strong, and coordinate-wise strong gradient estimates all hold with constant $\rho=1$. The sharp constants in the Poincar\'e and log-Sobolev inequalities are also equal to $1$. Thus, under this normalization, the standard random walk on the uniform Boolean hypercube has the same curvature and functional-inequality constants as the standard Gaussian Ornstein--Uhlenbeck semigroup.

\section{Isoperimetric Inequalities on the Boolean Hypercube}
\label{sec:boolean}
\begin{figure}[t]
  \centering
  \resizebox{0.8\linewidth}{!}{%
  \begin{tikzpicture}[
    block/.style={
      rectangle,
      rounded corners=3pt,
      draw=black!55,
      line width=0.6pt,
      fill=orange!15,
      text width=3.1cm,
      minimum height=1cm,
      inner sep=5pt,
      align=center,
      font=\small
    },
    assumption/.style={
      block,
      fill=green!12,
      text width=3.3cm,
    },
    result/.style={
      block,
      fill=violet!15,
      draw=violet!65!black,
      text width=3.6cm,
      font=\small\bfseries
    },
    arrow/.style={
      line width=0.8pt,
      -{Stealth[length=2.3mm]},
      draw=black!70
    }
  ]

  \node[assumption] (CSGE) at (-4,0)
    {Coordinate-wise strong gradient estimate~\eqref{eq:CSGE}};
  \node[assumption] (LSI) at (0,0)
    {Hypercontractivity};
  \node[assumption] (RBE) at (4,0)
    {Reinforced\\Bakry--\'Emery~\eqref{eq:rBE}};

  \node[result] (L1L2) at (-4,-1.7)
    {Talagrand's $L^1$--$L^2$\\inequality~\eqref{eq:l1_l2_talagrand_2}};
  \node[result] (Bobkov) at (4,-1.7)
    {Local Bobkov\\inequality~\eqref{eq:local_bobkov}};
  \node[result] (Talagrand) at (4,-3.4)
    {Talagrand's variance--surface-area inequality~\eqref{eq:improved_talagrand}};
  
  \node[result] (KKL) at (-4,-3.4)
    {KKL inequality~\eqref{eq:KKL_boolean_2}};

  \node[result] (EG) at (0,-5)
    {Eldan--Gross inequality~\eqref{eq:eldan_gross}};

  \draw[arrow] (CSGE.south) -- (L1L2.north);
  \draw[arrow] (LSI.south west) -- (L1L2.north east);

  \draw[arrow] (RBE.south) -- (Bobkov.north);
  \draw[arrow] (Bobkov.south) -- (Talagrand.north);
  \draw[arrow] (LSI.south) -- (Talagrand.west);

  \draw[arrow] (Talagrand.south) -- (EG.east);

  \draw[arrow,rounded corners=5pt]
    (CSGE.west) -- (-6.3,0) |- (EG.west);
  
  \draw[arrow] (L1L2.south) -- (KKL.north);

  \draw[arrow] (L1L2.east)
    to[out=0,in=110] (EG.north);

  \end{tikzpicture}%
  }
  \caption{Overview of the implications and inequalities established in \Cref{sec:boolean}.}
  \label{fig:iso_ineq_implication}
\end{figure}

In this section, we derive isoperimetric inequalities from curvature conditions.
As illustrated in~\Cref{sec:curvature}, we first combine curvature estimates with the log-Sobolev inequality to establish the local Bobkov inequality and Talagrand-type inequalities using ideas in~\cite{CorderoErausquinLedoux12,Talagrand97}, and then deduce the Eldan--Gross inequality from these results following the high-level strategy in~\cite{IvanisviliZhang26}.
The implications among these concepts are illustrated in \Cref{fig:iso_ineq_implication}.
The central idea is that curvature conditions provide a systematic way to control the dependence among coordinates.

We note that for the isoperimetric inequalities discussed in this section, the weight functions $\{w_i\}_{i \in [n]}$ are not necessary for establishing the results. In fact, the results and proofs would be simpler and more elegant without weights. The only reason to consider the weighted version is that we need weights to bound the curvature in \Cref{sec:GE}.

\subsection{Local Bobkov and Talagrand's variance--surface-area inequalities}
The Gaussian-profile isoperimetric inequality for the uniform measure on the Boolean cube was established in~\cite{Bob97}. Specifically, for any function $f:\cube \to \{0,1\}$,
\begin{align*}
    \+I(\E[\mu]{f}) \le \frac12\SurfArea_\mu(2f-1) = \E[\mu]{\sqrt{\Gamma(f)}},
\end{align*}
where $\mu$ denotes the uniform measure on $\cube$ and $\+I$ is the Gaussian isoperimetric profile (\Cref{def:gaussian_isoperimetric_profile}).

The local Bobkov inequality introduced in~\cite{BakryLedoux96} is a semigroup refinement of this estimate in the diffusion setting: it controls $\+I(P_t f)$ pointwise in terms of both $\+I(f)$ and the initial gradient.
Building on the semigroup approach in~\cite[Theorem~8.5.3]{BGL-book}, we establish below a local Bobkov inequality for general measures on the Boolean hypercube under the strong gradient estimate, without requiring independence among the coordinates.
\begin{theorem}[Local Bobkov inequality, \cite{BakryLedoux96}]\label{lem:local_bobkov}
    Let $\mu$ be a distribution supported on $\Omega = \{\pm 1\}^n$ and $\{q_i: \Omega \to \R_+\}_{i\in[n]}$ be transition rates reversible for $\mu$.
    Assume $\{w_i:\Omega \to \R_+\}_{i\in[n]}$ is a collection of weight functions that are $b$-bounded for some $b \in (0,1]$.

    If the reinforced Bakry--\'Emery condition (equivalently, the strong gradient estimate) holds with constant $\rho \in \mathbb{R}$,
    then, for any $f:\Omega \to [0,1]$ and $t\geq 0$, it holds
    \begin{align}\label{eq:local_bobkov}
        \+I(P_t f)\leq P_t\tp{\sqrt{\+I(f)^2 + \gamma(t) \Gamma_w(f)}},
    \end{align}
    where $\gamma(t) =
    \begin{cases}
        \frac{2(1-e^{-2\rho t})}{b \rho}, & \rho \neq 0\\
        \frac{4t}{b}, & \rho = 0
    \end{cases}$ and $\+I:[0,1] \to \R_{\ge 0}$ denotes the Gaussian isoperimetric profile.
\end{theorem}

\begin{proof}
    The assertion is trivial for constant functions, so we assume throughout that \(f\) is nonconstant.
    Fix $f : \Omega \to [0,1]$ and $t \ge 0$.
    By applying the argument to $f_\varepsilon:=\varepsilon+(1-2\varepsilon)f$, where $0<\varepsilon<1/2$, and then letting $\varepsilon\downarrow0$, we may assume that $f$ takes values in $(0,1)$; in particular, $F>0$ below.
    Let $\Psi(s) = P_s \tp{\sqrt{\+I(P_{t-s} f)^2 + \gamma(s) \Gamma_w(P_{t-s} f)}}$. It suffices to prove that $\Psi$ is monotone increasing. By a standard calculation, it holds
    \begin{align}\label{eq:dPsi-1}
    \frac{\dif}{\dif s} \Psi(s) &= P_s L\tp{\sqrt{\+I(P_{t-s} f)^2 + \gamma(s) \Gamma_w(P_{t-s} f)}} + P_s \frac{\dif}{\dif s}\sqrt{\+I(P_{t-s} f)^2 + \gamma(s) \Gamma_w(P_{t-s} f)} .
    \end{align}
    Let $F = \+I(P_{t-s} f)^2 + \gamma(s) \Gamma_w(P_{t-s} f)$. By chain rule, it holds
    \begin{align}\label{eq:dPsi-2}
    \frac{\dif}{\dif s} \sqrt{F} &= \frac{1}{2\sqrt{F}}\cdot\tp{- 2 \+I(P_{t-s} f) \+I'(P_{t-s} f) \cdot P_{t-s} L f + \gamma'(s) \Gamma_w(P_{t-s} f) - 2\gamma(s) \Gamma_w(P_{t-s} f, L P_{t-s} f)} .
    \end{align}
    By~\Cref{lem:chain-rule}, it holds
    \begin{align}\label{eq:dPsi-3}
        2\sqrt{F} L \sqrt{F} =  L F - 2\Gamma(\sqrt{F}).
    \end{align}
    Furthermore, the reinforced Bakry--\'Emery condition reads
    \begin{align}\label{eq:dPsi-4}
        \gamma(s) L \Gamma_w(P_{t-s} f) - 2 \gamma(s) \Gamma_w(P_{t-s} f, L P_{t-s} f) - 2\rho \gamma(s) \Gamma_w(P_{t-s} f) \ge 2 \gamma(s)\Gamma\tp{\sqrt{\Gamma_w(P_{t-s} f)}}.
    \end{align}
    Combining~\eqref{eq:dPsi-1},~\eqref{eq:dPsi-2},~\eqref{eq:dPsi-3},~\eqref{eq:dPsi-4} and $\gamma'(s) + 2\rho \gamma(s) = \frac{4}{b}$, it holds
    \begin{align}\label{eq:dPsi-5}
    \frac{\dif }{\dif s} \Psi(s) \ge P_s \frac{R}{2\sqrt{F}},
    \end{align}
    where
    \begin{align}\label{eq:dPsi-6}
    R = L \+I^2 (P_{t-s} f) - 2 \Gamma(\sqrt{F}) - 2 \+I(P_{t-s} f) \+I'(P_{t-s} f) P_{t-s} L f+ \frac{4}{b}\Gamma_w(P_{t-s} f) + 2\gamma(s) \Gamma(\sqrt{\Gamma_w(P_{t-s}f)}) .
    \end{align}
    Note that by Minkowski's inequality, it holds
    \begin{align*}
        & 2\Gamma(\sqrt{F})(x) \\
        ={}& \sum_{i=1}^n q_i(x) \tp{\sqrt{\+I^2(P_{t-s} f)(x) + \gamma(s) \Gamma_w(P_{t-s} f)(x)} - \sqrt{\+I^2(P_{t-s} f)(x^i) + \gamma(s) \Gamma_w(P_{t-s} f)(x^i)}}^2\\
        \le{}& \sum_{i=1}^n q_i(x) \tp{\tp{\+I(P_{t-s} f)(x) - \+I(P_{t-s} f)(x^i)}^2 + \gamma(s) \tp{\sqrt{\Gamma_w(P_{t-s} f)(x)} - \sqrt{\Gamma_w(P_{t-s} f)(x^i)}}^2}\\
        ={}& 2\Gamma(\+I(P_{t-s}f))(x) + 2\gamma(s) \Gamma(\sqrt{\Gamma_w(P_{t-s} f)})(x).
    \end{align*}
    Combining~\eqref{eq:dPsi-5} and~\eqref{eq:dPsi-6}, it only remains to show
    \begin{align}\label{eq:dPsi-7}
        L \+I^2(P_{t-s} f) - 2 \+I(P_{t-s} f) \+I'(P_{t-s} f) P_{t-s} L f+ \frac{4}{b}\Gamma_w(P_{t-s} f) - 2 \Gamma(\+I(P_{t-s} f)) \ge 0.
    \end{align}
    By definition of $L$ and $\Gamma_w$, the left-hand side of the above inequality satisfies
    \begin{align*}
    \text{LHS of~\eqref{eq:dPsi-7} at $x$} &= \sum_{i=1}^n q_i(x) \tp{\+I^2(P_{t-s} f)(x^i) - \+I^2(P_{t-s} f)(x)}\\
    &- 2 \sum_{i=1}^n q_i(x) \+I(P_{t-s} f)(x) \+I'(P_{t-s} f)(x) ((P_{t-s} f)(x^i) - (P_{t-s}f)(x))\\
    &+ \frac{2}{b}\sum_{i=1}^n q_i(x) w_i(x) \tp{P_{t-s} f(x^i) - P_{t-s} f(x)}^2\\
    &-\sum_{i=1}^n q_i(x) (\+I(P_{t-s} f)(x^i) - \+I(P_{t-s} f)(x))^2,
    \end{align*}
    which is non-negative by $w_i(x) \ge b$ for all $x\in \Omega,i\in [n]$ and
    \begin{align*}
        \+I(x)\tp{\+I(y) - \+I(x) - \+I'(x)(y-x)} + (y-x)^2 \ge 0,
    \end{align*}
    for all $x,y \in (0,1)$, by~\Cref{fact:gaussian_isoperimetric_profile_ineq}.
\end{proof}

The following improved Talagrand inequality can be obtained by~\Cref{lem:local_bobkov}.

\begin{corollary}[Talagrand's variance--surface-area inequality]
\label{lem:improved_talagrand}
    Let $\mu$ be a distribution supported on $\Omega = \{\pm 1\}^n$ and $\{q_i: \Omega \to \R_+\}_{i\in[n]}$ be transition rates reversible for $\mu$.
    Assume $\{w_i:\Omega \to \R_+\}_{i\in[n]}$ is a collection of weight functions that are $b$-bounded for some $b \in (0,1]$.

    Suppose the following conditions hold:
    \begin{itemize}
        \item The reinforced Bakry--\'Emery condition (equivalently, the strong gradient estimate) holds with constant $\rho_{\RBE}$;
        \item The log-Sobolev inequality holds with constant $\rho_{\LSI} > 0$.
    \end{itemize}
    Then, for any $f:\Omega \to \{0,1\}$, it holds
   \begin{align}\label{eq:improved_talagrand}
        A_w(f) \gtrsim \frac{\sqrt{b} \rho_{\LSI}}{\sqrt{\rho_{\LSI} + \kappa}} \cdot \Var[\mu]{f}\sqrt{\log \frac e{\Var[\mu]{f}}},
   \end{align}
   where $\kappa = \max \set{0,-\rho_{\RBE}}$.
\end{corollary}

\begin{proof}
    The assertion is trivial for constant functions, so we assume throughout that \(f\) is nonconstant.
    Fix a function $f: \Omega \to \{0,1\}$. Without loss of generality, we may assume $v := \E[\mu]{f} \le \frac{1}{2}$.
    By~\Cref{lem:local_bobkov}, it holds
    \begin{align}\label{eq:Bobkov-01}
        \E[\mu]{\+I(P_t f)} \le \sqrt{\gamma(t)} A_w(f).
    \end{align}
    The Gaussian isoperimetric profile $\+I(P_t f) \gtrsim (P_t f)(1-P_t f)$. Therefore, the expectation of $\+I(P_t f)$ can be lower bounded by
    \begin{align}\label{eq:isoperimetric-bound}
        \E[\mu]{\+I(P_t f)} \gtrsim \tp{v - \norm{P_t f}_{2,\mu}^2}.
    \end{align}
    By hypercontractivity, $\norm{P_t f}_{2,\mu}^2 \le \tp{v}^{1+\tanh(\rho_{\LSI} t)}$. Combining~\eqref{eq:Bobkov-01} and~\eqref{eq:isoperimetric-bound}, the surface area $A_w(f)$ then satisfies
    \begin{align*}
        A_w(f) \gtrsim \frac{v(1-v^{\tanh(\rho_{\LSI} t)})}{\sqrt{\gamma(t)}}.
    \end{align*}
    Take $t = \frac{1}{100(\kappa+\rho_{\LSI} \log 1/v)}$. By estimates $1-v^{\tanh(\rho_{\LSI} t)} \gtrsim \rho_{\LSI} t \log 1/v$ and $\gamma(t) \lesssim \frac{t}{b}$, it holds
    \begin{align*}
        A_w(f) \gtrsim \frac{\sqrt{b} \rho_{\LSI}}{\sqrt{\kappa + \rho_{\LSI} \log 1/v}} v \log 1/v \gtrsim \frac{\sqrt{b}\rho_{\LSI}}{\sqrt{\kappa + \rho_{\LSI}}} v \sqrt{\log 1/v} \gtrsim \frac{\sqrt{b}\rho_{\LSI}}{\sqrt{\kappa + \rho_{\LSI}}} \Var[\mu]{f} \sqrt{\log \frac{\e}{\Var[\mu]{f}}}.
    \end{align*}
    The inequalities follow from $v \le \frac{1}{2}$.
\end{proof}

\subsection{Talagrand's \texorpdfstring{$L^1$--$L^2$}{L1–L2} and KKL inequalities}

Talagrand's $L^1$--$L^2$ inequality~\cite{Talagrand94} is a logarithmic refinement of the Poincar\'e inequality.
For Markov random fields on bounded-degree graphs satisfying a log-Sobolev inequality, an analogue of this inequality was established in~\cite{KoehlerLifshitzMinzerMossel23}.
Our proof follows the semigroup approach of~\cite{CorderoErausquinLedoux12}, replacing the bounded-degree assumption with gradient estimates.
\begin{theorem}[Weak $L^1$--$L^2$ inequality]\label{lem:l1_l2_talagrand_weak}
    Let $\mu$ be a distribution supported on $\Omega = \{\pm 1\}^n$ and $\{q_i: \Omega \to \R_+\}_{i\in[n]}$ be transition rates reversible for $\mu$.
    Assume the strong gradient estimate~\ref{eq:SGE} and the log-Sobolev inequality hold with constants $\rho_{\SGE}$ and $\rho_{\LSI}$ respectively, and $\{w_i:\Omega \to \mathbb{R}_+\}_{i \in [n]}$ is a collection of weight functions that are $b$-bounded for some $b \in (0,1]$. It holds for any $f:\Omega \to \R$,
    \begin{align}\label{eq:l1_l2_talagrand_weak}
        I(f) \gtrsim \frac{b^2 \rho^2_{\LSI}}{\rho_{\LSI} + \kappa} \Var[\mu]{f}\tp{1 + \log \frac{I_w(f)}{A_w(f)^2}},
    \end{align}
    where $\kappa = \max\set{-\rho_{\SGE},0}$.
\end{theorem}

\begin{proof}
    The assertion is trivial for constant functions, so we assume throughout that \(f\) is nonconstant.
Note that for any $T \ge 0$,
\begin{align*}
\Var[\mu]{f} - \Var[\mu]{P_T f} = 2 \int_0^{T} \E[\mu]{\Gamma(P_t f)} \dif t &\le \frac{2}{b} \int_0^{T} \E[\mu]{\Gamma_w(P_tf)} \dif t\\
&\le \frac{2}{b} \int_0^{T} \e^{-2 \rho_{\mathrm{SGE}} t} \cdot \E[\mu]{\tp{P_t \sqrt{\Gamma_w f}}^2 } \dif t,
\end{align*}
where the last inequality follows from~\ref{eq:SGE}. By Poincar\'e inequality, it holds
\begin{align*}
    (1-\e^{-2\rho_{\LSI}T}) \Var[\mu]{f} \le \frac{2}{b} \int_0^{T} \e^{-2\rho_{\SGE} t} \cdot \E[\mu]{\tp{P_t \sqrt{\Gamma_w f}}^2} \dif t.
\end{align*}
 By hypercontractivity, it holds
\begin{align*}
    \E[\mu]{(P_t \sqrt{\Gamma_w f})^2} = \norm{P_t \sqrt{\Gamma_w f}}_{2,\mu}^2 \le \norm{\sqrt{\Gamma_w f}}_{p_t,\mu}^2 \le \norm{\sqrt{\Gamma_w f}}_{1,\mu}^{2 \theta_t} \norm{\sqrt{\Gamma_w f}}_{2,\mu}^{2 (1-\theta_t)},
\end{align*}
where $p_t = 1 + \e^{-2\rho_{\mathrm{LSI}} t}$ and $\theta_t = \frac{2}{p_t} - 1 = \tanh (\rho_{\LSI} t)$, and the last inequality follows from H\"older's inequality.
Note that $\norm{\sqrt{\Gamma_w f}}_{1,\mu} = A_w(f)$, $\norm{\sqrt{\Gamma_w f}}_{2,\mu}^2 = I_w(f)$. Therefore, it holds
\begin{align}\label{eq:var-bound}
    \nonumber (1-\e^{-2\rho_{\LSI} T}) \Var[\mu]{f} &\le \frac{2}{b} \int_0^{T} \e^{-2 \rho_{\SGE} t} \cdot A_w(f)^{2 \theta_t} I_w(f)^{1-\theta_t} \dif t\\
    &\le \frac{2 I_w(f)}{b} \int_0^{T} \e^{-2 \rho_{\SGE} t} \tp{\frac{A_w(f)^2}{I_w(f)}}^{\theta_t} \dif t.
\end{align}
We now choose $T = \frac{1}{\rho_{\LSI} + \kappa}$.
Let $R = \frac{I_w(f)}{A^2_w(f)} \ge 1$. Since $\tanh x \ge \tanh(1) x$ for all $x \in [0,1]$ and $\e^{-2\rho_{\SGE} t} \le \e^2$ for all $t \in [0,T]$, the integral satisfies
\begin{align*}
    \int_0^T \e^{-2\rho_{\SGE} t} R^{-\tanh\tp{\rho_{\LSI} t}}\dif t &\lesssim \int_0^T \e^{-\tanh(1) \rho_{\LSI} \log R \cdot t} \dif t \lesssim \min\set{T,\frac{1}{\rho_{\LSI} \log R}} \lesssim \frac{T}{1+\rho_{\LSI} T \log R}.
\end{align*}
By~\eqref{eq:var-bound} and $1-\e^{-2 \rho_{\LSI} T} \gtrsim \rho_{\LSI} T$, the variance $\Var[\mu]{f}$ then satisfies
\begin{align*}
    \Var[\mu]{f} \lesssim \frac{1}{b\rho_{\LSI}(1+\rho_{\LSI} T \log R)} I_w(f) \lesssim \frac{\rho_{\LSI} + \kappa}{b\rho_{\LSI}^2} \cdot \frac{I_w(f)}{1+ \log R},
\end{align*}
where the last inequality follows from $\rho_{\LSI} T \le 1$.
Since $I_w(f)\leq I(f)/b$, this concludes the proof.
\end{proof}

With the coordinate-wise strong gradient estimate, we establish the following $L^1$--$L^2$ inequality.
The proof presented here resembles the proof in \cite{CorderoErausquinLedoux12}, though our notion of curvature is defined with weights and absolute values; see \Cref{subsec:related-work,rmk:CSGE-CL12} for detailed discussion and comparison.

\begin{theorem}[Talagrand's $L^1$--$L^2$ inequality, \cite{CorderoErausquinLedoux12}]\label{lem:l1_l2_talagrand}
     Let $\mu$ be a distribution supported on $\Omega = \{\pm 1\}^n$ and $\{q_i: \Omega \to \R_+\}_{i\in[n]}$ be transition rates reversible for $\mu$.
    Assume~\ref{eq:0-CSGE} holds with constant $\rho_{\CSGE}$, log-Sobolev inequality holds with constant $\rho_{\LSI} > 0$, and $\{w_i:\Omega \to \mathbb{R}_+\}_{i \in [n]}$ is a collection of weight functions that are $b$-bounded for some $b \in (0,1]$. The following inequality holds for any $f:\Omega \to \R$:
    \begin{align}\label{eq:l1_l2_talagrand_2}
        \Var[\mu]{f} \lesssim \frac{\rho_{\LSI} + \kappa_0}{b \rho^2_{\LSI}} \cdot \sum_{i=1}^n \frac{I_{w,i} (f)}{1 + \log\tp{\sqrt{I_{w,i}(f)} / \E[\mu]{\abs{\delta_{w,i}} f}}},
    \end{align}
    where $\kappa_0 = \max\set{-\rho_{\CSGE},0}$.
    A summand is interpreted as zero when $I_{w,i}(f)=0$.

    Furthermore, we have that
    \begin{align}\label{eq:l1_l2_talagrand_3}
        I(f) \gtrsim \frac{b^2 \rho^2_{\LSI}}{\rho_{\LSI} + \kappa_0}\cdot \Var[\mu]{f}\tp{1 + \log \frac{I_w(f)}{J_w(f)}} .
    \end{align}
\end{theorem}

\begin{proof}
    The assertion is trivial for constant functions, so we assume throughout that \(f\) is nonconstant.
Similar to the previous proof, we set $T = \frac{1}{\rho_{\LSI} + \kappa_0}$. The variance $\Var[\mu]{f}$ can be bounded by
\begin{align*}
(1-\e^{-2\rho_{\LSI} T})\Var[\mu]{f} \le \frac{2}{b} \int_0^{T} \E[\mu]{\norm{\abs{\nabla_w} P_t f}_2^2} \dif t \le \frac{2}{b}\int_0^{T} \e^{-2 \rho_{\CSGE} t}\E[\mu]{ \norm{P_t \abs{\nabla_w} f}_2^2} \dif t.
\end{align*}
Let $g_i = \abs{\delta_{w,i}} f$. It holds
\begin{align}\label{eq:var-bound-2}
(1-\e^{-2\rho_{\LSI} T}) \Var[\mu]{f} \le \frac{2}{b} \sum_{i=1}^n \int_0^{T} \e^{- 2 \rho_{\CSGE} t} \norm{P_t g_i}_{2,\mu}^2 \dif t.
\end{align}
Following exactly the same argument as in the previous proof, the integral can be bounded by
\begin{align*}
\int_0^{T} \e^{-2 \rho_{\CSGE} t} \norm{P_t g_i}_{2,\mu}^2 \dif t &\le \norm{g_i}_{2,\mu}^2 \int_0^{T}  \e^{-2 \rho_{\CSGE} t} \tp{\frac{\norm{g_i}_{1,\mu}^2} {\norm{g_i}_{2,\mu}^2}}^{\theta_t} \dif t\\
 &\lesssim \frac{1}{\rho_{\LSI}} \cdot \frac{\norm{g_i}_{2,\mu}^2}{1+\log \frac{\norm{g_i}_{2,\mu}}{\norm{g_i}_{1,\mu}}}.
\end{align*}
The proof of~\eqref{eq:l1_l2_talagrand_2} then follows from~\eqref{eq:var-bound-2}, $\norm{g_i}_{1,\mu} = \E[\mu]{\abs{\delta_{w,i}} f}$ and $\norm{g_i}_{2,\mu} = \sqrt{I_{w,i}(f)}$.

To prove~\eqref{eq:l1_l2_talagrand_3}, we observe
\begin{align*}
    \sum_{i=1}^n \norm{P_t g_i}_{2,\mu}^2 \le \sum_{i=1}^n \norm{g_i}_{1,\mu}^{2\theta_t} \norm{g_i}_{2,\mu}^{2(1-\theta_t)} \le \tp{\sum_{i=1}^n \norm{g_i}_{1,\mu}^2}^{\theta_t} \tp{\sum_{i=1}^n \norm{g_i}_{2,\mu}^2}^{1-\theta_t},
\end{align*}
where $\theta_t = \tanh\tp{\rho_{\LSI} t}$ and the last inequality follows from H\"older's inequality. By $J_w(f) = \sum_{i=1}^n \norm{g_i}_{1,\mu}^2$ and $I_w(f) = \sum_{i=1}^n \norm{g_i}_{2,\mu}^2$, we have
\begin{align*}
    \Var[\mu]{f} \lesssim \frac{\rho_{\LSI} + \kappa_0}{b \rho_{\LSI}^2}\frac{I_w(f)}{1+\log \frac{I_w(f)}{J_w(f)}} \lesssim  \frac{\rho_{\LSI} + \kappa_0}{b^2 \rho_{\LSI}^2}\frac{I(f)}{1+\log \frac{I_w(f)}{J_w(f)}} .
\end{align*}
This completes the proof of~\eqref{eq:l1_l2_talagrand_3}.
\end{proof}

\begin{remark}
    Since $A^2_w(f) \ge J_w(f)$,~\eqref{eq:l1_l2_talagrand_3} is a tighter bound for $I(f)$ than~\eqref{eq:l1_l2_talagrand_weak}.
\end{remark}
If the transition rates are also marginally bounded, then the KKL inequality can be obtained by Talagrand's $L^1$--$L^2$ inequality.
\begin{corollary}[KKL]\label{cor:KKL}
    Let $\mu$ be a distribution supported on $\Omega = \{\pm 1\}^n$ and $\{q_i: \Omega \to \R_+\}_{i\in[n]}$ be transition rates reversible for $\mu$.
    Assume~\ref{eq:0-CSGE} and the log-Sobolev inequality hold with constants $\rho_{\CSGE}$ and $\rho_{\LSI}$ respectively.
    Given a collection of weight functions $\{w_i:\Omega \to \mathbb{R}_+\}_{i \in [n]}$, if both transition rates $\{q_i\}_{i\in [n]}$ and $\{w_i\}_{i\in [n]}$ are $b$-bounded for some $b \in (0,1]$, then for any Boolean function $f: \Omega\to \{\pm 1\}$,
    \begin{align}\label{eq:KKL_boolean_1}
        I_w(f) \gtrsim_{b} \frac{\rho_{\LSI}^2}{\rho_{\LSI}+\kappa_0} \cdot \log (1/\+M_w(f))\cdot \Var[\mu]{f},
    \end{align}
    where $\+M_w(f) = \max_{i\in[n]} I_{w,i}(f)$ and $\kappa_0 = \max\{-\rho_\CSGE,0\}$.
    Consequently, the following holds:
    \begin{align}\label{eq:KKL_boolean_2}
        \+M_w(f) \geq \Omega_b(1) \cdot \frac{\rho_{\LSI}^2}{(\rho_{\LSI}+\kappa_0)n}\log\tp{\frac{(\rho_{\LSI}+\kappa_0)n}{\rho_{\LSI}^2}} \cdot \Var[\mu]{f}.
    \end{align}
\end{corollary}
\begin{proof}
    The assertion is trivial for constant functions, so we assume throughout that \(f\) is nonconstant.
    Since $f$ is a Boolean function, by marginal boundedness, we can simplify $I_{w,i}$ and $J_{w}$ as follows,
    $$
        I_{w,i}(f) \geq 2b^2 \Pr[\mu]{f(x^i)\neq f(x)}\quad \text{ and } \quad J_{w}(f) \leq \frac{2}{b^2}\sum_{i=1}^n \Pr[\mu]{f(x^i)\neq f(x)}^2.
    $$
    Since $\+M_w(f) \geq I_{w,i}(f)$ for any $i\in[n]$, we conclude that
    \begin{align}\label{eq:IM_to_J}
         I_{w}(f) \+M_w(f) \geq 4b^4\sum_{i=1}^n \Pr[\mu]{f(x^i)\neq f(x)}^2 \geq 2b^6 J_w(f) .
    \end{align}
    By \eqref{eq:l1_l2_talagrand_3} in \Cref{lem:l1_l2_talagrand}, the first KKL inequality immediately follows from \eqref{eq:IM_to_J}.
    To obtain~\eqref{eq:KKL_boolean_2}, set $a=\rho_{\LSI}^2/((\rho_{\LSI}+\kappa_0)n)$ and $M=\+M_w(f)$.
    If $a\geq 1$, the claimed lower bound is nonpositive.
    If $0<a<1$ and $M\geq\sqrt a$, use $a\log(1/a)\leq \sqrt a$ and $\Var[\mu]{f}\leq 1$.
    Otherwise, $M<\sqrt a$, so $I_w(f)\leq nM$ and $\log(1/M)>\frac12\log(1/a)$ give the claimed bound directly from the first inequality.
\end{proof}

\subsection{Eldan--Gross inequality}

Conjectured by Talagrand~\cite{Talagrand97} and proved by Eldan and Gross~\cite{EldanGross22}, the Eldan--Gross inequality unifies the KKL inequality with Talagrand's variance--surface-area inequality, retaining the square-root sensitivity boundary while capturing the logarithmic KKL gain through the squared influences.
We establish the following counterpart under our curvature assumptions, following the strategy in \cite{IvanisviliZhang26}.

\begin{theorem}[Eldan--Gross inequality]\label{thm:eldan_gross}
     Let $\mu$ be a distribution supported on $\Omega = \{\pm 1\}^n$ and $\{q_i: \Omega \to \R_+\}_{i\in[n]}$ be transition rates reversible for $\mu$.
    Assume both~\ref{eq:0-CSGE} and~\ref{eq:infty-CSGE} hold with constant $\rho_{\CSGE} \in \mathbb{R}$, log-Sobolev inequality holds with constant $\rho_{\LSI} > 0$, and $\{w_i:\Omega \to \mathbb{R}_+\}_{i \in [n]}$ is a collection of weight functions that are $b$-bounded for some $b \in (0,1]$.
    For any Boolean function $f:\Omega \to \{\pm 1\}$, it holds
   \begin{align}\label{eq:eldan_gross}
        A_w(f) \gtrsim_{b,\rho_{\CSGE}} \frac{\sqrt{\rho_{\LSI}} \min\set{1,\rho_{\LSI}}}{\sqrt{1 + \log_+\tp{1/\rho_{\LSI}}}} \cdot \Var[\mu]{f}\sqrt{\log\tp{1 + \frac e{J_w(f)}}}.
   \end{align}
    Here $\log_+x:=\max\{\log x,0\}$ for $x>0$.
\end{theorem}
\begin{proof}
The assertion is trivial for constant functions, so we assume throughout that \(f\) is nonconstant.
Fix a Boolean function $f:\{\pm 1\}^n \to \{\pm 1\}$ and let $h=(1+f)/2$. When $\Var[\mu]{f} \le \sqrt{J_w(f)}$,
the improved Talagrand inequality (\Cref{lem:improved_talagrand}), applied to $h$ together with $A_w(h)=A_w(f)/2$ and $\Var[\mu]{h}=\Var[\mu]{f}/4$, implies the Eldan--Gross inequality:
\begin{align*}
A_w(f) & \gtrsim_{b,\rho_{\CSGE}} \sqrt{\rho_{\LSI}} \min\set{1,\rho_{\LSI}} \cdot \Var[\mu]{f} \sqrt{\log \frac{\e}{\Var[\mu]{f}}}\\
&\gtrsim_{b,\rho_{\CSGE}} \sqrt{\rho_{\LSI}} \min\set{1,\rho_{\LSI}} \cdot \Var[\mu]{f} \sqrt{\log \tp{1+\frac{\e}{\Var[\mu]{f}^2}}}\\
&\gtrsim_{b,\rho_{\CSGE}} \sqrt{\rho_{\LSI}} \min\set{1,\rho_{\LSI}}  \Var[\mu]{f} \sqrt{\log \tp{1+\frac{\e}{J_w(f)}}}.
\end{align*}
Therefore, we may assume $\Var[\mu]{f} > \sqrt{J_w(f)}$ in the following.

By Poincar\'e inequality,~\Cref{lem:l1_l2_talagrand} and~\ref{eq:infty-CSGE}, it holds for any $s \ge 0$ that
\begin{align*}
I(P_s f) &\gtrsim_{b,\rho_{\CSGE}} \Var[\mu]{P_s f}\left[\rho_{\LSI}+\rho_{\LSI}\min\set{1,\rho_{\LSI}}\tp{1+\log\frac{I_w(P_s f)}{J_w(P_s f)}}\right]\\
&\gtrsim_{b,\rho_{\CSGE}} \Var[\mu]{P_s f}\left[\rho_{\LSI}+\frac{\rho_{\LSI}\min\set{1,\rho_{\LSI}}}{1+\log_+\tp{1/\rho_{\LSI}}}
\tp{1+\tp{\log\frac{\Var[\mu]{P_s f}}{J_w(f)}-2\kappa s}_+}\right],
\end{align*}
where $(x)_+ = \max \set{x,0}$ and $\kappa = (-\rho_{\CSGE})_+$.
For simplicity, let $C = \frac{\rho_{\LSI} \min \set{1,\rho_{\LSI}} }{1+\log_+\tp{1/\rho_{\LSI}}}$, and further
let $T = \min\{s \ge 0 \mid \Var[\mu]{P_s f} \le \frac{\Var[\mu]{f}}{2}\}$.

We first consider the case where $\kappa T \le 1$. It holds
\begin{align*}
\frac{\Var[\mu]{f}}{2} = \Var[\mu]{f} - \Var[\mu]{P_T f} &= 2\int_0^{T} I(P_s f) \dif s\\
 &\gtrsim_{b,\rho_{\CSGE}} C \cdot \int_0^{T} \Var[\mu]{P_s f} \tp{1 + \log_+ \frac{\Var[\mu]{P_s f}}{J_w(f)}} \dif s\\
&\gtrsim_{b,\rho_{\CSGE}} C\cdot \Var[\mu]{f} \tp{1 + \log_+ \frac{\Var[\mu]{f}}{J_w(f)}} T.
\end{align*}
This implies $CT \tp{1 + \log_+ \frac{\Var[\mu]{f}}{J_w(f)}} \lesssim_{b,\rho_{\CSGE}} 1$.

By~\Cref{lem:local_bobkov}, it holds for any $s \ge 0$ that
\begin{align*}
P_s h - (P_s h)^2 \le \+I(P_s h) \lesssim \sqrt{\gamma(s)} P_s \sqrt{\Gamma_w(h)} .
\end{align*}
By taking expectation over $\mu$ on both sides, the following holds when $\kappa s \le 1$:
\begin{align}\label{eq:bobkov-exp}
\Var[\mu]{f} - \Var[\mu]{P_s f} \lesssim \sqrt{\gamma(s)} A_w(f) \lesssim_b \sqrt{s} A_w(f).
\end{align}
By taking $s = T$, it holds
\begin{align*}
A_w(f) \gtrsim_b \frac{\Var[\mu]{f}}{\sqrt{T}} &\gtrsim_{b,\rho_{\CSGE}} \sqrt{C} \Var[\mu]{f} \sqrt{1+\log_+\frac{\Var[\mu]{f}}{J_w(f)}}\\
&\gtrsim_{b,\rho_{\CSGE}} \frac{\sqrt{\rho_{\LSI}} \min \set{1,\rho_{\LSI}}}{\sqrt{1+\log_+\tp{1/\rho_{\LSI}}}} \cdot \Var[\mu]{f} \sqrt{\log \tp{1+\frac{\e}{J_w(f)}}},
\end{align*}
where the last inequality follows from the assumption $\Var[\mu]{f} \gtrsim \sqrt{J_w(f)}$ and the definition of $C$.

We now consider the remaining case $\kappa T > 1$. Let $T^\star = \frac{1}{\kappa} < T$. 
Since $T^\star<T$ and $\kappa T^\star=1$, a similar argument gives
\begin{align}\label{eq:EG-ineq-1}
    \Var[\mu]{f} - \Var[\mu]{P_{T^\star} f} \gtrsim_{b,\rho_{\CSGE}}\Var[\mu]{f}\tp{\rho_{\LSI}+C\tp{1+\log_+\frac{\Var[\mu]{f}}{J_w(f)}}}T^\star.
\end{align} 
By~\eqref{eq:bobkov-exp} and $\gamma(T^\star) \lesssim_{b,\rho_{\CSGE}} 1$, we have
\begin{align}\label{eq:EG-ineq-2}
    \Var[\mu]{f} - \Var[\mu]{P_{T^\star} f} \lesssim_{b,\rho_{\CSGE}} A_w(f).
\end{align}
Combining~\eqref{eq:EG-ineq-1},~\eqref{eq:EG-ineq-2} and our assumption $\Var[\mu]{f} \ge \sqrt{J_w(f)}$, it holds
\begin{align*}
    A_w(f) &\gtrsim_{b,\rho_{\CSGE}} \Var[\mu]{f}\tp{\rho_{\LSI}+C\tp{1+\log_+\frac{\Var[\mu]{f}}{J_w(f)}}}\\
    &\gtrsim_{b,\rho_{\CSGE}} \sqrt{\rho_{\LSI}C} \Var[\mu]{f}\sqrt{1+\log_+\frac{\Var[\mu]{f}}{J_w(f)}}\\
    &\gtrsim_{b,\rho_{\CSGE}}\frac{\sqrt{\rho_{\LSI}}\min\set{1,\rho_{\LSI}}}{\sqrt{1+\log_+\tp{1/\rho_{\LSI}}}}\Var[\mu]{f}\sqrt{\log\tp{1+\frac{\e}{J_w(f)}}}.
\end{align*}
This completes the proof.
\end{proof}

\section{Sufficient Conditions for Gradient Estimates}
\label{sec:GE}

In this section, we establish \ref{eq:CSGE} by choosing appropriate transition rates and weights, and the resulting bound can be controlled by norms of the Dobrushin influence matrix.

\subsection{Canonical transition rates and weights}
Given a distribution $\mu$ with full support on $\Omega = \{\pm 1\}^n$, for any $x \in \Omega$ and $i \in [n]$, we define
\begin{align}\label{eq:canonical-r-w}
    q^\star_i(x) = \frac 12\sqrt{\frac{\mu(x^i)}{\mu(x)}}
        \qquad\text{and}\qquad
    w^\star_i(x) = 2q^\star_i(x) = \sqrt{\frac{\mu(x^i)}{\mu(x)}}.
\end{align}
It is straightforward to show that such transition rates $\{q^\star_i\}_{i \in [n]}$ satisfy the detailed balance condition and hence the associated heat semigroup $(P_t)_{t\ge 0}$ is reversible with respect to $\mu$.
We call \eqref{eq:canonical-r-w} the \emph{canonical transition rates and weights} associated with $\mu$. 

Under the canonical rates and weights, many notions in the $\Gamma$-calculus simplify to a nice form, which we summarize below.
Throughout, we use superscript $\star$ to indicate that the weight functions are canonical, and thereby omit $w^\star$ in the subscript; e.g., $\Gamma^\star = \Gamma_{w^\star}$ denotes the weighted carr\'e du champ operator under $q^\star$ and $w^\star$.
We have the following:
\begin{align*}
    \delta^\star_i f(x) &= \sqrt{\frac{1}{2} w^\star_i(x) q^\star_i(x)} (f(x^i) - f(x)) = q^\star_i(x) (f(x^i) - f(x)); \\
    \grad^\star f &= (\delta^\star_1 f, \dots, \delta^\star_n f); \quad
    L f = \sum_{i=1}^n \delta^\star_i f; \quad
    \Gamma^\star f = \sum_{i=1}^n (\delta^\star_i f)^2.
\end{align*}

We also record two important properties of the canonical transition rates and weights. These properties are essential for establishing \ref{eq:SGE} and \ref{eq:CSGE}.

\begin{fact}[Square property]\label{fact:square}
    For any $i,j \in [n]$ and $x \in \Omega$, we have
    \begin{align*}
        q^\star_i(x) q^\star_j(x^i) = q^\star_j(x) q^\star_i(x^j).
    \end{align*}
\end{fact}

\begin{lemma}[Commutator identity]
\label{lem:commutator}
    Let $\mu$ be a distribution with full support on $\Omega = \{\pm 1\}^n$.
    Let $M: \Omega \to \R^{n \times n}$ be a matrix field defined by, for each $x \in \Omega$,
    \begin{align}\label{eq:local-matrix}
        M_{ij}(x) &:= \frac{1}2 \tp{\sqrt{\frac{\mu(x^{ij})}{\mu(x^i)}}-\sqrt{\frac{\mu(x^j)}{\mu(x)}}}, \quad i \neq j; \nonumber \\ 
        M_{ii}(x) &:= 0.
    \end{align}
    For any function $f: \Omega \to \R$, it holds
    \begin{align*}
        (L \grad^\star - \grad^\star L) f = \left( \diag\{M\*1\} - M^\top \right) \grad^\star f.
    \end{align*}
\end{lemma}

\begin{proof}
    Fix $f: \Omega\to \R$, $x\in \Omega$.
    For any $i,j \in [n]$ with $i\neq j$, by definition,
    $$
        \delta^\star_j \delta^\star_i f(x) = q^\star_j(x) (\delta^\star_i f(x^j) - \delta^\star_i f(x)) = q^\star_j(x) q^\star_i(x^j)(f(x^{ij})-f(x^j)) - q^\star_j(x)q^\star_i(x)(f(x^i)-f(x)).
    $$
    \Cref{fact:square} implies the cancellation of the second derivative in the commutator term $\delta^\star_j\delta^\star_i - \delta^\star_i\delta^\star_j$.
    We have
    \begin{align}
        (\delta^\star_j \delta^\star_i - \delta^\star_i \delta^\star_j) f(x) &= (q_j^\star(x)q_i^\star(x^j) - q_i^\star(x)q_j^\star(x^i)) f(x^{ij})  \nonumber
        \\&-(q_i^\star(x^j) - q_i^\star(x))q_j^\star(x)f(x^j)+(q_j^\star(x^i) - q_j^\star(x))q_i^\star(x)f(x^i) \nonumber
        \\&= M_{ij}(x)q_i^\star(x)f(x^i) - M_{ji}(x)q_j^\star(x)f(x^j) \nonumber
        \\&= M_{ij}(x)\delta^\star_i f(x) - M_{ji}(x)\delta^\star_j f(x), \label{eq:commutator_delta_ij}
    \end{align}
    where the last equality follows from $M_{ij}(x) q_i^\star(x) = M_{ji}(x) q_j^\star(x)$ by \Cref{fact:square}.
    Then the following equality immediately follows from
    \eqref{eq:commutator_delta_ij},
    \begin{align}
        (L\delta^\star_i - \delta^\star_iL) f(x) &= \sum_{j=1}^n (\delta^\star_j\delta^\star_i - \delta^\star_i\delta^\star_j)f(x)
        = \sum_{j=1}^n M_{ij}(x)\delta^\star_i f(x) - \sum_{j=1}^n M_{ji}(x)\delta^\star_j f(x) \nonumber
        \\&= \tp{M(x)\*1}_i\delta^\star_i f(x) - \inner{(M^\top(x))_i}{\nabla^\star f(x)}, \nonumber
    \end{align}
    that is, $(L \grad^\star - \grad^\star L) f = \left( \diag\{M\*1\} - M^\top \right) \grad^\star f$ holds for any $x\in \Omega$.
\end{proof}

\subsection{Sufficient condition for SGE} \label{sec:SGE_boolean_hypercube}

\begin{theorem}[SGE under canonical transition rates and weights]
\label{thm:SGE_boolean_hypercube}
    Let $\mu$ be a distribution on $\Omega = \{\pm 1\}^n$ with full support. Let $M,D: \Omega \to \R^{n \times n}$ be two matrix fields where $M$ is defined by \eqref{eq:local-matrix} and $D$ is defined by, for each $x \in \Omega$,
    \begin{align}\label{eq:def-D}
        D(x) := \diag\left\{\sqrt{\frac{\mu(x)}{\mu(x^i)}}\right\}_{i\in[n]}.
    \end{align}
    Then, under the canonical transition rates and weights, \ref{eq:SGE} holds with constant
    \begin{align*}
        \rho = \min_{x \in \Omega} \;
        \lambda_{\min}\left( D(x) + \diag\{M(x)\*1\} - \frac{1}{2}\left( M(x) + M(x)^\top \right) \right),
    \end{align*}
    where $\lambda_{\min}(A)$ denotes the minimum eigenvalue of a symmetric matrix $A$.
\end{theorem}

\begin{proof}
    By \Cref{prop:equivalence_BE_GE}, it is sufficient to establish the \ref{eq:rBE} condition.
    Specifically, for any function $f: \Omega \to \R$, we aim to show the following pointwise inequality:
    \begin{align*}
        \Gamma^\star_2(f) - \Gamma(\sqrt{\Gamma^\star(f)}) \ge \rho \cdot \Gamma^\star(f),
    \end{align*}
    where we recall that $\Gamma$ is the unweighted carr\'e du champ operator with respect to the canonical rates $q^\star$.
    By \Cref{lem:gamma2-formula}, we have
    \begin{align}
        \Gamma^\star_2(f) - \Gamma(\sqrt{\Gamma^\star(f)}) 
        = \sqrt{\Gamma^\star(f)} L \sqrt{\Gamma^\star(f)} - \Gamma^\star(f,Lf) 
        = \norm{\nabla^\star f}_2 L \norm{\nabla^\star f}_2 - \inner{\nabla^\star f}{\nabla^\star Lf}. \label{eq:first-rBE}
    \end{align}

    Fix $x\in \Omega$. For the first term, we have that
    \begin{align}
        \norm{\nabla^\star f}_2 L \norm{\nabla^\star f}_2 (x)
        &= \norm{\nabla^\star f(x)}_2 \sum_{i=1}^n q^\star_i(x) \tp{\norm{\nabla^\star f(x^i)}_2 - \norm{\nabla^\star f(x)}_2} \nonumber \\
        &= \sum_{i=1}^n q^\star_i(x) \tp{\norm{\nabla^\star f(x)}_2 \norm{\nabla^\star f(x^i)}_2 - \norm{\nabla^\star f(x)}_2^2}. \label{eq:norm-L-norm}
    \end{align}
    We then deduce from the Cauchy--Schwarz inequality that
    \begin{align}
        \norm{\nabla^\star f(x)}_2 \norm{\nabla^\star f(x^i)}_2
        &\ge \inner{\nabla^\star f(x)}{\nabla^\star f(x^i)} + 2\abs{\delta^\star_i} f(x) \abs{\delta^\star_i} f(x^i) \nonumber\\
        &= \inner{\nabla^\star f(x)}{\nabla^\star f(x^i)} + 2 \frac{q^\star_i(x^i)}{q^\star_i(x)} (\delta^\star_i f(x))^2, \label{eq:cauchy-schwarz}
    \end{align}
    where we crucially use the fact that $\delta^\star_i f(x) = q^\star_i(x) (f(x^i)-f(x))$ and $\delta^\star_i f(x^i) = q^\star_i(x^i) (f(x)-f(x^i))$ have opposite sign, or both vanish.
    Combining \eqref{eq:norm-L-norm} and \eqref{eq:cauchy-schwarz},
    \begin{align*}
        \norm{\nabla^\star f}_2 L \norm{\nabla^\star f}_2 (x)
        &\ge \sum_{i=1}^n q^\star_i(x) \tp{ \inner{\nabla^\star f(x)}{\nabla^\star f(x^i)} + 2 \frac{q^\star_i(x^i)}{q^\star_i(x)} (\delta^\star_i f(x))^2 - \inner{\nabla^\star f(x)}{\nabla^\star f(x)} } \\
        &= \sum_{i=1}^n 2 q^\star_i(x^i) (\delta^\star_i f(x))^2 + \sum_{i=1}^n q^\star_i(x) \inner{\nabla^\star f(x)}{\nabla^\star f(x^i) - \nabla^\star f(x)} \\
        &= \inner{\grad^\star f}{D \grad^\star f} (x) + \inner{\nabla^\star f}{L \nabla^\star f} (x).
    \end{align*}
    Plugging into \eqref{eq:first-rBE} and combining \Cref{lem:commutator}, we deduce
    \begin{align*}
        \Gamma^\star_2(f) - \Gamma(\sqrt{\Gamma^\star(f)}) 
        &\ge \inner{\grad^\star f}{D \grad^\star f} + \inner{\nabla^\star f}{(L \nabla^\star - \nabla^\star L) f} \\
        &= \inner{\grad^\star f}{\left( D + \diag\{M\*1\} - M^\top \right) \grad^\star f}
        \ge \rho \norm{\grad^\star f}_2^2 = \rho \cdot \Gamma^\star(f),
    \end{align*}
    where the last inequality holds for any constant $\rho \le \lambda_{\min} \left( D + \diag\{M\*1\} - \frac{1}{2}(M+M^\top) \right)$ uniformly at all points $x \in \Omega$.
\end{proof}

\subsection{Sufficient condition for CSGE} 
\label{sec:CSGE_boolean_hypercube}

\begin{theorem}[CSGE under canonical transition rates and weights]
\label{thm:CSGE_boolean_hypercube_w=q}
    Let $\mu$ be a distribution on $\Omega = \{\pm 1\}^n$ with full support. Let $M,D: \Omega \to \R^{n \times n}$ be two matrix fields defined by \eqref{eq:local-matrix} and \eqref{eq:def-D}, respectively. Let $M^\star := (\sup_x|M_{ij}(x)|)_{i,j \in [n]}$ and $D^\star = (\inf_x D_{ij}(x))_{i,j \in [n]}$ denote the entrywise uniform upper and lower bound on $\abs{M}$ and $D$, respectively. 
    Then, under the canonical transition rates and weights, for any $\tau \in [0,\infty]$, \ref{eq:CSGE} holds with constant
    \begin{align}\label{eq:CSGE_constant}
        \rho =
        \lambda_{\min}\left( D^\star - \diag\{M^\star\*1\} - \frac{1}{2}\left( M^\star + (M^\star)^\top \right) \right),
    \end{align}
    where $\lambda_{\min}(A)$ denotes the minimum eigenvalue of a symmetric matrix $A$.
\end{theorem}

\begin{proof}
    Fix an arbitrary function $f: \Omega \to \R$ and real numbers $\tau, t \ge 0$.
    Our goal is to show \ref{eq:CSGE} under the canonical rates and weights, namely, 
    \begin{align}\label{eq:goal-CSGE}
        \norm{P_\tau \abs{\nabla^\star}P_t f}_2 \leq e^{-\rho t} \norm{P_{\tau+t} \abs{\nabla^\star}f}_2. 
    \end{align}
    Define $F_s = e^{-2\rho s} \norm{P_{\tau+s}\abs{\nabla^\star} P_{t-s} f}_2^2$; hence, \eqref{eq:goal-CSGE} is equivalent to $F_0 \le F_t$.
    
    We first show that $F_s$ is differentiable almost everywhere on $[0,t]$.
    Any non-differentiability of $F_s$ comes from the absolute value function. 
    That is, $F_s$ is not differentiable at some $s\in [0,t]$ only if there exists $i\in [n]$ and $x\in \Omega$ such that $\delta_i^\star P_{t-s}f(x) = 0$.
    Fix $i\in [n]$ and $x\in \Omega$.
    For any $s,s_0\in [0,t]$, we have that
    $$
        P_{t-s}f(x) = \sum_{k=0}^{\infty} \frac{(L^k P_{t-s_0} f(x))}{k!} (s_0-s)^k
    $$
    and 
    $$ \sum_{k=0}^{\infty}\abs{ \frac{(L^k P_{t-s_0} f(x))}{k!} (s_0-s)^k} \leq \sum_{k=0}^{\infty} \frac{\norm{L}_2^k\abs{s_0-s}^k}{k!}\norm{P_{t-s_0}f}_{2} \leq e^{\norm{L}_2t}\norm{P_{t-s_0}f}_{2} < \infty,
    $$
    is absolutely convergent, where $\norm{L}_2$ is the operator norm of $L$.
    Hence $s\mapsto P_{t-s}f(x)$ is a real analytic function, and so is $s\mapsto \delta_i^\star P_{t-s}f(x)$.
    Then by the identity theorem of analytic functions, either $\delta_i^\star P_{t-s}f(x) \equiv 0$ for all $s\in [0,t]$, or the fiber $(\delta_i^\star P_{t-s}f(x))^{-1}(0)$ is a countable set in $[0,t]$.
    In the former case, the $i$-th coordinate does not affect the differentiability of $F_s$ at $x$.
    Therefore, we conclude that $F_s$ is differentiable almost everywhere on $[0,t]$.
    Each map $s\mapsto \delta_i^\star P_{t-s}f(x)$ is analytic and hence Lipschitz on $[0,t]$.
    Since the absolute-value map is also Lipschitz and $s\mapsto P_{\tau+s}$ is smooth on $[0,t]$, $F_s$ is Lipschitz, and hence absolutely continuous, on $[0,t]$.
    Therefore, it suffices to prove that $F'(s)\ge0$ for all differentiable $s \in [0,t]$.

    Fix $s \in [0,t]$ and $F_s$ is differentiable at $s$. To ease the notation, let $g = P_{t-s} f$.
    Taking the derivative of $F_s$ gives 
    \begin{align}
        \frac d{ds} F_s &= 2 e^{-2\rho s} \left( - \rho \norm{P_{\tau+s} |\grad^\star| g}_2^2 
        + \inner{P_{\tau+s} |\grad^\star| g}{\frac{d}{ds} P_{\tau+s} |\grad^\star| g}
        \right).
        \label{eq:CSGE_potential_function_derivative}
    \end{align}
    Since $P_{\tau+s} = e^{(\tau+s)L}$ and $g = P_{t-s} f = e^{(t-s)L} f$, we have that
    \begin{align}
    \label{eq:derivative-CSGE}
        \frac d{ds} P_{\tau+s} |\grad^\star| g
        &= P_{\tau+s} \left( L + \frac{d}{ds} \right) |\grad^\star| g 
        = P_{\tau+s} \left( L |\grad^\star| - \sgn(\grad^\star g) \odot \grad^\star L \right) g,
    \end{align}
    where $\sgn(\cdot)$ denotes the entrywise sign function, with $\sgn(0)=0$, and $\odot$ denotes the entrywise product of two vectors.
    To calculate this, we need an adapted version of \Cref{lem:commutator}, which is stated below and proved afterward.

    \begin{lemma}
    \label{lem:abs-commutator}
        Let $\mu$ be a distribution with full support on $\Omega = \{\pm 1\}^n$.
        Let $M,D: \Omega \to \R^{n \times n}$ be two matrix fields defined by \eqref{eq:local-matrix} and \eqref{eq:def-D}, respectively, and let $|M| := (|M_{ij}|)_{i,j \in [n]}$ be the entrywise absolute value of $M$.
        For any function $g: \Omega \to \R$, it holds
        \begin{align*}
            \left( L |\grad^\star| - \sgn(\grad^\star g) \odot \grad^\star L \right) g \ge \left( D + \diag\{M\*1\} - |M|^\top \right) |\grad^\star| g.
        \end{align*}
    \end{lemma}
    
    Combining \eqref{eq:derivative-CSGE} and \Cref{lem:abs-commutator}, 
    we deduce that
    \begin{align*}
        \inner{P_{\tau+s} |\grad^\star| g}{\frac{d}{ds} P_{\tau+s} |\grad^\star| g}
        &\geq \inner{P_{\tau+s} |\grad^\star| g}{P_{\tau+s} \left( D + \diag\{M\*1\} - |M|^\top \right) |\grad^\star| g}
        \\&\geq \inner{P_{\tau+s} |\grad^\star| g}{ \left( D^\star - \diag\{M^\star\*1\} -( M^\star)^\top \right) P_{\tau+s}|\grad^\star| g}.
    \end{align*}
    We conclude that 
    $$
        \frac{d}{ds}F_s \geq 2e^{-2\rho s}\tp{P_{\tau+s} |\grad^\star| g}^\top\tp{D^\star - \diag\{M^\star\*1\}-\frac{1}{2}(M^\star+(M^\star)^\top) - \rho I}\tp{P_{\tau+s} |\grad^\star| g} .
    $$
    Therefore, the coordinate-wise strong gradient estimate (\ref{eq:CSGE}) holds with constant $\rho = \lambda_{\min}(D^\star-\diag\{M^\star\*1\} - \frac{1}{2}(M^\star+(M^\star)^\top))$ at all $x \in \Omega$.
    The case $\tau=\infty$ follows by letting $\tau\to\infty$.
\end{proof}

\begin{proof}[Proof of \Cref{lem:abs-commutator}]
    Define $H_{ij}(x) = \frac{q_i^\star(x^j)}{q_i^\star(x)}$.
    By \Cref{fact:square}, we have that $H_{ij}(x) = \frac{q_i^\star(x^j)}{q_i^\star(x)} = \frac{q_j^\star(x^i)}{q_j^\star(x)} = H_{ji}(x)$.
    For notational convenience, we occasionally omit the dependence on $x$, writing, for example, $q_i^\star$ instead of $q_i^\star(x)$.
    For $i,j\in[n]$ with $i\neq j$, the following is deduced from $H_{ij}=H_{ji}$:
    \begin{align}
        (\delta_j^\star \abs{\delta_i^\star} - \sgn(\delta^\star_i g(x))\delta_i^\star \delta_j^\star) g(x)
        &=  q_j^\star q_i^\star\tp{H_{ij}\abs{g(x^{ij}) - g(x^j)} - \abs{g(x^i) - g(x)}} \nonumber
        \\&- \sgn(\delta^\star_i g) q_i^\star q_j^\star(H_{ji}(g(x^{ij})-g(x^i)) - g(x^j)+g(x))\nonumber
        \\&=q_i^\star q_j^\star H_{ij}\tp{\abs{g(x^{ij}) - g(x^j)} - \sgn(\delta^\star_i g)(g(x^{ij}) - g(x^j))} \label{eq:abs-commutator_second_term}
        \\&+q_i^\star q_j^\star\tp{-\frac{\abs{\delta^\star_i} g}{q_i^\star} - \sgn(\delta^\star_i g)\tp{H_{ij}(g(x^j) - g(x^i)) - \frac{\delta^\star_jg}{q_j^\star}}}. \label{eq:abs-commutator_eq}
    \end{align}
    Since $\abs{a}-\sgn(b)a \geq 0$ holds for any $a,b\in\R$, \eqref{eq:abs-commutator_second_term} is non-negative, resulting in the cancellation of the second derivative term in the commutator.
    Therefore, by \eqref{eq:abs-commutator_eq}, 
    \begin{align}
        (\delta_j^\star \abs{\delta_i^\star} - \sgn(\delta^\star_i g(x))\delta_i^\star \delta_j^\star) g(x)
        &\geq - q_j^\star\abs{\delta^\star_i} g - \sgn(\delta^\star_i g)\tp{H_{ij}q_i^\star q_j^\star(\frac{\delta_j^\star g}{q_j^\star} - \frac{\delta_i^\star g}{q_i^\star}) - q_i^\star\delta^\star_jg} \nonumber
        \\&= (H_{ij} - 1)q_j^\star\abs{\delta^\star_i}g - \sgn(\delta^\star_i g) \tp{H_{ij} - 1}q_i^\star\delta^\star_j g \nonumber
        \\&\geq (H_{ij} - 1)q_j^\star\abs{\delta^\star_i}g - \abs{H_{ji} - 1}q_i^\star\abs{\delta^\star_j}g. \label{eq:CSGE_commutator_term}
    \end{align}
    For the non-commutator term, we have that
    \begin{align} \label{eq:CSGE_non-commutator_term}
        (\delta_i^\star \abs{\delta_i^\star} - \sgn(\delta^\star_i g(x))\delta_i^\star \delta_i^\star)g(x) &= q_i^\star(x) \tp{\abs{\delta_i^\star}g(x^i) - \sgn(\delta^\star_i g(x))\delta_i^\star g(x^i)} \nonumber
        \\&= 2 q_i^\star(x) \abs{\delta_i^\star}g(x^i)
        = 2 q_i^\star(x^i) \abs{\delta_i^\star}g(x)
        = \frac 1{2q_i^\star(x)} \abs{\delta_i^\star}g(x),
    \end{align}
    where we use the fact that for any $i\in[n]$, $x^{ii} = x$ and $\sgn(\delta^\star_i g(x)) = \sgn(g(x^i) - g(x)) = - \sgn(g(x) - g(x^i)) = -\sgn(\delta^\star_i g(x^i))$.
    Combining \eqref{eq:CSGE_commutator_term} and \eqref{eq:CSGE_non-commutator_term}, we obtain the following inequality:
    \begin{align*}
        (L\abs{\delta_i^\star} - \sgn(\delta^\star_i g(x))\delta_i^\star L) g(x) &\geq \frac{1}{2q_i^\star}\abs{\delta^\star_i} g + \sum_{j\neq i} \tp{(H_{ij} - 1)q_j^\star\abs{\delta^\star_i} g - \abs{H_{ji} - 1}q_i^\star\abs{\delta^\star_j} g}
        \\&= (D_{ii} + (M_i)^\top \*1)\abs{\delta^\star_i} g - \inner{(\abs{M}^\top)_i}{\abs{\nabla^\star} g}, 
    \end{align*}
    where the last equality follows from $\tp{H_{ij}-1}q_j^\star = M_{ij}$.
    This is exactly
    \[
        \left( L |\grad^\star| - \sgn(\grad^\star g) \odot \grad^\star L \right) g \ge \left( D + \diag\{M\*1\} - |M|^\top \right) |\grad^\star| g. \qedhere
    \]
\end{proof}

\subsection{Dobrushin condition}
In this section, we state our isoperimetric inequalities under the conditions of the Dobrushin matrix and marginal bounds. 
In addition to the CSGE condition established above, the results of \Cref{sec:boolean} also require hypercontractivity.

It turns out that the smallest eigenvalue of $D^\star$ and the largest eigenvalue of $(M^\star)$--related matrices ($\diag\{M^\star\*1\}$ and $\frac 12(M^\star+(M^\star)^\top)$) can be controlled by suitable norms of Dobrushin influence matrix with some additional bounds associated with the canonical transition rates.
Consequently, a general and easy way to check the CSGE condition is to establish an upper bound on the norms of the Dobrushin matrix and check the marginal bounds of the distribution $\mu$.
The following is the definition of the Dobrushin influence matrix.
\begin{definition}[Dobrushin influence matrix]
    For a finite space $\+X$, the Dobrushin influence matrix of a distribution $\mu$ over $\+X^n$ is defined as follows:
    \begin{align}\label{eq:def_Dobrushin_matrix}
        R = (R_{ij})_{i,j\in[n]},\quad R_{ij} = \sup_{x,y,x_{-j}=y_{-j}}\TV{\mu_i(\cdot\mid x_{-i})}{\mu_i(\cdot\mid y_{-i})} .
    \end{align}
    We say a distribution satisfies the $L^2$--Dobrushin condition if $\norm{R(\mu)}_2 < 1$.
\end{definition}

\begin{lemma}[]\label{lem:Dobrushin_upper_bound_CSGE}
     If $w_i^\star(x) \in [b,1/b]$ for any $x\in \{\pm 1\}^n$ and $i\in[n]$, then under the canonical transition rates and weights, \ref{eq:CSGE} holds with constant $\rho = b - \frac{(1+b^2)^2}{4b^3}(\norm{R}_2 + \norm{R}_1)$.
\end{lemma}

\begin{proof}
    The entries of the Dobrushin influence matrix can be expressed in terms of the canonical transition rates as follows:
    \begin{align*}
        R_{ij} &= \sup_{x \in \Omega} \frac{4\abs{q_i^\star(x^j)^2-q_i^\star(x)^2}}{(1+4q_i^\star(x)^2)(1 + 4q_i^\star(x^j)^2)}, \quad i\neq j; \\
        R_{ii} &= 0.
    \end{align*}
    Then the Dobrushin matrix gives an entry-wise upper bound on $M^\star = (\sup_x \abs{M_{ij}(x)})_{i,j\in[n]}$.
    That is, for any $i, j\in[n]$ and $i\neq j$,
    \begin{align}\label{eq:Dobrushin_to_CSGE_constant}
        R_{ij} \geq \tp{\inf_{u,v\in [b,1/b]}\frac{u+v}{(1+u^2)(1+v^2)}}\sup_{x\in \Omega}2\abs{q_i^\star(x)-q_i^\star(x^j)} \geq \frac{4b^3}{(1+b^2)^2}M^\star_{ji} .
    \end{align}
    Since $D^\star_{ii} = \inf_x w_i^\star(x^i)$ for $i\in[n]$, $D^\star\geq_{\-{ew}} bI$.
    Therefore, by symmetry of matrices and \eqref{eq:Dobrushin_to_CSGE_constant}, we have that
    \begin{align*}
        & \lambda_{\min}\left( D^\star-\diag\{M^\star\*1\}-\frac{1}{2}\left(M^\star+(M^\star)^\top \right) \right) \\
        \geq{}& \lambda_{\min}(D^\star) - \lambda_{\max}(\diag\{M^\star\*1\}) - \lambda_{\max}\left( \frac{1}{2}\left( M^\star+(M^\star)^\top \right) \right) \\
        \geq{}& b - \frac{(1+b^2)^2}{4b^3}(\norm{R}_2 + \norm{R}_1)  \geq \rho.
    \end{align*}
    By \Cref{thm:CSGE_boolean_hypercube_w=q}, we complete the proof.
\end{proof}

The following theorem by Marton shows that for the discrete product space, the Dobrushin condition and the single-site conditional marginal lower bound are sufficient to obtain a positive \ref{eq:LSI} constant and hypercontractivity for the classical Glauber dynamics.
Via a simple comparison of the Dirichlet form, this also holds for the semigroup with the canonical rates.

\begin{theorem}[{\cite[Theorem 1.14 and Corollary 1.11]{marton2019logarithmic}}]\label{thm:marton_Dobrushin_to_LSI}
Let $\+X$ be a finite space.
Suppose the distribution $\mu$ on $\Omega = \+X^n$ satisfies the $L^2$--Dobrushin condition $\norm{R}_2 < 1$ and the one-site conditional marginal lower bound $\alpha = \min_{i\in[n],\,x\in\Omega} \mu_i(x_i\mid x_{-i})>0$, then for Gibbs sampler $G = \frac 1 n\sum_{i=1}^n K_i$ where $K_i(z\mid y) = \ind[z_{-i}=y_{-i}]\mu_i(z_i\mid z_{-i})$, the log-Sobolev inequality holds as follows,
$$
    \forall f\in \R^\Omega,\quad \frac 1n\Ent[\mu]{f^2} \leq \frac{2}{\alpha(1-\norm{R}_2)^2} \inner{f}{(I-G)f}_{\mu} .
$$
\end{theorem}
\begin{corollary}\label{cor:Dobrushin_to_LSI}
    Suppose $\norm{R}_2<1$ and $w_i^\star(x)\in [b,1/b]$ for any $x\in \{\pm 1\}^n$ and $i\in[n]$, then \ref{eq:LSI} holds with constant $\frac{b^2(1 - \norm{R}_2)^2}{1+b^2}$ with respect to the canonical transition rates $q^\star$.
\end{corollary}
\begin{proof}
    This immediately follows from
    \[
        \inner{f}{(I-G)f}_{\mu} = \E[\mu]{\frac 12\sum_{i=1}^n \frac{2q_i^\star(x)}{n(1/w_i^\star(x)+w_i^\star(x))}(f(x^i)-f(x))^2}
        \leq \frac 1n\E[\mu]{\Gamma(f,f)} . \qedhere
    \]
\end{proof}

For our canonical weights and transition rates, if $\{w_i^\star\}_{i\in[n]}$ is $b$-bounded, $\{q_i^\star\}_{i\in[n]}$ is $(b/2)$-bounded.
Therefore, by \Cref{thm:CSGE_boolean_hypercube_w=q}, \Cref{cor:Dobrushin_to_LSI} and \cref{sec:boolean}, we conclude the following theorem about isoperimetric inequalities.



\begin{theorem}\label{thm:Dobrushin_to_isoperimetric}
    Let $\mu$ be a distribution on $\{\pm 1\}^n$ with full support, and define the transition rates and weights as in~\eqref{eq:canonical-r-w}.
    Assume for any $x \in \cube$ and $i \in [n]$, $w_i^\star(x) \in [b,1/b]$ holds for constant $b\in (0,1]$.
    If the Dobrushin influence matrix satisfies $\norm{R}_1 \le C_0$ and $\norm{R}_2 \le 1-\delta$ for constants $C_0\geq 0$ and $\delta\in(0,1]$, then all the inequalities in~\Cref{sec:boolean} hold.
    Specifically, the local Bobkov inequality \eqref{eq:local_bobkov}, Talagrand's variance--surface-area inequality~\eqref{eq:improved_talagrand}, Talagrand's $L^1$--$L^2$ inequality \eqref{eq:l1_l2_talagrand_2}, the KKL inequality \eqref{eq:KKL_boolean_1} and the Eldan--Gross inequality \eqref{eq:eldan_gross} hold with constants depending on $C_0, \delta$ and $b$.
\end{theorem}

Our main result, \Cref{thm:informal}, is a consequence of \Cref{thm:Dobrushin_to_isoperimetric}. 
Note \Cref{thm:Dobrushin_to_isoperimetric} applies only to the canonical transition rates and weights.
The marginal boundedness assumption in \Cref{thm:informal} implies the canonical weights $\{w_i^\star\}_{i\in[n]}$ are bounded. This then allows us to compare, for Boolean $f$, the canonical influence quantities $I^\star(f)$, $A^\star(f)$ and $J^\star(f)$ with their Boolean counterparts $\Inf(f)$, $\SurfArea(f)$ and $\SqInf(f)$, up to multiplicative constants depending only on~$b$; see \Cref{ex:canonical}. These comparisons lead us to conclude \Cref{thm:informal} from \Cref{thm:Dobrushin_to_isoperimetric}.

\subsection{Application to the Ising model}
A distribution $\mu$ over $\{\pm 1\}^n$ is called an \emph{Ising model} with interaction matrix $J \in \R^{n\times n}$ and external field $h\in \R^n$ if 
$$
    \forall x\in \{\pm 1\}^n, \mu(x) \propto \exp(\frac 12 x^\top J x + h^\top x),
$$
where $J$ is a symmetric matrix.
Note that the diagonals of $J$ do not affect the distribution, hence we can assume that $J_{ii} = 0$ for $i\in[n]$.

Koehler, Lifshitz, Minzer, and Mossel showed that Talagrand's $L^1$--$L^2$ inequality \eqref{eq:l1_l2_talagrand_2} holds under marginal bounds and the log-Sobolev inequality \ref{eq:LSI} for the Gibbs sampler, where the underlying distribution is a bounded-degree Markov random field \cite{KoehlerLifshitzMinzerMossel23}.
In particular, this result can apply to the bounded-degree graphical Ising model in the high-temperature regime.
On the other hand, the Dobrushin condition~\cite{Dobrushin68} is a classical sufficient condition for rapid mixing of Glauber dynamics~\cite{DGJ09}.
In our work, we establish the isoperimetric inequalities through the Dobrushin condition and marginal bounds, without requiring bounded degrees.

In the following theorem, we establish a new criterion for the high-temperature Ising model to prove the isoperimetric inequalities under a Dobrushin-type condition by \cref{thm:Dobrushin_to_isoperimetric}.
\begin{theorem}\label{thm:Ising}
    Suppose that $\mu$ is an Ising model parameterized by zero external field $h = \*0$ and interaction matrix $J$ satisfying $\norm{\abs{J}}_2 < 1$ and $\norm{J}_1\leq C$ for some constant $C > 0$.
    Let the transition rates and weights be the canonical ones defined in \eqref{eq:canonical-r-w}.
    Then the local Bobkov inequality \eqref{eq:local_bobkov}, Talagrand's variance--surface-area inequality \eqref{eq:improved_talagrand}, Talagrand's $L^1$--$L^2$ inequality \eqref{eq:l1_l2_talagrand_2}, the KKL inequality \eqref{eq:KKL_boolean_1} and the Eldan--Gross inequality \eqref{eq:eldan_gross} hold with constants depending only on $\norm{\abs{J}}_2$ and $C$.
\end{theorem}
\begin{proof}
    Let $b = e^{-\norm{J}_1}$.
    For any $i,x$, $w_i^\star(x) = \exp(-x_iJ_i^\top x) \in [b,1/b]$.
    Fix $x\in \Omega$ and $i,j\in [n]$ with $i\neq j$.
    Define $\alpha = \sum_{k\neq i,j}J_{ik} x_k$ and $\beta = \abs{J_{ij}}$.
    We have that 
    \begin{align*}
        \TV{\mu_i(\cdot\mid x_{-i})}{\mu_i(\cdot\mid (x^j)_{-i})} &= \abs{\mu(x_i=+1\mid x_j=+1,x_{-i,j}) - \mu(x_i=+1\mid x_j=-1,x_{-i,j})}
        \\&= \frac 12\abs{\tanh(\alpha+\beta) - \tanh(\alpha-\beta)}
        = \frac{2\abs{\sinh(\beta)\cosh(\beta)}}{\cosh(2\alpha)+\cosh(2\beta)}
        \\&\leq \frac{2\abs{\sinh(\beta)\cosh(\beta)}}{1+\cosh(2\beta)} = \tanh(\beta) \leq \beta.
    \end{align*}
    Therefore, $R \leq_{\-{ew}} \abs{J}$, which implies $\norm{R}_2\leq \norm{\abs{J}}_2 < 1$ and $\norm{R}_1\leq \norm{\abs{J}}_1\leq C$.
    By \cref{thm:Dobrushin_to_isoperimetric}, all the isoperimetric inequalities in \cref{sec:boolean} hold.
\end{proof}
The symmetry of $J$ gives $\norm{\abs{J}}_2 \leq \sqrt{\norm{\abs{J}}_1\norm{\abs{J}}_{\infty}} = \norm{J}_1$.
Together with~\Cref{thm:Ising}, we prove~\Cref{cor:informal-ising}.
\section{Extension to Schreier Graphs}
\label{sec:Schreier}
In this section, we extend the framework from the Boolean hypercube $\Omega=\{\pm 1\}^n$ to Schreier graphs. This more general setting was also considered in~\cite{ODonnellWimmer13,ODonnellWimmerSharp13,CorderoErausquinLedoux12}, which established a KKL theorem for Schreier graphs with a generating set closed under conjugation, given a log-Sobolev constant.
Here, we derive an Eldan--Gross inequality in this setting under a slightly stronger assumption.

\subsection{Basic definitions}
Let $G$ be a group acting on a finite set $X$, and $U$ be a finite generating set for $G$ that is symmetric, i.e., $U^{-1} = U$. The \emph{Schreier graph} $\Sch(G,X,U)$ is a graph with vertex set $V = X$ and edge set $E = \{(x,gx) \mid x \in X, g \in U\}$. Let $\mathcal A$ be a finite index set with $d=|\mathcal A|$, and suppose that $U = \bigcup_{i\in\mathcal A} U_i$ where each $U_i$ is symmetric, i.e., $U_i^{-1} = U_i$.
In this work, we assume that $\{U_i\}_{i\in\mathcal A}$ is a partition of $U$.

Schreier graphs encompass a broad range of discrete structures, including product spaces such as $[q]^d$ and slices of the Boolean hypercube such as $\binom{[n]}{k}$.
\begin{example}\label{example:SGraph}
    Let $G = X = \Z_m^n$, $\mathcal A=[n]$, and $U_i = \{k e_i: k \in [m-1]\}$. The Schreier graph $\Sch(G,X,U)$ is a Hamming graph.
\end{example}

\begin{example}\label{example:k_slice}
    Let $G = S_n$, $X = \binom{[n]}{k}$, and $\+A=\{(i,j):1\le i<j\le n\}$. For each $(i,j)\in\mathcal A$, the set $U_{(i,j)} = \{(i\ j)\}$ consists of a single transposition. The Schreier graph $\Sch(G,X,U)$ is the exchange graph on the Boolean slice $\binom{[n]}{k}$.
\end{example}

Similar to the Boolean hypercube case, we define the carr\'e du champ operators, influences, and surface area on Schreier graphs.

\begin{definition}[$\Gamma$-calculus]
Let $\Sch(G,X,U)$ be a Schreier graph with $U = \bigcup_{i\in\mathcal A} U_i$. Given the transition rates $\{q_u: X \to \mathbb{R}_{+}\}_{u \in U}$ and weights $\{w_u: X \to \mathbb{R}_+\}_{u \in U}$, we define the operator $L_w$ and carr\'e du champ operators as follows:
\begin{itemize}
    \item The operator $L_w$ on functions $f: X \to \mathbb{R}$ satisfies
    \begin{align*}
        L_w f(x) = \sum_{u \in U} w_u(x) q_u(x) (f(x^u) - f(x)),
    \end{align*}
    where $x^u = ux$, and we omit the subscript $w$ if $w \equiv 1$.
    \item For any $f,g : X \to \mathbb{R}$, the carr\'e du champ operator $\Gamma_w$ is defined as
    \begin{align*}
        \Gamma_w (f,g) = \frac{1}{2}\tp{L_w (fg) - f L_w g - g L_w f}.
    \end{align*}
    By a straightforward calculation,
    \begin{align*}
        \Gamma_w(f,g)(x) = \frac{1}{2} \sum_{i \in \mathcal A} \sum_{u \in U_i} w_u(x) q_u(x) \tp{f(x^u) - f(x)}\tp{g(x^u)-g(x)}.
    \end{align*}
    Hence, the operator $\Gamma_{w,i}$ for the $i$-th ``coordinate'' is defined as
    \begin{align*}
    \Gamma_{w,i} (f,g) (x) = \frac{1}{2}\sum_{u \in U_i} w_u(x) q_u(x) \tp{f(x^u) - f(x)}\tp{g(x^u)-g(x)}.
    \end{align*}
\end{itemize}
\end{definition}

Similar to~\Cref{def:weighted-gradient}, we define the absolute weighted discrete gradient.
\begin{definition}
    For any function $f:X \to \R$, the absolute weighted discrete gradient is given by
    \begin{align*}
        \abs{\nabla_w} f = (\abs{\delta_{w,i}} f)_{i \in \mathcal A}, \quad \text{where $\abs{\delta_{w,i}} f = \sqrt{\Gamma_{w,i}(f)}$}.
    \end{align*}
\end{definition}

With the above definitions, we are now ready to define the influences and surface area.

\begin{definition}
    Let $f: X \to \mathbb{R}$ be a function.

    \begin{enumerate}[(1)]
        \item The \textit{total influence} of $f$ is defined as
        \begin{align*}
            I_w(f) := \E[\mu]{\Gamma_{w}(f)} = \E[\mu]{\norm{\abs{\nabla_w} f}^2_2}.
        \end{align*}
        Also, let $I_{w,i}(f) := \E[\mu]{\Gamma_{w,i}(f)}$ denote the influence of the $i$-th coordinate.

        \item The \textit{(Boolean) surface area} of $f$ is defined as
        \begin{align*}
            A_w(f) := \E[\mu]{\sqrt{\Gamma_wf}} = \E[\mu]{\norm{\abs{\nabla_w} f}_2}.
        \end{align*}

        \item The \textit{total squared influence} of $f$ is defined as
        \begin{align*}
            J_{w}(f) := \sum_{i\in\mathcal A} \tp{\E[\mu]{\abs{\delta_{w,i}} f}}^2 = \norm{\E[\mu]{\abs{\nabla_w} f}}^2_2.
        \end{align*}
    \end{enumerate}
\end{definition}

Finally, we define the gradient estimates for Schreier graphs.

\begin{definition}[Gradient estimates]
    Let $\mu$ be a distribution supported on $X$ and $\{q_u: X \to \R_+\}_{u\in U}$ be transition rates reversible for $\mu$.
    Suppose $\{w_u:X \to \R_+\}_{u\in U}$ is a collection of weight functions.
    \begin{itemize}
        \item We say that the \emph{Weak Gradient Estimate (WGE)} holds with constant $\rho_{\WGE} \in \mathbb{R}$ if for any function $f:X \to \R$ and $t\geq 0$,
        \begin{align}\label{eq:WGE-2}
            \Gamma_w(P_t f) \leq e^{-2\rho_{\WGE} t} P_t(\Gamma_w(f)) \quad\iff\quad \norm{\abs{\nabla_w} P_t f}^2_2 \leq e^{-2\rho_{\WGE} t} P_t(\norm{\abs{\nabla_w} f}^2_2); \tag{WGE}
        \end{align}
        \item We say that the \emph{Strong Gradient Estimate (SGE)} holds with constant $\rho_{\SGE} \in \mathbb{R}$ if for any function $f:X \to \R$ and $t\geq 0$,
        \begin{align}\label{eq:SGE-2}
            \sqrt{\Gamma_w(P_t f)} \leq e^{-\rho_{\SGE} t} P_t  \sqrt{\Gamma_w(f)} \quad\iff\quad \norm{\abs{\nabla_w} P_t f}_2 \leq e^{-\rho_{\SGE} t} P_t(\norm{\abs{\nabla_w} f}_2); \tag{SGE}
        \end{align}
        \item  We say that the \emph{Coordinate-wise Strong Gradient Estimate (CSGE)} holds with constant $\rho_{\CSGE} \in \mathbb{R}$ if for any $\tau \in [0,\infty]$, any function $f:X \to \R$, and $t\geq 0$,
        \begin{align}\label{eq:CSGE-2}
            \norm{P_\tau \abs{\nabla_w}P_t f}_2 \leq e^{-\rho_{\CSGE} t}\norm{P_{\tau+t} \abs{\nabla_w}f}_2. \tag{CSGE}
        \end{align}
        We record two particularly important subcases:
        \begin{itemize}
            \item ($\tau = 0$) For any function $f:X \to \R$ and $t\geq 0$,
            \begin{align}\label{eq:0-CSGE-2}
                \norm{\abs{\nabla_w} P_t f}_2 \leq e^{-\rho_{\CSGE} t}\norm{P_t \abs{\nabla_w}f}_2. \tag{$0$-CSGE}
            \end{align}

            \item ($\tau = \infty$) For any function $f:X \to \R$ and $t\geq 0$,
            \begin{align}\label{eq:infty-CSGE-2}
                \norm{\E[\mu]{\abs{\nabla_w}P_t f}}_2 \leq e^{-\rho_{\CSGE} t}\norm{\E[\mu]{\abs{\nabla_w}f}}_2. \tag{$\infty$-CSGE}
            \end{align}
        \end{itemize}
    \end{itemize}
\end{definition}

\subsection{Isoperimetric inequalities via gradient estimates}\label{sec:SGraph_iso_ineq}

In this section, we extend the isoperimetric inequalities from~\Cref{sec:boolean} to Schreier graphs. With the exception of Bobkov's inequality, the proofs do not rely on the explicit form of the operator $L_w$ and therefore carry over directly to this setting.
Complete proofs of the results in this subsection are deferred to~\Cref{app:schreier-proofs}.
\begin{theorem}[Local Bobkov inequality]\label{lem:local_bobkov-sch}
    Let $\Sch(G,X,U)$ be a connected Schreier graph. Let $\mu$ be a distribution supported on $X$ and $\{q_u: X \to \R_+\}_{u \in U}$ be transition rates reversible for $\mu$. 
    Assume $\{w_u:X \to \R_+\}_{u\in U}$ is a collection of weight functions that are $b$-bounded for some $b \in (0,1]$. 
    
    If the strong gradient estimate~\ref{eq:SGE-2} holds with constant $\rho \in \mathbb{R}$,
    then, for any $f:X \to [0,1]$ and $t\geq 0$, it holds
    \begin{align}\label{eq:local_bobkov-sch}
        \+I(P_t f)\leq P_t\tp{\sqrt{\+I(f)^2 + \gamma(t) \Gamma_w(f)}},
    \end{align}
    where $\gamma(t) =
    \begin{cases}
        \frac{2(1-e^{-2\rho t})}{b \rho}, & \rho \neq 0,\\
        \frac{4t}{b}, & \rho = 0,
    \end{cases}$ and $\+I:[0,1] \to \R_{\ge 0}$ denotes the Gaussian isoperimetric profile.

\end{theorem}

\begin{corollary}[Talagrand's Variance-Surface Area inequality]\label{lem:improved_talagrand_sch}
    Let $\Sch(G,X,U)$ be a connected Schreier graph. Let $\mu$ be a distribution supported on $X$ and $\{q_u:X \to \R_+\}_{u\in U}$ be transition rates reversible for $\mu$.
    Assume $\{w_u:X \to \R_+\}_{u\in U}$ is a collection of weight functions that are $b$-bounded for some $b \in (0,1]$.

    Suppose the following conditions hold:
    \begin{itemize}
        \item The strong gradient estimate~\ref{eq:SGE-2} holds with constant $\rho_{\SGE} \in \mathbb{R}$;
        \item The log-Sobolev inequality holds with constant $\rho_{\LSI} > 0$.
    \end{itemize}
    Then, for any $f:X \to \{0,1\}$, it holds
   \begin{align}\label{eq:improved_talagrand_sch}
        A_w(f) \gtrsim \frac{\sqrt{b} \rho_{\LSI}}{\sqrt{\rho_{\LSI} + \kappa}} \cdot \Var[\mu]{f}\sqrt{\log \frac e{\Var[\mu]{f}}},
   \end{align}
   where $\kappa = \max \set{0,-\rho_{\SGE}}$.
\end{corollary}

\begin{theorem}[Talagrand's $L^1$--$L^2$ inequality, \cite{CorderoErausquinLedoux12}]\label{lem:l1_l2_talagrand_sch}
     Let $\Sch(G,X,U)$ be a connected Schreier graph.
     Let $\mu$ be a distribution supported on $X$ and $\{q_u: X \to \R_+\}_{u \in U}$ be transition rates reversible for $\mu$.
    Assume~\ref{eq:0-CSGE-2} holds with constant $\rho_{\CSGE} \in \mathbb{R}$, the log-Sobolev inequality holds with constant $\rho_{\LSI} > 0$, and $\{w_u:X \to \mathbb{R}_+\}_{u \in U}$ is a collection of weight functions that are $b$-bounded for some $b \in (0,1]$. The following inequality holds for any $f:X \to \R$:
    \begin{align}\label{eq:l1_l2_talagrand_2_sch}
        \Var[\mu]{f} \lesssim \frac{\rho_{\LSI} + \kappa_0}{b \rho^2_{\LSI}} \cdot \sum_{i\in\mathcal A} \frac{I_{w,i} (f)}{1 + \log\tp{\sqrt{I_{w,i}(f)} / \E[\mu]{\abs{\delta_{w,i}} f}}},
    \end{align}
    where $\kappa_0 = \max\set{-\rho_{\CSGE},0}$.
    A summand is interpreted as zero when $I_{w,i}(f)=0$.

    Furthermore, we have that
    \begin{align}\label{eq:l1_l2_talagrand_3_sch}
        I(f) \gtrsim \frac{b^2 \rho^2_{\LSI}}{\rho_{\LSI} + \kappa_0}\cdot \Var[\mu]{f}\tp{1 + \log \frac{I_w(f)}{J_w(f)}} .
    \end{align}
\end{theorem}
\begin{corollary}[KKL]\label{cor:KKL_sch}
    Let $\Sch(G,X,U)$ be a connected Schreier graph.
    Let $\mu$ be a distribution supported on $X$ and $\{q_u: X \to \R_+\}_{u \in U}$ be transition rates reversible for $\mu$.
    Assume~\ref{eq:0-CSGE-2} and the log-Sobolev inequality hold with constants $\rho_{\CSGE} \in \mathbb{R}$ and $\rho_{\LSI} > 0$ respectively.
    Given a collection of weight functions $\{w_u:X \to \mathbb{R}_+\}_{u \in U}$, if both transition rates $\{q_u\}_{u\in U}$ and $\{w_u\}_{u\in U}$ are $b$-bounded for some $b \in (0,1]$, then for any Boolean function $f:X\to \{\pm 1\}$,
    \begin{align}\label{eq:KKL_sch_1}
        I_w(f) \geq \frac{C(b)\cdot \rho_{\LSI}^2}{\rho_{\LSI}+\kappa_0} \cdot \log (1/\+M_w(f))\cdot \Var[\mu]{f},
    \end{align}
    where $\+M_w(f) = \max_{i\in\mathcal A} I_{w,i}(f)$ and $\kappa_0 = \max\{-\rho_{\CSGE},0\}$.
    Hence, by applying the elementary implication $x/\log(1/x)\geq y\implies x\gtrsim y\log(1/y)$, the following bound holds,
    \begin{align}\label{eq:KKL_sch_2}
        \+M_w(f) \geq \Omega_{b}(1) \cdot \frac{\rho_{\LSI}^2}{(\rho_{\LSI}+\kappa_0)d}\log\tp{\frac{(\rho_{\LSI}+\kappa_0)d}{\rho_{\LSI}^2}} \cdot \Var[\mu]{f}.
    \end{align}
\end{corollary}

\begin{theorem}[Eldan--Gross inequality]\label{thm:eldan_gross_sch}
      Let $\Sch(G,X,U)$ be a connected Schreier graph.
     Let $\mu$ be a distribution supported on $X$ and $\{q_u:X \to \R_+\}_{u\in U}$ be transition rates reversible for $\mu$.
    Assume both~\ref{eq:0-CSGE-2} and~\ref{eq:infty-CSGE-2} hold with constant $\rho_{\CSGE} \in \mathbb{R}$, the log-Sobolev inequality holds with constant $\rho_{\LSI} > 0$, and $\{w_u:X \to \mathbb{R}_+\}_{u \in U}$ is a collection of weight functions that are $b$-bounded for some $b \in (0,1]$.
    For any Boolean function $f:X \to \{\pm 1\}$, it holds
   \begin{align}\label{eq:eldan_gross-sch}
        A_w(f)
        \gtrsim_{b,\rho_{\CSGE}}
        \frac{\sqrt{\rho_{\LSI}}\min\set{1,\rho_{\LSI}}}
        {\sqrt{1+\log_+\tp{1/\rho_{\LSI}}}}
        \cdot \Var[\mu]{f}\sqrt{\log\tp{1 + \frac e{J_w(f)}}}.
   \end{align}
\end{theorem}
\subsection{Establishing CSGE}
Similar to the Boolean hypercube case, we also define the canonical transition rates and weights for Schreier graphs as follows:
\begin{equation}\label{eq:canonical_SGraph}
    \forall u\in U, x\in X, q_u^\star(x) = \frac 1{2}\sqrt{\frac{\mu(x^u)}{\mu(x)}}\quad \text{and}\quad w_u^\star(x) = 2q_u^\star(x) = \sqrt{\frac{\mu(x^u)}{\mu(x)}}.
\end{equation}
The semigroup associated with the canonical transition rates is reversible with respect to $\mu$. Together with the canonical weights, these rates yield simple expressions for gradients and weighted carr\'e du champ operators.
We introduce the following notations for the canonical transition rates and weights, for any $u\in U$ and $i\in \+A$,
$$
    \delta^\star_u f(x) = q_u^\star(x)(f(x^u) - f(x)), 
    \quad \grad_i^\star f(x) = (\delta^\star_u f(x))_{u\in U_i};
$$
$$
    \quad \abs{\delta^\star_i} f(x) = \norm{\grad_i^\star f(x)}_2,
    \quad \abs{\nabla^\star} f(x) = (\abs{\delta_i^\star} f(x))_{i\in\mathcal A};
$$
$$
    \Gamma_i^\star f(x) = \Gamma_{w^\star,i} f(x)=  \sum_{u\in U_i} \delta_u^\star f(x)^2, \quad L f = \sum_{u\in U} \delta^\star_u f.
$$
Here we follow the notation in \Cref{sec:CSGE_boolean_hypercube} and use a superscript $\star$ to denote operators associated with canonical transition rates and weights.

We define $I: U\to \+A$ to map each operator $u$ to its index, i.e. $u\in U_{I(u)}$.
Accordingly, one natural commutativity assumption about the operators is that
\begin{equation}\label{eq:commutativity_assumption_CSGE}
    \forall u,v\in U, x\in X, I(u)\neq I(v) \implies x^{uv} = x^{vu},
\end{equation}
where we use $x^{uv}$ to denote $(x^u)^v$.
Under this assumption, we have a similar square property for the distribution $\mu$ over Schreier graphs.
\begin{fact}\label{fact:square_property_SGraph}
    Suppose \eqref{eq:commutativity_assumption_CSGE} holds for Schreier graph $\Sch(G,X,U)$.
    For any $u,v\in U$ with $I(u)\neq I(v)$, we have that
    $$
        q_u^\star(x)q_v^\star(x^u) = q_v^\star(x)q_u^\star(x^v)
    $$
    holds for any $x\in X$.
\end{fact}

In this section, we generalize the sufficient condition of \ref{eq:CSGE-2} from the Boolean hypercube to Schreier graphs.
By decomposing the relevant quantity into the local curvature term and other influence terms, we can derive a matrix-based lower bound for the optimal \ref{eq:CSGE-2} constant.
As in the Boolean hypercube setting, with the canonical transition rates and weights, the matrix can be bounded entrywise in terms of the Dobrushin influence matrix and marginal bounds.
In the more general Schreier graph setting, however, an additional lower bound on the local curvature is required.
\begin{theorem}\label{thm:SGraph_CSGE_general}
Suppose \eqref{eq:commutativity_assumption_CSGE} holds for Schreier graph $\Sch(G,X,U)$.
Let $\mu$ be a distribution on $X$ with full support. 
Let $M^\star,D^\star,D^\star_{\-{loc}}\in \R^{d \times d}$ be matrices satisfying that
\begin{equation}\label{eq:def_M_star_Sgraph}
    M^\star := \tp{\sup_x\norm{\tp{\frac{\ind[i\neq j]}{2}\tp{\sqrt{\frac{\mu(x^{uv})}{\mu(x^v)}}-\sqrt{\frac{\mu(x^u)}{\mu(x)}}}}_{u\in U_i, v\in U_j}}_2}_{i,j\in\mathcal A};
\end{equation}
\begin{equation}\label{eq:def_D_star_Sgraph}
    D^\star := \diag\left\{\inf_{x\in X,u\in U_i} \frac {1}{2}\sum_{v\in U\setminus U_i} \tp{\sqrt{\frac{\mu(x^{uv})}{\mu(x^u)}} - \sqrt{\frac{\mu(x^v)}{\mu(x)}}}\right\}_{i\in\mathcal A}.
\end{equation}
\begin{equation}\label{eq:def_D_star_loc_Sgraph}
    D^\star_{\-{loc}} := \diag\left\{\inf_{\substack{f:X\to \R,x\in X\\\Gamma_i^\star f(x) > 0}} \frac{\tp{\sqrt{\Gamma_i^\star f}L_{U_i}\sqrt{\Gamma_i^\star f} - \Gamma_i^\star(f,L_{U_i}f)}(x)}{\Gamma_{i}^\star f (x)} \right\}_{i\in\mathcal A} \quad\text{where}\quad L_S(f) = \sum_{u\in S} \delta_u^\star f.
\end{equation}
Then, under the canonical transition rates and weights, for any $\tau \in [0,\infty]$, \ref{eq:CSGE-2} holds with constant
\begin{align*}
    \rho =
    \lambda_{\min}\left( D_{\-{loc}}^\star + D^\star - \frac{1}{2}\left( M^\star + (M^\star)^\top \right) \right),
\end{align*}
where $\lambda_{\min}(A)$ denotes the minimum eigenvalue of a symmetric matrix $A$.
\end{theorem}

We first prove the following lemma, which provides a lower bound for the commutator-type term arising in the proof.
\begin{lemma}\label{lem:abs-commutator-Sgraph}
    Suppose \eqref{eq:commutativity_assumption_CSGE} holds for Schreier graph $\Sch(G,X,U)$.
    Let $M^\star,D^\star,D^\star_{\-{loc}}$ be matrices defined in \Cref{thm:SGraph_CSGE_general}.
    For any function $g: X \to \R$, it holds that for any $i\in\mathcal A$,
    \begin{align*}
        \abs{\delta^\star_i} g L \abs{\delta^\star_i} g - \inner{\grad_i^\star g}{\grad_i^\star L g} \geq ((D^\star_{\-{loc}})_{ii}+D^\star_{ii})\tp{\abs{\delta_i^\star} g}^2 - \sum_{j\neq i} M^\star_{ij} \abs{\delta_j^\star} g\abs{\delta^\star_i} g .
    \end{align*}
\end{lemma}
\begin{proof}[Proof of \Cref{lem:abs-commutator-Sgraph}]
    Fix $i\in\mathcal A$.
    Recall that the definition of $D_{\-{loc}}^\star$ in \eqref{eq:def_D_star_loc_Sgraph} gives
    \begin{equation}\label{eq:local_term_Sgraph}
        \abs{\delta_i^\star} g\cdot L_{U_i} \abs{\delta_i^\star} g - \inner{\nabla_i^\star g}{\nabla^\star_i L_{U_i}g} \geq (D_{\-{loc}}^\star)_{ii}\cdot(\abs{\delta_i^\star} g)^2.
    \end{equation}
    Let $Q = U\setminus U_i$.
    It is sufficient to lower bound the non-local term $\abs{\delta_i^\star} g\cdot L_{Q} \abs{\delta_i^\star} g - \inner{\nabla_i^\star g}{\nabla^\star_i L_Qg}$.
    
    For notational convenience, we occasionally omit the dependence on $x$.
    Define $H_{uv}(x) = \frac{q_u^\star(x^v)}{q_u^\star(x)}$.
    By \Cref{fact:square_property_SGraph}, for any $u,v\in U$ with $I(u)\neq I(v)$, $H_{uv}(x) = H_{vu}(x)$ holds.
    Fix $u,v\in U$ with $I(u)\neq I(v)$ and $x\in X$, since $x^{uv} = x^{vu}$ and $H_{uv}(x) = H_{vu}(x)$, we have that
    \begin{align}
        (\delta_v^\star \delta_u^\star - \delta_u^\star \delta_v^\star)g &= q_v^\star q_u^\star\tp{H_{uv}(g(x^{uv}) - g(x^v)) - (g(x^u)-g(x))} \nonumber
        \\&-q_v^\star q_u^\star\tp{H_{vu}(g(x^{vu}) - g(x^u)) - (g(x^v)-g(x))} \nonumber
        \\&= q_v^\star q_u^\star(H_{uv}-1)(g(x^u) - g(x^v)) \nonumber
        \\&= (H_{uv}-1)(q_v^\star \delta_u^\star - q_u^\star  \delta_v^\star) g. \label{eq:commutator-SGraph}
    \end{align}
    Hence, we have that for any $v\in Q$,
    \begin{align}
        |\delta_i^\star| g \cdot \delta_v^\star |\delta_i^\star| g - \inner{\nabla_i^\star g}{\nabla_i^\star \delta_v^\star g}
        &= q_v^\star (|\delta_i^\star| g\cdot  |\delta_i^\star| g(x^v) - \inner{\nabla_i^\star g}{\nabla_i^\star g(x^v)}) \nonumber
        \\&- \inner{\grad_i^\star g}{\grad_i^\star \delta_v^\star g  + q_v^\star(\grad_i^\star g - \grad_i^\star g(x^v))} \nonumber
        \\&\geq \sum_{u\in U_i} \delta_u^\star g \cdot \tp{\delta_v^\star \delta_u^\star - \delta_u^\star \delta_v^\star} g  \nonumber
        \\&\geq \sum_{u\in U_i} \delta^\star_u g \cdot (H_{uv}-1)(q_v^\star \delta_u^\star - q_u^\star  \delta_v^\star) g,
    \end{align}
    where the first inequality follows from the Cauchy--Schwarz inequality and the second inequality follows from \eqref{eq:commutator-SGraph}.
    Now it is sufficient to lower bound 
    \begin{align}
        &\abs{\delta^\star_i} g\cdot L_Q \abs{\delta^\star_i} g -  \inner{\grad_i^\star g}{\grad_i^\star L_Q g} \nonumber
        \\\geq{} &\sum_{v\in Q}\sum_{u\in U_i} \delta^\star_u g \cdot (H_{uv}-1)\tp{q_v^\star \delta_u^\star g - q_u^\star  \delta_v^\star g} \nonumber
        \\\geq{} & D^\star_{ii}(\abs{\delta_i^\star} g)^2 - \sum_{j\neq i}\sum_{v\in U_j}\sum_{u\in U_i}(q_u^\star(x^v)-q_u^\star) \delta^\star_u g \cdot \delta^\star_v g \nonumber
        \\\geq{} &D^\star_{ii}(\abs{\delta_i^\star} g)^2 - \sum_{j\neq i} M^\star_{ij} \abs{\delta_i^\star} g\cdot \abs{\delta_j^\star} g , \label{eq:non-local_term_Sgraph}
    \end{align}
    where $D^\star$ and $M^\star$ are defined in \eqref{eq:def_D_star_Sgraph} and \eqref{eq:def_M_star_Sgraph}, respectively.
    We conclude this lemma by \eqref{eq:local_term_Sgraph} and \eqref{eq:non-local_term_Sgraph}.
\end{proof}

Now we prove \Cref{thm:SGraph_CSGE_general}.
\begin{proof}[Proof of \Cref{thm:SGraph_CSGE_general}]
    Fix $f: X \to \R$ and $\tau, t\geq 0$.
    Define $F_s := e^{-2\rho s}\norm{P_{\tau+s}\abs{\nabla^\star}P_{t-s}f}_2^2$.
    In the same way as in the proof of \Cref{thm:CSGE_boolean_hypercube_w=q}, it is sufficient to prove that $\frac d{ds}F_s \geq 0$ for any $s\in [0,t]$ at which $F_s$ is differentiable.
    Note that $F_s$ is absolutely continuous and differentiable almost everywhere on $[0,t]$, since the differences $P_{t-s}f(x^u)-P_{t-s}f(x)$ are analytic in $s$.

    Fix $s\in(0,t)$ where all maps $s\mapsto|\delta_i^\star|P_{t-s}f(x)$ are differentiable, and set $g=P_{t-s}f$.
    For each $i,x$, either $|\delta_i^\star|g(x)>0$, or
    $$
        |\delta_i^\star|g(x)=0\quad\Longleftrightarrow\quad g(x^u)=g(x)\quad\text{for all }u\in U_i,
    $$
    in which case  we set $\inner{\nabla_i^\star g}{\nabla_i^\star Lg}(x)/|\delta_i^\star|g(x) = 0$ (the convention $0/0:=0$).
    We have the following equation for the derivative of $F_s$,
    \begin{align}
        \frac d{ds} F_s &= 2 e^{-2\rho s} \left( - \rho \norm{P_{\tau+s} |\grad^\star| g}_2^2 
        + \inner{P_{\tau+s} |\grad^\star| g}{\frac{d}{ds} P_{\tau+s} |\grad^\star| g}
        \right) ;
        \label{eq:CSGE_potential_function_derivative-SGraph}
    \end{align}
    \begin{align}\nonumber
        \frac d{ds} P_{\tau+s} |\delta^\star_i| g
        &= P_{\tau+s} \left( L \abs{\delta^\star_i} g -  \frac{\inner{\grad_i^\star g}{\grad_i^\star L g}}{\abs{\delta^\star_i} g}\right).
    \end{align}
    By \Cref{lem:abs-commutator-Sgraph}, we have the entrywise inequality
    \begin{align}
    \label{eq:derivative-CSGE-Sgraph}
        \frac d{ds} P_{\tau+s} |\nabla^\star| g
        &\geq \left( D^\star_{\-{loc}}+D^\star - M^\star \right) P_{\tau+s}\abs{\nabla^\star}g .
    \end{align}
    Therefore, \eqref{eq:CSGE_potential_function_derivative-SGraph} and \eqref{eq:derivative-CSGE-Sgraph}, after symmetrizing the resulting quadratic form, show that
    $$
        \frac d{ds} F_s \geq 2 e^{-2\rho s} (P_{\tau+s}\abs{\nabla^\star}g)^\top  \left( D^\star_{\-{loc}}+D^\star - \frac{M^\star + (M^\star)^\top}2 - \rho I \right)(P_{\tau+s}\abs{\nabla^\star}g) .
    $$
    Hence, the coordinate-wise strong gradient estimate (\ref{eq:CSGE-2}) holds with constant $\rho \leq \lambda_{\min}(D_{\-{loc}}^\star+D^\star - \frac{1}{2}(M^\star+(M^\star)^\top))$ at all $x \in X$.
    The case $\tau=\infty$ follows by letting $\tau\to\infty$.
\end{proof}

\subsection{Application to \texorpdfstring{$k$-slice}{k-slice}}
In this section, we take the uniform distribution over the $k$-slice as an example and derive isoperimetric inequalities under canonical transition rates and weights through the log-Sobolev inequality and \ref{eq:CSGE-2} condition.

    Let $\mu$ be the uniform distribution on $\binom{[n]}{k}$.
    The transition rates and weights are given by $q_u^\star \equiv 1/2$ and $w_u^\star \equiv 1$.
    For each transposition $u$, let $L_u f(x)=\frac12(f(x^u)-f(x))=\delta_u^\star f(x)$. The set $U$ is closed under conjugation, so $L_uP_t=P_tL_u$. Therefore,
    \begin{align*}
        \abs{\delta^\star_u P_t f} = \abs{L_u P_t f} = \abs{P_t L_u f} \le P_t \abs{\delta^\star_u f}.
    \end{align*}
    Applying $P_\tau$ and taking the Euclidean norm over $u\in U$ proves~\ref{eq:CSGE-2} with constant $0$; the case $\tau=\infty$ follows by taking $\mu$-expectations. Minkowski and Jensen also give~\ref{eq:SGE-2} and~\ref{eq:WGE-2}, respectively.

\begin{theorem}[\cite{lee1998logarithmic}]
    For $n\geq 3$ and $1\leq k\leq n-1$, the log-Sobolev constant of uniform distribution with transition rates $q_u \equiv 1/2$ over Schreier graph on $\binom{[n]}{k}$ is as follows,
    $$
        \rho = \Theta\tp{\frac n{\log(1/\nu(k))}},
    $$
    where $\nu(k) = \nu(n-k) = k(n-k)/\binom{n}2$.
\end{theorem}

For a Boolean function $f:\binom{[n]}{k}\to\{\pm1\}$, under canonical transition rates and weights we have:
\begin{itemize}
    \item [(1)] Influence of transposition $u$ and total influence:
    $$
        \mathsf{Inf}_{\mu,u}(f) := I^\star_u(f) = \Pr[\mu]{f(x)\neq f(x^{u})},
    $$
    $$
        \Inf_\mu(f) :=  I^\star(f) = \sum_{u\in U}\Pr[\mu]{f(x)\neq f(x^u)} .
    $$
    \item [(2)] Surface area:
    $$
        \mathsf{SA}_\mu(f) := A^\star(f) =  \E[\mu]{\sqrt{s(f)}}
    $$
    where $s(f)$ is the sensitivity of $f$; namely $s(f)(x) = |\{u\in U: f(x^u)\neq f(x)\}|$.
    \item [(3)] Total squared influence:
    $$
        \mathsf{SqInf}_\mu(f):= J^\star(f) =  \sum_{u \in U} \tp{\Pr[\mu]{f(x) \neq f(x^u)}}^2.
    $$
\end{itemize}
Combining the preceding gradient estimates and log-Sobolev bound with \Cref{sec:SGraph_iso_ineq} gives the following inequalities for the uniform $k$-slice.
\begin{theorem}\label{thm:k_slice_isoperimetric}
    Let $\-{Sch}(G,X,U)$ be the exchange graph on the Boolean slice $\binom{[n]}{k}$ as in \Cref{example:k_slice} and $\mu$ be the uniform distribution over $\binom{[n]}k$.
    Fix $\eps \in (0,1/2]$ and suppose $\eps n \leq k \leq (1-\eps)n$.
    Then for any Boolean function $f: \binom{[n]}k\to \{\pm 1\}$, all the following isoperimetric inequalities hold with constants depending only on~$\eps$:
    \begin{itemize}
        \item [(1)] KKL inequality:
        $$
            \max_{u\in U}\Inf_{\mu,u}(f) \gtrsim \frac{\log n}{n}\Var[\mu]{f} ;
        $$
        \item [(2)] Talagrand's $L^1$--$L^2$ inequality: Let $L_u f(x) = \frac{1}{2} (f(x^u) - f(x))$, and it holds
        $$
            \Var[\mu]{f} \lesssim \frac{1}{n} \sum_{u \in U}\frac{\norm{L_u f}_2^2}{1 + \log(\norm{L_u f}_2/\norm{L_u f}_1)} \lesssim \frac{1}{n} \sum_{u\in U} \frac{\Inf_{\mu,u}(f)}{1 + \log(1/\Inf_{\mu,u}(f))} ;
        $$
        \item [(3)] Talagrand's variance--surface-area inequality:
        $$
            \mathsf{SA}_\mu(f) \gtrsim \sqrt n\cdot \Var[\mu]{f}\sqrt{\log \frac e{\Var[\mu]{f}}} ;
        $$
        \item [(4)] Eldan--Gross inequality:
        $$
            \mathsf{SA}_\mu(f) \gtrsim \sqrt n \cdot \Var[\mu]{f}\sqrt{\log\tp{1+\frac e{\SqInf_\mu(f)}}}.
        $$
    \end{itemize}
\end{theorem}

The first two inequalities recover the known KKL and Talagrand $L^1$--$L^2$ bounds for balanced slices~\cite{ODonnellWimmer13,ODonnellWimmerSharp13,CorderoErausquinLedoux12}, while the last two give Talagrand's variance--surface-area and Eldan--Gross inequalities in this setting.


\subsection{Application to hypergrid}

Similarly in \Cref{sec:CSGE_boolean_hypercube}, we can bound matrix $M^\star$ from above and $D^\star$ from below by marginal bounds and the Dobrushin matrix we defined in \eqref{eq:def_Dobrushin_matrix}.

\begin{lemma}\label{lem:hyper_grid_Dobrushin_to_AB_in_CSGE}
    Suppose the Schreier graph is the Hamming graph and the transition rates and weights are the canonical transition rates and weights we define in \eqref{eq:canonical_SGraph}.
    If $w_u^\star(x) \in [b,1/b]$ for $x\in X$ and $u\in U$, then
    we have that
    $$
        D^\star \geq_{\-{ew}}-\frac{(m+2b)(b^2+m-1)}{4b^3}\cdot  \diag\{R^\top \*1\}
        \qquad \text{and} \qquad M^\star \leq_{\-{ew}}\frac{(m-1)(b^2+m-1)(b+1)}{4b^3} \cdot R .
    $$
\end{lemma}
\begin{proof}
    Fix $x\in X$, $i\neq j$ and $k_1\in [m-1]$.
    WLOG, assume $x_j = 0$.
    Let $\alpha$ be the marginal distribution $\mu_j(\cdot\mid x_{-j})$, that is, $\alpha_k = \mu(X_j = k\mid X_{-j}=x_{-j})$.
    In the same way, let $\beta$ be the marginal distribution $\mu_j(\cdot\mid x_{-j}^{k_1 e_i})$.
    The marginal lower bound is as follows,
    $$
        \forall k\in \{0,1,\ldots,m-1\}, \quad\min\{\beta_k,\alpha_k\} \geq \frac{1}{1 + (m-1)/b^2}.
    $$
    Therefore, for any $k_2\in [m-1]$,
    \begin{align}
        2\abs{q_{k_2 e_j}^\star(x) - q_{k_2 e_j}^\star(x^{k_1 e_i})}
        &= \abs{\sqrt{\frac{\alpha_{k_2}}{\alpha_0}} - \sqrt{\frac{\beta_{k_2}}{\beta_0}}} 
        \leq \abs{\sqrt{\frac{\alpha_{k_2}}{\alpha_0}} - \sqrt{\frac{\beta_{k_2}}{\alpha_0}}} + \abs{\sqrt{\frac{\beta_{k_2}}{\alpha_0}} - \sqrt{\frac{\beta_{k_2}}{\beta_0}}} \nonumber
        \\&\leq \frac{\abs{\alpha_{k_2} - \beta_{k_2}}}{\sqrt{\alpha_{0}\alpha_{k_2}}+\sqrt{\alpha_0\beta_{k_2}}} + \sqrt{\frac{\beta_{k_2}}{\beta_0}}\frac{\abs{\alpha_0 - \beta_0}}{\alpha_0 + \sqrt{\alpha_0\beta_0}} \nonumber
        \\&\leq \frac{b^2+m-1}{2b^2}\abs{\alpha_{k_2} - \beta_{k_2}} + \frac{b^2+m-1}{2b^3}\abs{\alpha_0 - \beta_0} , \label{eq:Dobrushin_upper_bound_AB}
    \end{align}
    where the last inequality follows from the marginal bounds.
    Note that in Hamming graph, 
    $$
        \forall i\in[n], D^\star_{ii} = \inf_{x\in X, k_1\in [m-1]} \sum_{j\neq i}\sum_{k_2\in [m-1]}(q_{k_2 e_j}^\star(x^{k_1 e_i}) - q_{k_2 e_j}^\star(x)).
    $$
    Therefore, given $x,k_1$, we have that by \eqref{eq:Dobrushin_upper_bound_AB},
    \begin{multline*}
        \sum_{j\neq i}\sum_{k_2}\abs{q_{k_2 e_j}^\star(x) - q_{k_2 e_j}^\star(x^{k_1 e_i})} \\
        \leq \frac 12\sum_{j\neq i} \tp{\frac{b^2+m-1}{b^2} + \frac{m(b^2+m-1)}{2b^3}}\TV{\mu_j(\cdot\mid x_{-j})}{\mu_j(\cdot\mid x_{-j}^{k_1 e_i})}.
    \end{multline*}
    This implies that $D^\star \geq_{\-{ew}} -\frac{(m+2b)(b^2+m-1)}{4b^3}\diag\{R^\top \*1\}$.
    On the other hand, \eqref{eq:Dobrushin_upper_bound_AB} shows that
    for any $x \in X$,
    $$
        \tp{\abs{\sqrt{\frac{\mu(x^{k_1e_i})}{\mu(x)}}-\sqrt{\frac{\mu(x^{k_1e_i,k_2e_j})}{\mu(x^{k_2e_j})}}}}_{k_1,k_2\in [m-1]} \leq_{\-{ew}} \frac{(b^2+m-1)(b+1)}{2b^3}\cdot R_{ij} J,
    $$
    where $J$ is the all-ones matrix.
    Hence, $M^\star \leq_{\-{ew}} \frac{(m-1)(b^2+m-1)(b+1)}{4b^3} \cdot R$.
\end{proof}

\begin{lemma}[Local curvature lower bound]\label{lem:local_curvature}
    Suppose the Schreier graph is the Hamming graph and the transition rates and weights are the canonical transition rates and weights we define in \eqref{eq:canonical_SGraph}.
    If $w_u^\star(x) \in [b,1/b]$ for $x\in X$ and $u\in U$, then we have that
    $$
        \inf_{\substack{f,x:\,\Gamma_i^\star f(x)>0}} \frac{\sqrt{\Gamma_{i}^\star f}L_{U_i}\sqrt{\Gamma_{i}^\star f} - \Gamma_{i}^\star (f,L_{U_i}f)}{\Gamma_{i}^\star f} \geq \frac12\left[b(\sqrt{m-1}+1)-(m-1)(\frac 1b-b)\right].
    $$
\end{lemma}
\begin{proof}
    \allowdisplaybreaks
    Fix $i\in[n]$ and $f\in \R^X$.
    Set $\widehat L_{U_i}:=2L_{U_i}$ and $\widehat\Gamma_i:=4\Gamma_i^\star$. Then
    \begin{align*}
        \frac{\sqrt{\widehat\Gamma_i f}\,\widehat L_{U_i}\sqrt{\widehat\Gamma_i f}-\widehat\Gamma_i(f,\widehat L_{U_i}f)}{\widehat\Gamma_i f}
        =2\frac{\sqrt{\Gamma_i^\star f}\,L_{U_i}\sqrt{\Gamma_i^\star f}-\Gamma_i^\star(f,L_{U_i}f)}{\Gamma_i^\star f}.
    \end{align*}
    We prove the corresponding lower bound for the rescaled operators and, for notational convenience, write them again as $L_{U_i}$ and $\Gamma_i^\star$ below.
    Set $q_0=q_m=1$, $d_0=d_m=0$, $q_k=w_{k e_i}^\star=2q_{k e_i}^\star$ and $d_k = f(x^{k e_i}) - f(x)$ for $k\in[m-1]$.
    Write $x^k = x^{ke_i}$, $\*q = (q_k)_{k=1}^{m-1}$ and $\*q\*d = (q_kd_k)_{k=1}^{m-1}$.
    Also, we omit $(x)$ for simplicity.
    The canonical definition gives $q_\ell/q_k=w_{(\ell-k)e_i}^\star(x^{k e_i})\in[b,1/b]$ for all distinct $k,\ell\in\{0\}\cup[m-1]$, where the indices are interpreted modulo $m$.
    For the first part $\sqrt{\Gamma_{i}^\star f}L_{U_i}\sqrt{\Gamma_{i}^\star f}$, we have the following,
    \begin{align}
        \sum_{k=1}^{m-1} q_k\sqrt{\frac{\Gamma_{i}^\star f(x^k)}{\Gamma_{i}^\star f}} 
        &= \sum_{k=1}^{m-1}q_k\sqrt{\frac{\sum_{\ell = 1}^{m-1}q_\ell(x^k)^2(f(x^{(k+\ell)e_i}) - f(x^{k e_i}))^2}{\sum_{\ell=1}^{m-1}q_\ell^2d_\ell^2}} \nonumber
        \\& = \sum_{k=1}^{m-1}\sqrt{\tp{\sum_{\ell = 1}^{m}q_\ell^2(d_\ell - d_k)^2} / \tp{\sum_{\ell=1}^{m-1}q_\ell^2d_\ell^2}} &&(\text{Rearrange}) \nonumber
        \\&\geq \sqrt{\tp{\sum_{\ell = 1}^m q_\ell^2\tp{\sum_{k=1}^{m-1} \tp{d_\ell - d_k}}^2 }/ \tp{\sum_{\ell=1}^{m-1}q_\ell^2d_\ell^2}} &&(\text{Triangle ineq.}) \nonumber
        \\&\geq b\sqrt{\tp{\sum_{\ell = 1}^m \tp{\sum_{k=1}^{m-1} \tp{d_\ell - d_k}}^2 }/ \tp{\sum_{\ell=1}^{m-1}d_\ell^2}} &&(\forall \ell,k,q_\ell/q_k \in [b,1/b]) \nonumber
        \\&= b\sqrt{(m-1)^2 - (m-2)\tp{\sum_{k=1}^{m-1}d_k}^2/\tp{\sum_{k=1}^{m-1}d_k^2}} \nonumber
        \\&\geq b\sqrt{(m-1)}. &&(\text{Cauchy-–Schwarz}) \label{eq:Local_curvature_part1}
    \end{align}
    Define $Q = \sum_{k=1}^{m-1} q_k$.
    Therefore, 
    $$
        \frac{\sqrt{\Gamma_{i}^\star f}L_{U_i}\sqrt{\Gamma_{i}^\star f}}{\Gamma_{i}^\star f} = \sum_{k=1}^{m-1}  q_k\tp{\sqrt{\frac{\Gamma_{i}^\star f(x^k)}{\Gamma_{i}^\star f}} - 1}\geq b\sqrt{m-1} - Q.
    $$
    For the second part $-\Gamma_{i}^\star(f,L_{U_i} f)$, its lower bound is as follows,
    \begin{align}
        -\Gamma_{i}^\star(f,L_{U_i} f)
        &= -\sum_{k=1}^{m-1} q_k^2 d_k \sum_{\ell=1}^{m-1} \tp{q_\ell(x^k)(f(x^{(k+\ell)e_i}) - f(x^k)) - q_\ell d_\ell} \nonumber
        \\&=\sum_{k=1}^{m-1}\sum_{\ell=1}^{m-1}q_k^2q_\ell d_kd_\ell
        + \sum_{k=1}^{m-1}\sum_{\ell=1}^m q_kq_\ell d_k^2
        - \sum_{k=1}^{m-1}\sum_{\ell=1}^m q_kq_\ell d_k d_\ell \nonumber
        \\&= (\*q\*d)^\top (\*q\*1^\top + (Q+1)\cdot \diag\{\*q\}^{-1} - \*1\*1^\top)  (\*q\*d) \nonumber
        \\&\geq \norm{\*q\*d}_2^2 \lambda_{\min}\tp{\frac{(\*q-\*1)\*1^\top + \*1(\*q-\*1)^\top}{2} + (Q+1)\cdot \diag\{\*q\}^{-1}} \label{eq:Local_curvature_part2} .
    \end{align}
    Since $\lambda_{\min}(A+B)\geq \lambda_{\min}(A)+\lambda_{\min}(B)$ holds for symmetric matrices $A,B$, we have that,
    $$
        \lambda_{\min}\tp{\frac{(\*q-\*1)\*1^\top + \*1(\*q-\*1)^\top}{2} + (Q+1)\cdot \diag\{\*q\}^{-1}} \geq \frac{\inner{\*q-\*1}{\*1}-\sqrt{m-1}\norm{\*q-\*1}_2}{2} + b(Q+1) .
    $$
    Since $\Gamma_{i}^\star f = \norm{\*q\*d}_2^2$, by \eqref{eq:Local_curvature_part1} and \eqref{eq:Local_curvature_part2}, we conclude that,
    \begin{align*}
        &\frac{\sqrt{\Gamma_{i}^\star f}L_{U_i}\sqrt{\Gamma_{i}^\star f} - \Gamma_{i}^\star(f,L_{U_i} f)}{\Gamma_{i}^\star f}
        \\\geq{} & b\sqrt{m-1} - Q +  \frac{Q-(m-1)-\sqrt{m-1}\norm{\*q-\*1}_2}{2} + b(Q+1)
        \\={} & b\sqrt{m-1} -\frac{m-1-2b}{2} + (b-\frac{1}{2})Q - \frac{\sqrt{m-1}\norm{\*q-\*1}_2}{2} 
        \\\geq{} & b\sqrt{m-1} -\frac{m-1-2b}{2} + (m-1)\min\{b^2-\frac{b}{2},1-\frac{1}{2b}\} - \frac{(m-1)(1-b)}{2b} 
        \\\geq{} & b\sqrt{m-1} -\frac{m-1-2b}{2} + (m-1)(b-\frac 1{2b}) - \frac{(m-1)(1-b)}{2b} 
        \\={} & b(\sqrt{m-1}+1)-(m-1)(\frac 1b-b),
    \end{align*}
    where the second inequality follows from $Q \in [(m-1)b,\frac {m-1}b]$ and $\norm{\*q-\*1}_2 \leq \sqrt{m-1}(1-b)/b$.
    Returning to the original normalization divides the curvature quotient by two, which proves the claim.
\end{proof}

\Cref{lem:hyper_grid_Dobrushin_to_AB_in_CSGE} and \Cref{lem:local_curvature} then yield the following lower bound on the \ref{eq:CSGE-2} constant.
\begin{lemma}\label{lem:hyper_grid_Dobrushin_to_CSGE}
    Suppose the Schreier graph is the Hamming graph and the transition rates and weights are the canonical transition rates and weights we define in \eqref{eq:canonical_SGraph}.
    If $w_u^\star(x) \in [b,1/b]$ for $x\in X$ and $u\in U$, then under the canonical transition rates and weights, for any $\tau \in [0,\infty]$, \ref{eq:CSGE-2} holds with constant
    \begin{align*}
        \rho^\star = \frac12\left[b(\sqrt{m-1}+1)-(m-1)(\frac 1b-b)\right] - \frac{(b^2+m-1)}{4b^3}((m+2b)\norm{R}_1 + (m-1)(b+1)\norm{R}_2).
    \end{align*}
\end{lemma}
\begin{proof}
    By \Cref{lem:local_curvature}, we have $D_{\-{loc}}^\star \geq_{\-{ew}}\frac12\tp{ b(\sqrt{m-1}+1)-(m-1)(\frac 1b-b)} I$.
    Hence, we deduce from \Cref{lem:hyper_grid_Dobrushin_to_AB_in_CSGE} that
    $$
        \lambda_{\min}(D^\star) \geq -\frac{(m+2b)(b^2+m-1)}{4b^3}\cdot \norm{R}_1 ;
    $$
    $$
        \lambda_{\max}\tp{\frac{M^\star + (M^\star)^\top}{2}} \leq \frac{(m-1)(b^2+m-1)(b+1)}{4b^3} \norm{R}_2 .
    $$
    Since $\lambda_{\min}(A+B) \geq \lambda_{\min}(A) + \lambda_{\min}(B)$ holds for any symmetric matrices $A,B$, by \Cref{thm:SGraph_CSGE_general},  \ref{eq:CSGE-2} holds with constant $\rho^\star$.
\end{proof}

Recall that \Cref{thm:marton_Dobrushin_to_LSI} gives an \ref{eq:LSI} constant for the Gibbs sampler in terms of the Dobrushin matrix.
For hypergrids, the discrepancy between the Gibbs generator and our generator can be bounded by one-site marginal bounds.
Therefore, together with the \ref{eq:CSGE-2} lower bound in \Cref{lem:hyper_grid_Dobrushin_to_CSGE}, we conclude the following theorem: under suitable Dobrushin conditions and marginal bounds, the isoperimetric inequalities in \Cref{sec:SGraph_iso_ineq} hold.
Specifically, for any Boolean function $f: \Omega \to \{\pm 1\}$, we have that under canonical transition rates and weights,
\begin{itemize}
    \item [(1)] Influence of $i$-th coordinate and total influence: 
    $$
        \mathsf{Inf}_{\mu,i}(f) := I^\star_i(f) = \sum_{k\in [m-1]}\Pr[\mu]{f(x)\neq f(x^{k e_i})},
    $$
    $$
        \Inf_\mu(f) :=  I^\star(f) = \sum_{i=1}^n I_i^\star(f) = \sum_{u\in U}\Pr[\mu]{f(x)\neq f(x^u)} .
    $$
    \item [(2)] Surface area:
    $$
        b\cdot  A^\star(f)\leq \mathsf{SA}_\mu(f) := \E[\mu]{\sqrt{s(f)}} \leq \frac 1b\cdot A^\star(f)
    $$
    where $s(f)$ is the sensitivity of $f$; namely $s(f)(x) = |\{u\in U: f(x^u)\neq f(x)\}|$.
    \item [(3)] Total squared influence:
    $$
        b^2 \cdot J^\star(f) \leq \mathsf{SqInf}_\mu(f):=  \sum_{i=1}^n \tp{\E[x \sim \mu]{\sqrt{\sum_{k\in[m-1]}\*1[f(x) \neq f(x^{k e_i})]}}}^2 \leq \frac 1{b^2}\cdot J^\star(f).
    $$
\end{itemize}
Here we assume that the canonical weights are $b$-marginally bounded.
\begin{theorem}\label{thm:hyper_grid_Dobrushin_to_iso}
    Let $m\geq 2$ be an integer, $b\in (0,1]$, $C_0\geq 0$ and $\delta\in(0,1]$.
    Consider the Hamming graph on $[m]^n$ and $\mu$ be a distribution on $[m]^n$ with full support.
    Take the transition rates to be the canonical transition rates we define in \eqref{eq:canonical_SGraph}.
    Suppose we have the following conditions:
    \begin{itemize}
        \item [(i)] (Marginal boundedness) For any $u\in U$ and $x\in [m]^n$,
        $$
            b\leq \sqrt{\frac{\mu(x^u)}{\mu(x)}} \leq \frac 1b;
        $$
        \item [(ii)] (Dobrushin condition) The Dobrushin matrix $R$ of $\mu$ satisfies $\norm{R}_1 \leq C_0$ and $\norm{R}_2\leq 1-\delta$.
    \end{itemize}
    Then for any Boolean function $f: [m]^n \to \{\pm 1\}$, all the following isoperimetric inequalities hold with a constant $C>0$ depending only on $b,C_0,\delta,m$.
    \begin{itemize}
        \item [(1)] KKL inequality:
        $$
            \max_{i\in [n]}\Inf_{\mu,i}(f) \geq C\, \frac{\log n}n \Var[\mu]{f} ;
        $$
        \item [(2)] Talagrand's $L^1$--$L^2$ inequality: Let $L_i f(x)=\sum_{k=1}^{m-1}q_{k e_i}^\star(x)(f(x^{k e_i})-f(x))$, and it holds
        $$
            C\,\Var[\mu]{f} \leq \sum_{i=1}^n \frac{\norm{L_i f}_{2,\mu}^2}{1 + \log(\norm{L_i f}_{2,\mu}/\norm{L_i f}_{1,\mu})};
        $$
        \item [(3)] Talagrand's variance--surface-area inequality:
        $$
            \mathsf{SA}_\mu(f) \geq C\, \Var[\mu]{f}\sqrt{\log \frac e{\Var[\mu]{f}}} ;
        $$
        \item [(4)] Eldan--Gross inequality:
        $$
            \mathsf{SA}_\mu(f) \geq C\, \Var[\mu]{f}\sqrt{\log\tp{1+\frac e{\SqInf_\mu(f)}}}.
        $$
    \end{itemize}
    Finally, the local Bobkov inequality \eqref{eq:local_bobkov-sch} holds as well.
\end{theorem}
\begin{proof}
    By \Cref{thm:marton_Dobrushin_to_LSI}, since
    $$
        \inner{f}{(I-G)f}_\mu = \E[\mu]{\frac 1{2n}\sum_{i=1}^n  \sum_{k=1}^{m-1} \frac{w_{k e_i}^\star(x)^2}{1 + \sum_{\ell=1}^{m-1}w_{\ell e_i}^\star(x)^2}\tp{f(x^{k e_i}) - f(x)}^2}
        \leq \frac 2{nmb^2}\E[\mu]{\Gamma(f,f)},
    $$
    the \ref{eq:LSI} holds with constant $\frac{mb^4(1-\norm{R}_2)^2}{2(m-1+b^2)}$.
    Hence, hypercontractivity holds.
    By \Cref{lem:hyper_grid_Dobrushin_to_CSGE}, the \ref{eq:CSGE-2} constant is independent of $n$.

    Therefore, by the results in \Cref{sec:SGraph_iso_ineq}, all the isoperimetric inequalities hold. Since the distribution is $b$-marginally bounded, the multiplicative discrepancies between $I^\star(f)$, $A^\star(f)$, and $J^\star(f)$ and their respective counterparts $\Inf_\mu(f)$, $\mathsf{SA}_\mu(f)$, and $\mathsf{SqInf}_\mu(f)$ are bounded by constants depending only on $b$. Therefore, statements 1--4 follow for Boolean functions.
    For (2), set $g_i:=\abs{\delta_i^\star}f$. Since $f$ is Boolean, all nonzero differences $f(x^u)-f(x)$ have the same sign at each fixed $x$, so
    $$
        \abs{L_i f}=\sum_{u\in U_i}\abs{\delta_u^\star f}\geq \sqrt{\sum_{u\in U_i}(\delta_u^\star f)^2}=g_i.
    $$
    The map $(s,t)\mapsto s^2/(1+\log(s/t))$ is increasing in each argument on $0<t\leq s$. Thus, replacing $g_i$ by $\abs{L_i f}$ in the norm form of~\eqref{eq:l1_l2_talagrand_2_sch} gives (2).
\end{proof}

\bibliographystyle{alpha}
\bibliography{ref.bib}

@article{marton2019logarithmic,
  title={Logarithmic {Sobolev} inequalities in discrete product spaces},
  author={Marton, Katalin},
  journal={Combinatorics, Probability and Computing},
  volume={28},
  number={6},
  pages={919--935},
  year={2019},
  publisher={Cambridge University Press},
  doi={10.1017/S0963548319000099}
}

@article{lee1998logarithmic,
  title={Logarithmic {Sobolev} inequality for some models of random walks},
  author={Lee, Tzong-Yow and Yau, Horng-Tzer},
  journal={The Annals of Probability},
  volume={26},
  number={4},
  pages={1855--1873},
  year={1998},
  publisher={Institute of Mathematical Statistics},
  doi={10.1214/aop/1022855885}
}

@article{BHLLMY15,
  title={{L}i--{Y}au inequality on graphs},
  author={Bauer, Frank and Horn, Paul and Lin, Yong and Lippner, Gabor and Mangoubi, Dan and Yau, Shing-Tung},
  journal={Journal of Differential Geometry},
  volume={99},
  number={3},
  pages={359--405},
  year={2015},
  doi={10.4310/jdg/1424880980}
}

@article{FS18,
  title={Curvature and transport inequalities for {M}arkov chains in discrete spaces},
  author={Fathi, Max and Shu, Yan},
  journal={Bernoulli},
  volume={24},
  number={1},
  pages={672--698},
  year={2018},
  doi={10.3150/16-BEJ892}
}

@inproceedings{KKL88,
  author    = {Jeff Kahn and Gil Kalai and Nathan Linial},
  title     = {The Influence of Variables on {Boolean} Functions},
  booktitle = {Proceedings of the 29th Annual Symposium on Foundations of Computer Science},
  pages     = {68--80},
  publisher = {IEEE},
  year      = {1988},
  doi       = {10.1109/SFCS.1988.21923}
}

@article{Talagrand93,
  author  = {Michel Talagrand},
  title   = {Isoperimetry, Logarithmic {Sobolev} Inequalities on the Discrete Cube, and {Margulis}' Graph Connectivity Theorem},
  journal = {Geometric and Functional Analysis},
  volume  = {3},
  number  = {3},
  pages   = {295--314},
  year    = {1993},
  doi     = {10.1007/BF01895691}
}

@article{Talagrand94,
  author  = {Michel Talagrand},
  title   = {On {Russo}'s Approximate Zero-One Law},
  journal = {The Annals of Probability},
  volume  = {22},
  number  = {3},
  pages   = {1576--1587},
  year    = {1994},
  doi     = {10.1214/aop/1176988612}
}

@article{Talagrand97,
  author  = {Michel Talagrand},
  title   = {On Boundaries and Influences},
  journal = {Combinatorica},
  volume  = {17},
  number  = {2},
  pages   = {275--285},
  year    = {1997},
  doi     = {10.1007/BF01200910}
}

@article{EldanGross22,
  author  = {Ronen Eldan and Renan Gross},
  title   = {Concentration on the {Boolean} Hypercube via Pathwise Stochastic Analysis},
  journal = {Inventiones Mathematicae},
  volume  = {230},
  number  = {3},
  pages   = {935--994},
  year    = {2022},
  doi     = {10.1007/s00222-022-01135-8},
  url     = {https://arxiv.org/abs/1909.12067}
}

@article{EldanKindlerLifshitzMinzer25,
  author  = {Ronen Eldan and Guy Kindler and Noam Lifshitz and Dor Minzer},
  title   = {Isoperimetric Inequalities Made Simpler},
  journal = {Discrete Analysis},
  year    = {2025},
  note    = {Paper No.~7, 23 pp.},
  doi     = {10.19086/da.142095},
  url     = {https://arxiv.org/abs/2204.06686}
}

@article{IvanisviliZhang26,
  author  = {Paata Ivanisvili and Haonan Zhang},
  title   = {On the {Eldan--Gross} Inequality},
  journal = {Journal of Functional Analysis},
  volume  = {290},
  number  = {4},
  pages   = {111255},
  year    = {2026},
  doi     = {10.1016/j.jfa.2025.111255},
  eprint  = {2407.17864},
  url     = {https://arxiv.org/abs/2407.17864}
}

@book{ODonnell14,
  author    = {Ryan O'Donnell},
  title     = {Analysis of {Boolean} Functions},
  publisher = {Cambridge University Press},
  year      = {2014},
  doi       = {10.1017/CBO9781139814782}
}

@incollection{BenOrLinial90,
  author    = {Michael Ben-Or and Nathan Linial},
  title     = {Collective Coin Flipping},
  booktitle = {Randomness and Computation},
  editor    = {Silvio Micali},
  series    = {Advances in Computing Research},
  volume    = {5},
  pages     = {91--115},
  publisher = {JAI Press},
  year      = {1989}
}

@article{BKKKL92,
  author  = {Jean Bourgain and Jeff Kahn and Gil Kalai and Yitzhak Katznelson and Nathan Linial},
  title   = {The Influence of Variables in Product Spaces},
  journal = {Israel Journal of Mathematics},
  volume  = {77},
  number  = {1--2},
  pages   = {55--64},
  year    = {1992},
  doi     = {10.1007/BF02808010}
}

@article{FriedgutKalai96,
  author  = {Ehud Friedgut and Gil Kalai},
  title   = {Every Monotone Graph Property Has a Sharp Threshold},
  journal = {Proceedings of the American Mathematical Society},
  volume  = {124},
  number  = {10},
  pages   = {2993--3002},
  year    = {1996},
  doi     = {10.1090/S0002-9939-96-03732-X}
}

@article{KellerMosselSen12,
  author  = {Nathan Keller and Elchanan Mossel and Arnab Sen},
  title   = {Geometric Influences},
  journal = {The Annals of Probability},
  volume  = {40},
  number  = {3},
  pages   = {1135--1166},
  year    = {2012},
  doi     = {10.1214/11-AOP643}
}

@article{BobkovGotze99,
  author  = {Sergey G. Bobkov and Friedrich G{\"o}tze},
  title   = {Discrete Isoperimetric and {P}oincar{\'e}-Type Inequalities},
  journal = {Probability Theory and Related Fields},
  volume  = {114},
  number  = {2},
  pages   = {245--277},
  year    = {1999},
  doi     = {10.1007/s004400050225}
}

@incollection{CorderoErausquinLedoux12,
  author    = {Dario {Cordero-Erausquin} and Michel Ledoux},
  title     = {Hypercontractive Measures, {Talagrand}'s Inequality, and Influences},
  booktitle = {Geometric Aspects of Functional Analysis},
  series    = {Lecture Notes in Mathematics},
  volume    = {2050},
  pages     = {169--189},
  publisher = {Springer},
  year      = {2012},
  doi       = {10.1007/978-3-642-29849-3_10},
  url       = {https://arxiv.org/abs/1105.4533}
}

@article{ODonnellWimmer13,
  author  = {Ryan O'Donnell and Karl Wimmer},
  title   = {{KKL}, {Kruskal--Katona}, and Monotone Nets},
  journal = {SIAM Journal on Computing},
  volume  = {42},
  number  = {6},
  pages   = {2375--2399},
  year    = {2013},
  doi     = {10.1137/100787325},
  url     = {https://www.cs.cmu.edu/~odonnell/papers/kkl-kk.pdf}
}

@article{ODonnellWimmerSharp13,
  author  = {Ryan O'Donnell and Karl Wimmer},
  title   = {Sharpness of {KKL} on {Schreier} Graphs},
  journal = {Electronic Communications in Probability},
  volume  = {18},
  pages   = {1--12},
  year    = {2013},
  note    = {Paper no. 18},
  doi     = {10.1214/ECP.v18-1961},
  url     = {https://www.cs.cmu.edu/~odonnell/papers/kkl-sharp.pdf}
}

@article{GrahamGrimmett06,
  author  = {Benjamin T. Graham and Geoffrey R. Grimmett},
  title   = {Influence and Sharp-Threshold Theorems for Monotonic Measures},
  journal = {The Annals of Probability},
  volume  = {34},
  number  = {5},
  pages   = {1726--1745},
  year    = {2006},
  doi     = {10.1214/009117906000000278}
}

@article{DuminilCopinRaoufiTassion19,
  author  = {Hugo Duminil-Copin and Aran Raoufi and Vincent Tassion},
  title   = {Sharp Phase Transition for the Random-Cluster and {Potts} Models via Decision Trees},
  journal = {Annals of Mathematics},
  volume  = {189},
  number  = {1},
  pages   = {75--99},
  year    = {2019},
  doi     = {10.4007/annals.2019.189.1.2}
}

@inproceedings{KoehlerLifshitzMinzerMossel23,
  author    = {Frederic Koehler and Noam Lifshitz and Dor Minzer and Elchanan Mossel},
  title     = {Influences in Mixing Measures},
  booktitle = {Proceedings of the 56th Annual {ACM} Symposium on Theory of Computing},
  pages     = {527--536},
  publisher = {Association for Computing Machinery},
  year      = {2024},
  doi       = {10.1145/3618260.3649731},
  url       = {https://arxiv.org/abs/2307.07625}
}

@article{DGJ09,
  title   = {Matrix norms and rapid mixing for spin systems},
  author  = {Dyer, Martin and Goldberg, Leslie Ann and Jerrum, Mark},
  journal = {The Annals of Applied Probability},
  volume  = {19},
  number  = {1},
  pages   = {71--107},
  year    = {2009},
  doi     = {10.1214/08-AAP532}
}

@inproceedings{Hayes06,
  title={A simple condition implying rapid mixing of single-site dynamics on spin systems},
  author={Hayes, Thomas P},
  booktitle={2006 47th Annual IEEE Symposium on Foundations of Computer Science},
  pages={39--46},
  year={2006},
  organization={IEEE}
}

@article{SS20,
  title   = {Logarithmic {Sobolev} inequalities for finite spin systems and applications},
  author  = {Sambale, Holger and Sinulis, Arthur},
  journal = {Bernoulli},
  volume  = {26},
  number  = {3},
  pages   = {1863--1890},
  year    = {2020},
  doi     = {10.3150/19-BEJ1172}
}

@article{CGMV26,
  title={Learning {AC}$^0$ Under Graphical Models},
  author={Chandrasekaran, Gautam and Gaitonde, Jason and Moitra, Ankur and Vasilyan, Arsen},
  journal={arXiv preprint arXiv:2604.06109},
  year={2026}
}

@article{FYYZ26,
  title={Learning {AC}$^0$ under Locally Sampleable Graphical Models},
  author={Feng, Weiming and Yang, Xiongxin and Yu, Yixiao and Zhang, Yiyao},
  journal={arXiv preprint arXiv:2607.08303},
  year={2026}
}

@incollection{BakryEmery85,
  author    = {Bakry, Dominique and {\'E}mery, Michel},
  title     = {Diffusions hypercontractives},
  booktitle = {S\'eminaire de Probabilit\'es XIX, 1983/84},
  editor    = {Az\'ema, Jacques and Yor, Marc},
  series    = {Lecture Notes in Mathematics},
  volume    = {1123},
  pages     = {177--206},
  publisher = {Springer},
  address   = {Berlin},
  year      = {1985},
  doi       = {10.1007/BFb0075847}
}

@article{BakryLedoux96,
  author  = {Dominique Bakry and Michel Ledoux},
  title   = {{L\'evy--Gromov}'s Isoperimetric Inequality for an Infinite Dimensional Diffusion Generator},
  journal = {Inventiones Mathematicae},
  volume  = {123},
  number  = {2},
  pages   = {259--281},
  year    = {1996},
  doi     = {10.1007/s002220050026}
}

@book{BGL-book,
  author    = {Bakry, Dominique and Gentil, Ivan and Ledoux, Michel},
  title     = {Analysis and Geometry of {Markov} Diffusion Operators},
  series    = {Grundlehren der Mathematischen Wissenschaften},
  volume    = {348},
  publisher = {Springer},
  address   = {Cham},
  year      = {2014},
  doi       = {10.1007/978-3-319-00227-9}
}

@misc{Salez-survey,
  author        = {Salez, Justin},
  title         = {Modern Aspects of {Markov} Chains: Entropy, Curvature and the Cutoff Phenomenon},
  year          = {2025},
  eprint        = {2508.21055},
  archiveprefix = {arXiv},
  primaryclass  = {math.PR},
  note          = {arXiv:2508.21055},
  doi           = {10.48550/arXiv.2508.21055},
  url           = {https://arxiv.org/abs/2508.21055}
}

@book{NR-book,
  editor    = {Najman, Laurent and Romon, Pascal},
  title     = {Modern Approaches to Discrete Curvature},
  series    = {Lecture Notes in Mathematics},
  volume    = {2184},
  publisher = {Springer},
  address   = {Cham},
  year      = {2017},
  doi       = {10.1007/978-3-319-58002-9}
}

@article{Ollivier09,
  author  = {Ollivier, Yann},
  title   = {Ricci Curvature of {Markov} Chains on Metric Spaces},
  journal = {Journal of Functional Analysis},
  volume  = {256},
  number  = {3},
  pages   = {810--864},
  year    = {2009},
  doi     = {10.1016/j.jfa.2008.11.001},
  url     = {https://arxiv.org/abs/math/0701886}
}

@article{ErbarMaas12,
  author  = {Erbar, Matthias and Maas, Jan},
  title   = {Ricci Curvature of Finite {Markov} Chains via Convexity of the Entropy},
  journal = {Archive for Rational Mechanics and Analysis},
  volume  = {206},
  number  = {3},
  pages   = {997--1038},
  year    = {2012},
  doi     = {10.1007/s00205-012-0554-z},
  url     = {https://arxiv.org/abs/1111.2687}
}

@phdthesis{Kamtue-thesis,
  title={Discrete curvatures motivated from Riemannian geometry and optimal transport: Bonnet-Myers-type diameter bounds and rigidity},
  author={Kamtue, Supanat},
  year={2021},
  school={Durham University},
  type={Doctoral thesis},
  url={https://etheses.dur.ac.uk/14124/}
}

@article{CKLP22,
  title={{{Bakry--\'Emery}} Curvature on Graphs as an Eigenvalue Problem},
  author={Cushing, David and Kamtue, Supanat and Liu, Shiping and Peyerimhoff, Norbert},
  journal={Calculus of Variations and Partial Differential Equations},
  volume={61},
  number={2},
  year={2022},
  note={Article 62},
  publisher={Springer},
  doi={10.1007/s00526-021-02179-z}
}

@article{GG11,
  title={Sharp thresholds for the random-cluster and {Ising} models},
  author={Graham, Benjamin T. and Grimmett, Geoffrey R.},
  journal={The Annals of Applied Probability},
  volume={21},
  number={1},
  pages={240--265},
  year={2011},
  doi={10.1214/10-AAP693}
}

@misc{RvH-remarks,
  title={Ramon van {H}andel's remarks on the discrete cube},
  author={Rosenthal, Gregory},
  year={2020},
  note={Lecture notes, 13 October 2020. \url{https://www.cs.toronto.edu/~rosenthal/RvH_discrete_cube.pdf}}
}

@article{BIM23,
  title={On sharp isoperimetric inequalities on the hypercube},
  author={Beltran, David and Ivanisvili, Paata and Madrid, Jos{\'e}},
  journal={Transactions of the American Mathematical Society},
  year={2026},
  doi={10.1090/tran/9743}
}

@article{KLMP23,
  title={{Bakry--\'Emery} calculus for entropic curvature, new diameter estimates, and spectral gaps},
  author={Kamtue, Supanat and Liu, Shiping and M{\"u}nch, Florentin and Peyerimhoff, Norbert},
  journal={arXiv preprint arXiv:2312.09686},
  year={2023}
}

@article{KLMP24,
  title={Entropic curvature not comparable to other curvatures--or is it?},
  author={Kamtue, Supanat and Liu, Shiping and M{\"u}nch, Florentin and Peyerimhoff, Norbert},
  journal={arXiv preprint arXiv:2404.04581},
  year={2024}
}

@article{GrimmettJansonNorris16,
  author  = {Geoffrey R. Grimmett and Svante Janson and James R. Norris},
  title   = {Influence in Product Spaces},
  journal = {Advances in Applied Probability},
  volume  = {48},
  number  = {A},
  pages   = {145--152},
  year    = {2016},
  doi     = {10.1017/apr.2016.46}
}

@article{DurcikIvanisviliRoosXie26,
  author        = {Polona Durcik and Paata Ivanisvili and Joris Roos and Xinyuan Xie},
  title         = {Sharp Isoperimetric Inequalities on the {Hamming} Cube {II}: The Critical Exponent},
  journal       = {arXiv preprint arXiv:2602.20462},
  year          = {2026},
  eprint        = {2602.20462},
  archivePrefix = {arXiv},
  primaryClass  = {math.CA},
  url           = {https://arxiv.org/abs/2602.20462}
}

@article{EHMT17,
  author  = {Matthias Erbar and Christopher Henderson and Georg Menz and Prasad Tetali},
  title   = {Ricci Curvature Bounds for Weakly Interacting {Markov} Chains},
  journal = {Electronic Journal of Probability},
  volume  = {22},
  pages   = {Paper No. 40, 1--23},
  year    = {2017},
  doi     = {10.1214/17-EJP49}
}

@article{Pedrotti25,
  author  = {Francesco Pedrotti},
  title   = {Contractive Coupling Rates and Curvature Lower Bounds for {Markov} Chains},
  journal = {The Annals of Applied Probability},
  volume  = {35},
  number  = {1},
  pages   = {196--250},
  year    = {2025},
  doi     = {10.1214/24-AAP2113}
}

@article{IvanisviliXieZhang26,
  author        = {Paata Ivanisvili and Xinyuan Xie and Haonan Zhang},
  title         = {A {Beckmann} Boundary Form of {Talagrand}'s Conjecture on the Discrete Cube},
  journal       = {arXiv preprint arXiv:2606.31961},
  year          = {2026},
  eprint        = {2606.31961},
  archivePrefix = {arXiv},
  primaryClass  = {math.CA},
  url           = {https://arxiv.org/abs/2606.31961}
}

@article{Dobrushin68,
  author  = {Dobrushin, Roland L'vovich},
  title   = {The Description of a Random Field by Means of
             Conditional Probabilities and Conditions of Its Regularity},
  journal = {Theory of Probability \& Its Applications},
  volume  = {13},
  number  = {2},
  pages   = {197--224},
  year    = {1968},
  doi     = {10.1137/1113026}
}

@article{Hop26,
  title={Toward a {KKL} Theorem for any {HDX}},
  author={Hopkins, Max},
  journal={arXiv preprint arXiv:2606.29449},
  year={2026}
}

@article{Bob97,
  title={An isoperimetric inequality on the discrete cube, and an elementary proof of the isoperimetric inequality in {G}auss space},
  author={Bobkov, Sergey G},
  journal={The Annals of Probability},
  volume={25},
  number={1},
  pages={206--214},
  year={1997},
  publisher={Institute of Mathematical Statistics},
  doi={10.1214/aop/1024404285}
}

\appendix

\section{Proofs of Isoperimetric Inequalities on Schreier Graphs}
\label{app:schreier-proofs}

\begin{proof}[Proof of~\Cref{lem:local_bobkov-sch}]
    The assertion is trivial for constant functions, so we assume throughout that \(f\) is nonconstant.
    Fix $f : X \to [0,1]$ and $t \ge 0$.
    By applying the argument to $f_\varepsilon:=\varepsilon+(1-2\varepsilon)f$, where $0<\varepsilon<1/2$, and then letting $\varepsilon\downarrow0$, we may assume that $f$ takes values in $(0,1)$; in particular, $F>0$ below.
    Let $\Psi(s) = P_s \tp{\sqrt{\+I(P_{t-s} f)^2 + \gamma(s) \Gamma_w(P_{t-s} f)}}$. It suffices to prove that $\Psi$ is monotone increasing. By a standard calculation, it holds
    \begin{align}\label{eq:dPsi-1-sch-app}
    \frac{\dif}{\dif s} \Psi(s) &= P_s L\tp{\sqrt{\+I(P_{t-s} f)^2 + \gamma(s) \Gamma_w(P_{t-s} f)}} + P_s \frac{\dif}{\dif s}\sqrt{\+I(P_{t-s} f)^2 + \gamma(s) \Gamma_w(P_{t-s} f)}
    \end{align}
    Let $F = \+I(P_{t-s} f)^2 + \gamma(s) \Gamma_w(P_{t-s} f)$. By chain rule, it holds
    \begin{align}\label{eq:dPsi-2-sch-app}
    \frac{\dif}{\dif s} \sqrt{F} &= \frac{1}{2\sqrt{F}}\cdot\tp{- 2 \+I(P_{t-s} f) \+I'(P_{t-s} f) \cdot P_{t-s} L f + \gamma'(s) \Gamma_w(P_{t-s} f) - 2\gamma(s) \Gamma_w(P_{t-s} f, L P_{t-s} f)}
    \end{align}
    By~\Cref{lem:chain-rule}, it holds
    \begin{align}\label{eq:dPsi-3-sch-app}
        2\sqrt{F} L \sqrt{F} =  L F - 2\Gamma(\sqrt{F}).
    \end{align}
    Furthermore, the infinitesimal form of~\ref{eq:SGE-2} reads
    \begin{align}\label{eq:dPsi-4-sch-app}
        \gamma(s) L \Gamma_w(P_{t-s} f) - 2 \gamma(s) \Gamma_w(P_{t-s} f, L P_{t-s} f) - 2\rho \gamma(s) \Gamma_w(P_{t-s} f) \ge 2 \gamma(s)\Gamma\tp{\sqrt{\Gamma_w(P_{t-s} f)}}.
    \end{align}
    Combining~\eqref{eq:dPsi-1-sch-app},~\eqref{eq:dPsi-2-sch-app},~\eqref{eq:dPsi-3-sch-app},~\eqref{eq:dPsi-4-sch-app} and $\gamma'(s) + 2\rho \gamma(s) = \frac{4}{b}$, it holds
    \begin{align}\label{eq:dPsi-5-sch-app}
    \frac{\dif }{\dif s} \Psi(s) \ge P_s \frac{R}{2\sqrt{F}},
    \end{align}
    where
    \begin{align}\label{eq:dPsi-6-sch-app}
    R = L \+I^2 (P_{t-s} f) - 2 \Gamma(\sqrt{F}) - 2 \+I(P_{t-s} f) \+I'(P_{t-s} f) P_{t-s} L f+ \frac{4}{b}\Gamma_w(P_{t-s} f) + 2\gamma(s) \Gamma(\sqrt{\Gamma_w(P_{t-s}f)})
    \end{align}
    Note that by Minkowski's inequality, it holds
    \begin{align*}
        & 2\Gamma(\sqrt{F})(x) \\
        ={}& \sum_{i\in\mathcal A}\sum_{u\in U_i} q_u(x) \tp{\sqrt{\+I^2(P_{t-s} f)(x) + \gamma(s) \Gamma_w(P_{t-s} f)(x)} - \sqrt{\+I^2(P_{t-s} f)(x^u) + \gamma(s) \Gamma_w(P_{t-s} f)(x^u)}}^2\\
        \le{}& \sum_{i\in\mathcal A}\sum_{u\in U_i} q_u(x) \tp{\tp{\+I(P_{t-s} f)(x) - \+I(P_{t-s} f)(x^u)}^2 + \gamma(s) \tp{\sqrt{\Gamma_w(P_{t-s} f)(x)} - \sqrt{\Gamma_w(P_{t-s} f)(x^u)}}^2}\\
        ={}& 2\Gamma(\+I(P_{t-s}f))(x) + 2\gamma(s) \Gamma(\sqrt{\Gamma_w(P_{t-s} f)})(x).
    \end{align*}
    Combining~\eqref{eq:dPsi-5-sch-app} and~\eqref{eq:dPsi-6-sch-app}, it only remains to show
    \begin{align}\label{eq:dPsi-7-sch-app}
        L \+I^2(P_{t-s} f) - 2 \+I(P_{t-s} f) \+I'(P_{t-s} f) P_{t-s} L f+ \frac{4}{b}\Gamma_w(P_{t-s} f) - 2 \Gamma(\+I(P_{t-s} f)) \ge 0.
    \end{align}
    By definition of $L$ and $\Gamma_w$, the left-hand side of the above inequality satisfies
    \begin{align*}
    \text{LHS of~\eqref{eq:dPsi-7-sch-app}}(x) &= \sum_{i\in\mathcal A} \sum_{u \in U_i} q_u(x) \tp{\+I^2(P_{t-s} f)(x^u) - \+I^2(P_{t-s} f)(x)}\\
    &- 2\sum_{i\in\mathcal A} \sum_{u \in U_i} q_u(x) \+I(P_{t-s} f)(x) \+I'(P_{t-s} f)(x) ((P_{t-s} f)(x^u) - (P_{t-s}f)(x))\\
    &+ \frac{2}{b}\sum_{i\in\mathcal A} \sum_{u \in U_i} q_u(x) w_u(x) \tp{P_{t-s} f(x^u) - P_{t-s} f(x)}^2\\
    &-\sum_{i\in\mathcal A} \sum_{u \in U_i} q_u(x) (\+I(P_{t-s} f)(x^u) - \+I(P_{t-s} f)(x))^2,
    \end{align*}
    which is non-negative by $w_u(x) \ge b$ and
    \begin{align*}
        \+I(x)\tp{\+I(y) - \+I(x) - \+I'(x)(y-x)} + (y-x)^2 \ge 0,
    \end{align*}
    for all $x,y \in (0,1)$.
\end{proof}

\begin{proof}[Proof of~\Cref{lem:improved_talagrand_sch}]
    The assertion is trivial for constant functions, so we assume throughout that \(f\) is nonconstant.
    Fix a function $f: X \to \{0,1\}$. Without loss of generality, we may assume $v := \E[\mu]{f} \le \frac{1}{2}$.
    By~\Cref{lem:local_bobkov-sch}, it holds
    \begin{align}\label{eq:Bobkov-01-sch-app}
        \E[\mu]{\+I(P_t f)} \le \sqrt{\gamma(t)} A_w(f).
    \end{align}
    The Gaussian isoperimetric profile $\+I(P_t f) \gtrsim (P_t f)(1-P_t f)$. Therefore, the expectation of $\+I(P_t f)$ can be lower bounded by
    \begin{align}\label{eq:isoperimetric-bound-sch-app}
        \E[\mu]{\+I(P_t f)} \gtrsim \tp{v - \norm{P_t f}_{2,\mu}^2}.
    \end{align}
    By hypercontractivity, $\norm{P_t f}_{2,\mu}^2 \le \tp{v}^{1+\tanh(\rho_{\LSI} t)}$. Combining~\eqref{eq:Bobkov-01-sch-app} and~\eqref{eq:isoperimetric-bound-sch-app}, the surface area $A_w(f)$ then satisfies
    \begin{align*}
        A_w(f) \gtrsim \frac{v(1-v^{\tanh(\rho_{\LSI} t)})}{\sqrt{\gamma(t)}}.
    \end{align*}
    Take $t = \frac{1}{100(\kappa+\rho_{\LSI} \log 1/v)}$. By estimates $1-v^{\tanh(\rho_{\LSI} t)} \gtrsim \rho_{\LSI} t \log 1/v$ and $\gamma(t) \lesssim \frac{t}{b}$, it holds
    \begin{align*}
        A_w(f) \gtrsim \frac{\sqrt{b} \rho_{\LSI}}{\sqrt{\kappa + \rho_{\LSI} \log 1/v}} v \log 1/v \gtrsim \frac{\sqrt{b}\rho_{\LSI}}{\sqrt{\kappa + \rho_{\LSI}}} v \sqrt{\log 1/v} \gtrsim \frac{\sqrt{b}\rho_{\LSI}}{\sqrt{\kappa + \rho_{\LSI}}} \Var[\mu]{f} \sqrt{\log \frac{\e}{\Var[\mu]{f}}},
    \end{align*}
    where the inequalities follow from $v \le \frac{1}{2}$.
\end{proof}

\begin{proof}[Proof of~\Cref{lem:l1_l2_talagrand_sch}]
    The assertion is trivial for constant functions, so we assume throughout that \(f\) is nonconstant.
As in the proof of~\Cref{lem:l1_l2_talagrand}, we set $T = \frac{1}{\rho_{\LSI} + \kappa_0}$. The variance $\Var[\mu]{f}$ can be bounded by
\begin{align*}
(1-\e^{-2\rho_{\LSI} T})\Var[\mu]{f} \le \frac{2}{b} \int_0^{T} \E[\mu]{\norm{\abs{\nabla_w} P_t f}_2^2} \dif t \le \frac{2}{b}\int_0^{T} \e^{-2 \rho_{\CSGE} t}\E[\mu]{ \norm{P_t \abs{\nabla_w} f}_2^2} \dif t.
\end{align*}
Let $g_i = \abs{\delta_{w,i}} f$. It holds
\begin{align}\label{eq:var-bound-2-sch-app}
(1-\e^{-2\rho_{\LSI} T}) \Var[\mu]{f} \le \frac{2}{b} \sum_{i\in\mathcal A} \int_0^{T} \e^{- 2 \rho_{\CSGE} t} \norm{P_t g_i}_{2,\mu}^2 \dif t.
\end{align}
Following exactly the same argument as in the proof of~\Cref{lem:l1_l2_talagrand}, the integral can be bounded by
\begin{align*}
\int_0^{T} \e^{-2 \rho_{\CSGE} t} \norm{P_t g_i}_{2,\mu}^2 \dif t &\le \norm{g_i}_{2,\mu}^2 \int_0^{T}  \e^{-2 \rho_{\CSGE} t} \tp{\frac{\norm{g_i}_{1,\mu}^2} {\norm{g_i}_{2,\mu}^2}}^{\theta_t} \dif t\\
 &\lesssim \frac{1}{\rho_{\LSI}} \cdot \frac{\norm{g_i}_{2,\mu}^2}{1+\log \frac{\norm{g_i}_{2,\mu}}{\norm{g_i}_{1,\mu}}}.
\end{align*}
The proof of~\eqref{eq:l1_l2_talagrand_2_sch} then follows from~\eqref{eq:var-bound-2-sch-app}, $\norm{g_i}_{1,\mu} = \E[\mu]{\abs{\delta_{w,i}} f}$ and $\norm{g_i}_{2,\mu} = \sqrt{I_{w,i}(f)}$.

To prove~\eqref{eq:l1_l2_talagrand_3_sch}, we observe
    \begin{align*}
    \sum_{i\in\mathcal A} \norm{P_t g_i}_{2,\mu}^2 \le \sum_{i\in\mathcal A} \norm{g_i}_{1,\mu}^{2\theta_t} \norm{g_i}_{2,\mu}^{2(1-\theta_t)} \le \tp{\sum_{i\in\mathcal A} \norm{g_i}_{1,\mu}^2}^{\theta_t} \tp{\sum_{i\in\mathcal A} \norm{g_i}_{2,\mu}^2}^{1-\theta_t},
    \end{align*}
where $\theta_t = \tanh\tp{\rho_{\LSI} t}$ and the last inequality follows from H\"older's inequality. By $J_w(f) = \sum_{i\in\mathcal A} \norm{g_i}_{1,\mu}^2$ and $I_w(f) = \sum_{i\in\mathcal A} \norm{g_i}_{2,\mu}^2$, we have
    \begin{align*}
    \Var[\mu]{f} \lesssim \frac{\rho_{\LSI} + \kappa_0}{b \rho_{\LSI}^2}\frac{I_w(f)}{1+\log \frac{I_w(f)}{J_w(f)}} \lesssim  \frac{\rho_{\LSI} + \kappa_0}{b^2 \rho_{\LSI}^2}\frac{I(f)}{1+\log \frac{I_w(f)}{J_w(f)}} .
    \end{align*}
This completes the proof of~\eqref{eq:l1_l2_talagrand_3_sch}.
\end{proof}

\begin{proof}[Proof of~\Cref{cor:KKL_sch}]
    The assertion is trivial for constant functions, so we assume throughout that \(f\) is nonconstant.
    Since $f$ is a Boolean function, by marginal boundedness, we have
    $$
        \Gamma_{w,i}(f)(x) > 0 \implies \Gamma_{w,i}(f)(x) \geq 2b^2\quad \text{ and } \quad \tp{\E[\mu]{\abs{\delta_{w,i}}f}}^2 \leq \frac{1}{2b^2} I_{w,i}(f)^2.
    $$
    Since $\+M_w(f) \geq I_{w,i}(f)$ for any $i\in\mathcal A$, we conclude that
    \begin{align}\label{eq:IM-to-J-sch-app}
         I_w(f) \+M_w(f) \geq \sum_{i\in\mathcal A} I_{w,i}(f)^2 \geq 2b^2 J_w(f).
    \end{align}
    By \eqref{eq:l1_l2_talagrand_3_sch} in \Cref{lem:l1_l2_talagrand_sch}, the first KKL inequality immediately follows from \eqref{eq:IM-to-J-sch-app}.
    The second follows by the same case split as in the proof of \Cref{cor:KKL}, with $d$ in place of $n$ and using $I_w(f)\leq d\+M_w(f)$.
\end{proof}

\begin{proof}[Proof of~\Cref{thm:eldan_gross_sch}]
The assertion is trivial for constant functions, so we assume throughout that \(f\) is nonconstant.
Fix a Boolean function $f:X \to \{\pm 1\}$ and let $h=(1+f)/2$. When $\Var[\mu]{f} \le \sqrt{J_w(f)}$,
the improved Talagrand inequality (\Cref{lem:improved_talagrand_sch}), applied to $h$ together with $A_w(h)=A_w(f)/2$ and $\Var[\mu]{h}=\Var[\mu]{f}/4$, implies the Eldan--Gross inequality:
\begin{align*}
A_w(f) & \gtrsim_{b,\rho_{\CSGE}} \sqrt{\rho_{\LSI}} \min\set{1,\rho_{\LSI}} \cdot \Var[\mu]{f} \sqrt{\log \frac{\e}{\Var[\mu]{f}}}\\
&\gtrsim_{b,\rho_{\CSGE}} \sqrt{\rho_{\LSI}} \min\set{1,\rho_{\LSI}} \cdot \Var[\mu]{f} \sqrt{\log \tp{1+\frac{\e}{\Var[\mu]{f}^2}}}\\
&\gtrsim_{b,\rho_{\CSGE}} \sqrt{\rho_{\LSI}} \min\set{1,\rho_{\LSI}}  \Var[\mu]{f} \sqrt{\log \tp{1+\frac{\e}{J_w(f)}}}.
\end{align*}
Therefore, we may assume $\Var[\mu]{f} > \sqrt{J_w(f)}$ in the following.

By Poincar\'e inequality,~\Cref{lem:l1_l2_talagrand_sch} and~\ref{eq:infty-CSGE-2}, it holds for any $s \ge 0$ that
\begin{align*}
I(P_s f) &\gtrsim_{b,\rho_{\CSGE}} \Var[\mu]{P_s f}\left[\rho_{\LSI}+\rho_{\LSI}\min\set{1,\rho_{\LSI}}\tp{1+\log\frac{I_w(P_s f)}{J_w(P_s f)}}\right]\\
&\gtrsim_{b,\rho_{\CSGE}} \Var[\mu]{P_s f}\left[\rho_{\LSI}+\frac{\rho_{\LSI}\min\set{1,\rho_{\LSI}}}{1+\log_+\tp{1/\rho_{\LSI}}}
\tp{1+\tp{\log\frac{\Var[\mu]{P_s f}}{J_w(f)}-2\kappa s}_+}\right],
\end{align*}
where $(x)_+ = \max \set{x,0}$ and $\kappa = (-\rho_{\CSGE})_+$.
For simplicity, let $C = \frac{\rho_{\LSI} \min \set{1,\rho_{\LSI}} }{1+\log_+\tp{1/\rho_{\LSI}}}$, and further
let $T = \min\{s \ge 0 \mid \Var[\mu]{P_s f} \le \frac{\Var[\mu]{f}}{2}\}$.

We first consider the case where $\kappa T \le 1$. It holds
\begin{align*}
\frac{\Var[\mu]{f}}{2} = \Var[\mu]{f} - \Var[\mu]{P_T f} &= 2\int_0^{T} I(P_s f) \dif s\\
 &\gtrsim_{b,\rho_{\CSGE}} C \cdot \int_0^{T} \Var[\mu]{P_s f} \tp{1 + \log_+ \frac{\Var[\mu]{P_s f}}{J_w(f)}} \dif s\\
&\gtrsim_{b,\rho_{\CSGE}} C\cdot \Var[\mu]{f} \tp{1 + \log_+ \frac{\Var[\mu]{f}}{J_w(f)}} T.
\end{align*}
This implies $CT \tp{1 + \log_+ \frac{\Var[\mu]{f}}{J_w(f)}} \lesssim_{b,\rho_{\CSGE}} 1$.
By~\Cref{lem:local_bobkov-sch}, it holds for any $s \ge 0$ that
\begin{align*}
P_s h - (P_s h)^2 \le \+I(P_s h) \lesssim \sqrt{\gamma(s)} P_s \sqrt{\Gamma_w(h)}
\end{align*}
By taking expectation over $\mu$ on both sides, the following holds when $\kappa s \le 1$:
\begin{align}\label{eq:bobkov-exp-sch-app}
\Var[\mu]{f} - \Var[\mu]{P_s f} \lesssim \sqrt{\gamma(s)} A_w(f) \lesssim_b \sqrt{s} A_w(f).
\end{align}
By taking $s = T$, it holds
    \begin{align*}
A_w(f) \gtrsim_b \frac{\Var[\mu]{f}}{\sqrt{T}} &\gtrsim_{b,\rho_{\CSGE}} \sqrt{C} \Var[\mu]{f} \sqrt{1+\log_+\frac{\Var[\mu]{f}}{J_w(f)}}\\
&\gtrsim_{b,\rho_{\CSGE}} \frac{\sqrt{\rho_{\LSI}} \min \set{1,\rho_{\LSI}}}{\sqrt{1+\log_+\tp{1/\rho_{\LSI}}}} \cdot \Var[\mu]{f} \sqrt{\log \tp{1+\frac{\e}{J_w(f)}}},
\end{align*}
where the last inequality follows from the assumption $\Var[\mu]{f} \gtrsim \sqrt{J_w(f)}$ and the definition of $C$.

We now consider the remaining case $\kappa T > 1$. Let $T^\star = \frac{1}{\kappa} < T$.
Since $T^\star<T$ and $\kappa T^\star=1$, a similar argument gives
\begin{align}\label{eq:EG-ineq-1-sch-app}
    \Var[\mu]{f} - \Var[\mu]{P_{T^\star} f} \gtrsim_{b,\rho_{\CSGE}}\Var[\mu]{f}\tp{\rho_{\LSI}+C\tp{1+\log_+\frac{\Var[\mu]{f}}{J_w(f)}}}T^\star.
\end{align}
By~\eqref{eq:bobkov-exp-sch-app} and $\gamma(T^\star) \lesssim_{b,\rho_{\CSGE}} 1$, we have
\begin{align}\label{eq:EG-ineq-2-sch-app}
    \Var[\mu]{f} - \Var[\mu]{P_{T^\star} f} \lesssim_{b,\rho_{\CSGE}} A_w(f).
\end{align}
Combining~\eqref{eq:EG-ineq-1-sch-app},~\eqref{eq:EG-ineq-2-sch-app} and our assumption $\Var[\mu]{f} \ge \sqrt{J_w(f)}$, it holds
\begin{align*}
    A_w(f) &\gtrsim_{b,\rho_{\CSGE}} \Var[\mu]{f}\tp{\rho_{\LSI}+C\tp{1+\log_+\frac{\Var[\mu]{f}}{J_w(f)}}}\\
    &\gtrsim_{b,\rho_{\CSGE}} \sqrt{\rho_{\LSI}C} \Var[\mu]{f}\sqrt{1+\log_+\frac{\Var[\mu]{f}}{J_w(f)}}\\
    &\gtrsim_{b,\rho_{\CSGE}}\frac{\sqrt{\rho_{\LSI}}\min\set{1,\rho_{\LSI}}}{\sqrt{1+\log_+\tp{1/\rho_{\LSI}}}}\Var[\mu]{f}\sqrt{\log\tp{1+\frac{\e}{J_w(f)}}}.
\end{align*}
This completes the proof.
\end{proof}

\section{Proofs of Gradient Estimates for the One-edge Ising Model}
\label{app:one-edge-Ising}
\begin{proof}[Proof of \Cref{tab:one-edge-Ising}]
Let $0<\beta<1$ and $\beta_c:=\frac12\log2$.
The canonical transition rates are
\[
q_i^\star(x)=\frac12e^{-\beta x_1x_2}.
\]
Thus, transition rates equal $\frac12e^{-\beta}$ at aligned states
and $\frac12e^\beta$ at anti-aligned states.

We omit the subscript $w$ for simplicity and write $\Gamma$ and
$\Gamma^\star$ for the choices $w=1$ and $w^\star=2q^\star$,
respectively. Set
\[
\widetilde\Gamma_2^\star(f)
:=
\sqrt{\Gamma^\star(f)}\,L\sqrt{\Gamma^\star(f)}
-\Gamma^\star(f,Lf)
=
\Gamma_2^\star(f)
-\Gamma\left(\sqrt{\Gamma^\star(f)}\right),
\]
with the analogous definition for $\widetilde\Gamma_2$.

Fix a state $x$, and denote the transition rates at $x$ and its
neighbors by $\eta_0$ and $\eta_1$, respectively. Thus,
\[
(\eta_0,\eta_1)=
\begin{cases}
\left(\frac12e^{-\beta},\frac12e^\beta\right),
& x_1x_2=1,\\[1mm]
\left(\frac12e^\beta,\frac12e^{-\beta}\right),
& x_1x_2=-1.
\end{cases}
\]
By subtracting a constant from $f$, write
\[
f(x)=0,\qquad f(x^1)=u,\qquad f(x^2)=v,\qquad f(x^{12})=z.
\]
Set
\[
R^2:=u^2+v^2,\qquad S:=u+v,\qquad D:=u-v,
\]
so that $S^2+D^2=2R^2$.
Without loss of generality, assume $R>0$, since the curvature
inequalities hold automatically when $R=0$ and impose no restriction
on $\rho$.

\paragraph{BE and WGE.}
We have
\[
\Gamma(f)(x)=\frac{\eta_0}{2}R^2,
\qquad
\Gamma^\star(f)(x)=\eta_0^2R^2.
\]
Direct expansion gives
\[
\Gamma_2(f)(x)
=
\frac{\eta_0}{2}
\left[
\eta_1(z-S)^2+2\eta_1R^2+2(\eta_0-\eta_1)uv
\right],
\]
and
\[
\begin{aligned}
\Gamma_2^\star(f)(x)
={}&
\eta_0\tp{
\left(\eta_1 z-\frac{\eta_0+\eta_1}{2}S\right)^2
+
\frac{(3\eta_0+\eta_1)(\eta_0-\eta_1)}4S^2
+(\eta_1^2-\eta_0^2+2\eta_0\eta_1)R^2
}.
\end{aligned}
\]
Minimizing these quadratic forms yields the pointwise constants
\[
\kappa_{\mathrm{BE}}(\eta_0,\eta_1)
=
\min_{u,v,z: R>0}\frac{\Gamma_2(f)(x)}{\Gamma(f)(x)}
=
\min\{\eta_0+\eta_1,3\eta_1-\eta_0\},
\]
and
\[
\kappa_{\mathrm{BE}}^\star(\eta_0,\eta_1)
=
\min_{u,v,z: R>0}\frac{\Gamma_2^\star(f)(x)}{\Gamma^\star(f)(x)}
=
\min\left\{
\frac{(\eta_0+\eta_1)^2}{2\eta_0},
\frac{\eta_1^2-\eta_0^2+2\eta_0\eta_1}{\eta_0}
\right\}.
\]
Substituting the two possible pairs $(\eta_0,\eta_1)$ shows that both
global minima are attained at anti-aligned states. Therefore,
\[
\rho_{\mathrm{BE}}=\rho_{\mathrm{WGE}}
=
\frac12(3e^{-\beta}-e^\beta),
\]
and
\[
\rho_{\mathrm{BE}}^\star=\rho_{\mathrm{WGE}}^\star
=
e^{-\beta}+\frac12e^{-3\beta}-\frac12e^\beta.
\]

\paragraph{Unweighted rBE and SGE.}
Set
\[
N_1:=\sqrt{u^2+(z-u)^2},
\qquad
N_2:=\sqrt{v^2+(z-v)^2}.
\]
A direct calculation gives
\[
\mathcal R(\eta_0,\eta_1)
:=
\frac{\widetilde\Gamma_2(f)(x)}{\Gamma(f)(x)}
=
2\eta_1+
\frac{
\sqrt{\eta_0\eta_1}\,R(N_1+N_2)-\eta_1zS-\eta_0D^2
}{R^2}.
\]
Cauchy--Schwarz, the triangle inequalities and $2R^2 - D^2 = S^2 \geq 0$ imply
\[
R(N_1+N_2)\ge u^2 + v(z-u)+v^2+u(z-v) = D^2 + zS,  
\]
\[
R(N_1+N_2)\ge R\norm{(u,z-u) - (v,z-v)}_2 =
\sqrt2\,R|D|\ge D^2.
\]
We distinguish three cases.
\begin{itemize}
    \item [1.] $\eta_0\ge\eta_1$. These inequalities give
        \begin{align}
        \mathcal R(\eta_0,\eta_1)
        &=
        2\eta_1+
        \frac{
        \eta_1[R(N_1+N_2)-zS]
        +(\sqrt{\eta_0\eta_1}-\eta_1)R(N_1+N_2)
        -\eta_0D^2 
        }{R^2}\nonumber\\
        &\ge
        2\eta_1+
        (\sqrt{\eta_0\eta_1}-\eta_0)\frac{D^2}{R^2}\nonumber\\
        &\ge
        2(\eta_1+\sqrt{\eta_0\eta_1}-\eta_0). \label{eq:R_lower_case_1}
        \end{align}
        Equality holds when $u=-v\ne0$ and $z=0$.
        Therefore, $\inf_{u,v,z:R>0}\mathcal R(\eta_0,\eta_1) = 1 + e^{-\beta} - e^\beta $.
    \item [2.] $\eta_0<\eta_1\le2\eta_0$. It suffices to prove
        $\mathcal R(\eta_0,\eta_1)\ge2\eta_0 \geq 1 + e^{-\beta} - e^\beta$.
        For fixed $\eta_0,u,v,z$, the expression $\mathcal R(\eta_0,\eta_1)$
        is concave in $\eta_1$.
        At $\eta_1=\eta_0$, \eqref{eq:R_lower_case_1} gives
        $\mathcal R(\eta_0,\eta_0)\ge2\eta_0$.
        At $\eta_1=2\eta_0$, we have
        \[
        \mathcal R(\eta_0,2\eta_0)-2\eta_0
        =
        \frac{\eta_0}{R^2}
        \left[
        \sqrt2\,R(N_1+N_2)-S(2z-S)
        \right]
        \ge0,
        \]
        since $N_1+N_2\ge|2z-S|$ and $|S|\le\sqrt2\,R$.
        Concavity therefore gives
        $\mathcal R(\eta_0,\eta_1)\ge2\eta_0$ throughout
        $\eta_1\in[\eta_0,2\eta_0]$.
        Note that for this interpolation argument, we regard $\eta_0$ and $\eta_1$ as independent positive parameters.

    \item [3.] $\eta_1 > 2\eta_0$. 
        Then $\eta_1 = \frac 12 e^{\beta} > e^{-\beta} = 2\eta_0$, that is, $\beta > \beta_c$.
        Take an aligned state and choose
        $u=v=\varepsilon>0$, $z=1$. Then
        \[
        \mathcal R\left(\frac12e^{-\beta},\frac12e^\beta\right)
        =
        \frac{
        \frac1{\sqrt2}\sqrt{\varepsilon^2+(1-\varepsilon)^2}
        -\frac12e^\beta
        }{\varepsilon}
        +e^\beta.
        \]
        Since $e^{\beta} > \sqrt 2$, the numerator tends to $1/\sqrt 2 - e^\beta/2 < 0$.
        Thus, $\mathcal R\left(\frac12e^{-\beta},\frac12e^\beta\right) \to -\infty$ as $\eps \to 0$.
\end{itemize}
Therefore,
\[
\rho_{\mathrm{rBE}}=\rho_{\mathrm{SGE}}
=
\begin{cases}
1+e^{-\beta}-e^\beta, & 0<\beta\le\beta_c,\\
-\infty, & \beta_c<\beta<1.
\end{cases}
\]

\paragraph{Unweighted CSGE.}
If $0$-CSGE holds with a finite constant $\rho$, differentiating its
squared form at $t=0$ gives the necessary condition
\[
\begin{gathered}
\left.
\frac{ d}{ dt}
\sum_{i=1}^{2}\Gamma_i(P_t f)(x)
\right|_{t=0^+}
\le
\left.
\frac{ d}{ dt}
\left[
e^{-2\rho t}
\sum_{i=1}^{2}
\left(P_t\sqrt{\Gamma_i(f)}(x)\right)^2
\right]
\right|_{t=0^+}
\\[0.1em]
\Downarrow
\\[-0.5em]
\sum_{i=1}^{2}
\sqrt{\Gamma_i(f)(x)}\,L\sqrt{\Gamma_i(f)}(x)
-\Gamma(f,Lf)(x)
\ge
\rho\,\Gamma(f)(x).
\end{gathered}
\]

At an aligned state $x$, take
\[
u = v = \eps, z = 1,
\]
where $\eps \in (0,1)$.
Here $\eta_0=\frac12e^{-\beta}$ and $\eta_1=\frac12e^\beta$.
Direct calculation gives
\[
\frac{
\sum_{i=1}^2\sqrt{\Gamma_i(f)} L\sqrt{\Gamma_i(f)}-\Gamma(f,Lf)
}{
\Gamma(f)
}(x)
=
\frac{\sqrt{\eta_0\eta_1}-\eta_1}{\varepsilon}+2\eta_1
=
\frac{1-e^\beta}{2\varepsilon}+e^\beta.
\]
For every $\beta>0$, this tends to $-\infty$ as
$\varepsilon\downarrow0$, contradicting the necessary condition
for any finite $\rho$. 
Thus \(0\)-CSGE admits no finite uniform constant. Consequently, no finite constant can make \(\tau\)-CSGE hold simultaneously for all \(\tau\ge0\). 
\paragraph{Canonical weighted rBE, SGE, and CSGE.}
By \Cref{thm:CSGE_boolean_hypercube_w=q},
the optimal coordinate-wise strong gradient constant satisfies the following lower bound:
\[
\rho_{\mathrm{CSGE}}^\star\ge\lambda_{\min}\tp{
\frac12
\begin{pmatrix}
3e^{-\beta}-e^\beta & -(e^\beta-e^{-\beta})\\
-(e^\beta-e^{-\beta}) & 3e^{-\beta}-e^\beta
\end{pmatrix}} = 2e^{-\beta}-e^\beta.
\]
For a matching upper bound, take an anti-aligned state and choose
\[
u = 1, v = -1, z = 0.
\]
Then
\[
\Gamma^\star(f)(x)=\frac12e^{2\beta},
\qquad
\widetilde\Gamma_2^\star(f)(x)
=
e^\beta-\frac12e^{3\beta}.
\]
Consequently,
\[
\frac{
\widetilde\Gamma_2^\star(f)(x)
}{
\Gamma^\star(f)(x)
}
=
2e^{-\beta}-e^\beta,
\]
and hence $\rho_{\mathrm{rBE}}^\star\le2e^{-\beta}-e^\beta$.

Since CSGE implies SGE and SGE is equivalent to rBE,
\[
2e^{-\beta}-e^\beta
\le\rho_{\mathrm{CSGE}}^\star
\le\rho_{\mathrm{SGE}}^\star
=\rho_{\mathrm{rBE}}^\star
\le2e^{-\beta}-e^\beta.
\]
Therefore,
\[
\rho_{\mathrm{CSGE}}^\star
=\rho_{\mathrm{SGE}}^\star
=\rho_{\mathrm{rBE}}^\star
=
2e^{-\beta}-e^\beta.
\]
This proves all the entries in the table.
\end{proof}

\end{document}